\documentclass[reqno, 11pt]{amsart}
\usepackage{amsmath,amssymb,amsfonts,amsthm,graphicx,color}
\usepackage{mathrsfs}
\usepackage{hyperref}
\usepackage{mathtools}
\usepackage{enumitem}
\usepackage{graphicx,type1cm,eso-pic,color, float}
\usepackage[margin=0.8in]{geometry}
\allowdisplaybreaks
\usepackage{orcidlink}
\usepackage[dvipsnames]{xcolor}   
\newcommand{\annubc}[2]{%
  \overbrace{#1}^{\text{\scriptsize #2}}%
}

\newcommand{\annotopc}[2]{%
  \underbrace{#1}_{\text{\scriptsize #2}}%
}

\newcommand{\R}{\mathbb{R}}

\def\be{\begin{equation}}
\def\ee{\end{equation}}
\def\bse{\begin{subequations}}
\def\ese{\end{subequations}}
\def\bge{\begin{eqnarray}}
\def\bgee{\begin{eqnarray*}}
\def\ege{\end{eqnarray}}
\def\egee{\end{eqnarray*}}

\newtheorem{theorem}{Theorem}[section]
\newtheorem{proposition}[theorem]{Proposition}
\newtheorem{lemma}[theorem]{Lemma}
\newtheorem{corollary}[theorem]{Corollary}

\newtheorem{definition}[theorem]{Definition}
\newtheorem{assumption}[theorem]{Assumption}
\theoremstyle{remark}
\newtheorem{remark}[theorem]{Remark}

\DeclareMathOperator{\Var}{Var}

\title[Asymptotic Behaviour of a Nonlocal Stochastic Fractional Equation]{Long-time behaviour of a nonlocal stochastic fractional reaction--diffusion equation arising in tumour dynamics}
\date{\today}
\author[N. I. Kavallaris]{Nikos I. Kavallaris  \orcidlink{0000-0002-9743-8636}}
\address{
 Department of Mathematics and Computer Science, Karlstad University, Karlstad, Sweden}
\email{nikos.kavallaris@kau.se} 

\thanks{Corresponding author: N.I. Kavallaris (\texttt{nikos.kavallaris@kau.se})}
\author[S. Sankar]{Subramani Sankar}
\address{Department of Mathematics, School of Basic Sciences,
Indian Institute of Technology Bhubaneswar, Khurda 752 050, India}
\email{subusankar27@gmail.com}

\author[M. T. Mohan]{Manil T. Mohan} 
\address{Department of Mathematics, Indian
	Institute of Technology Roorkee, Roorkee 247 667, India}
\email{maniltmohan@ma.iitr.ac.in}

\author[C. V. Nikolopoulos]{Christos V. Nikolopoulos \orcidlink{0000-0002-1235-5058}}
\address{Department of Mathematics, University of Aegean, Karlovassi, Samos, Greece}
\email{cnikolo@aegean.gr}

\author[S. Karthikeyan]{Shanmugasundaram Karthikeyan}
\address{Department of Mathematics, Periyar University, Salem, India}
\email{karthi@periyaruniversity.ac.in}

\date{\today}
\allowdisplaybreaks

\usepackage{xcolor}
\allowdisplaybreaks

\usepackage[pagewise]{lineno}

\usepackage{graphicx,eurosym}
\usepackage{hyperref}
\usepackage{mathtools}

\colorlet{darkblue}{blue!50!black}

\hypersetup{
	colorlinks,%
	citecolor=blue,%
	filecolor=red,%
	linkcolor=darkblue,%
	urlcolor=blue,%
	pdfnewwindow=true,%
	pdfstartview={FitH}
}

\usepackage{graphicx,amscd,mathrsfs,wrapfig,mathrsfs,lipsum}
\usepackage{eufrak}
\usepackage{tikz}
\usepackage{multicol}
\usepackage{caption}
\usetikzlibrary{arrows}

\colorlet{darkblue}{red!100!black}

\let\originalleft\left
\let\originalright\right
\renewcommand{\left}{\mathopen{}\mathclose\bgroup\originalleft}
\renewcommand{\right}{\aftergroup\egroup\originalright}

\begin{document}
	\maketitle \setcounter{page}{1}
	\numberwithin{equation}{section}
	
	\newtheorem{app}{Appendix:}
	\newtheorem{ack}{Acknowledgement:}

	\begin{abstract}
We introduce a stochastic nonlocal
reaction--diffusion model arising in tumour dynamics. Spatial dispersal is described by the fractional
Laplacian, accounting for anomalous diffusion and long--range relocation events.
The system is perturbed by multiplicative fractional Brownian motion (fBm) with  Hurst parameter $H>1/2$,
which we interpret as temporally correlated fluctuations in the tumour
microenvironment and host response.

We first establish well--posedness and identify parameter regimes leading to
global--in--time solutions or finite--time blow--up under general multiplicative
fractional noise. We then focus on linear multiplicative noise
and, via a Doss--Sussmann transformation, derive sharper results: explicit lower
and upper bounds for the blow--up time together with quantitative estimates of
the blow--up probability, clarifying how noise intensity can accelerate
progression or, on favourable paths, enhance suppression consistent with
extinction (loss of viability). Finally, one--dimensional simulations illustrate
the interplay between anomalous diffusion, fractional noise, and the nonlocal
reaction mechanism in shaping the long--time dynamics.

\end{abstract}

\tableofcontents
	\section{Introduction}\label{sec1}
The main aim of this work is to analyse the long--time dynamics of the following
nonlocal stochastic fractional reaction--diffusion equation driven by fractional
noise:
\begin{equation}\label{b1}
\left\{
\begin{aligned}
du(t,x)-\Delta_{\alpha}u(t,x)\,dt
&=\Bigg[ \delta \int_{D}u^{q}(t,y)\,dy+\gamma u(t,x)- \beta u^{p}(t,x) \Bigg]dt
+\sigma\big(u(t,x)\big)\,dB^{H}(t), \quad x\in D,\ t>0,\\
u(0,x)&=f(x), \quad x\in D,\\
u(t,x)&=0, \quad t>0,\ x\in \mathbb{R}^{d}\setminus D.
\end{aligned}
\right.
\end{equation}
Here $\Delta_{\alpha}:=(-\Delta)^{\alpha/2}$, $0<\alpha\le 2$, denotes the
fractional Laplacian (see Section~\ref{sec3} for the precise definition and the
functional-analytic setting), and $B^{H}=(B^{H}(t))_{t\ge 0}$ is a fractional
Brownian motion with Hurst index $H\in[1/2,1)$ defined on a filtered probability
space $(\Omega,\mathcal{F},(\mathcal{F}_{t})_{t\ge 0},\mathbb{P})$.
The parameters $\beta,\delta,\gamma>0$ and $p,q>1$ are fixed, the initial datum
$f$ is assumed to be nonnegative and not identically zero (e.g.\ $f\in C^{2}(D)$),
and the noise coefficient $\sigma(\cdot)$ is a prescribed function, see again Section \ref{sec3} for more details.

The investigation of \eqref{b1} is motivated by tumour dynamics, where $u(t,x)$ counts for tumour density. Nonlocal
diffusion operators such as $\Delta_{\alpha}$ arise as continuum limits of
random walks with heavy-tailed jump distributions and provide effective
macroscopic descriptions of L\'evy-type (superdiffusive) transport observed in
biological systems; see, for instance,
\cite{chen2019_nonlocal_migration,meerschaert_scheffler2017_ctrw,tsai2012_fractal_tumour, zaburdaev2015_levy_walks}.
Moreover, the multiplicative perturbation $\sigma(u(t,x))\,dB^{H}(t)$ models
temporally correlated environmental fluctuations (long-memory noise), which may
represent persistent variability in growth conditions, immune pressure, or
microenvironmental support; cf.\ \cite{kwossek2025_sde_fbm}.
A more detailed discussion of the link between \eqref{b1} and metastatic tumour
dynamics is provided in Section~\ref{sec2}. From a mathematical viewpoint, our goal is to characterise parameter regimes in
which the solution to \eqref{b1} exists globally in time and, possibly,
decays to the trivial state (interpreted biologically as tumour extinction),
versus regimes in which the solution exhibits finite--time blow--up, corresponding
to a loss of regulatory control and uncontrolled tumour growth.

From a theoretical standpoint, blow--up and nonexistence phenomena play a
fundamental role in the qualitative analysis of nonlinear evolution equations
and are equally valuable from an applied viewpoint.  Classical contributions by
Kaplan~\cite{kaplan} and Fujita~\cite{fuji1966,Fuji1968} initiated systematic
methods for proving finite--time blow--up for broad classes of nonlinear
parabolic problems, and these ideas have since inspired a large body of work in the deterministic setting.  

In the stochastic setting the analytical picture is substantially more subtle and is still developing. In contrast to deterministic parabolic equations-whose evolution is entirely prescribed by the initial datum-solutions to SPDEs are shaped by the interaction between nonlinear reaction mechanisms and random fluctuations. This interplay may either accelerate or inhibit singularity formation, shift critical thresholds separating global existence from blow--up, and produce genuinely pathwise phenomena with no deterministic analogue. These features, together with applications in physics, biology, and finance, have motivated a growing literature on rigorous blow--up criteria, quantitative bounds on explosion times, and estimates of blow--up probabilities for nonlinear SPDEs.

Within this direction, Chow~\cite{chow09,chow11} studied finite--time blow--up in the mean $L^{p}$-sense for semilinear parabolic equations driven by space--time noise, adapting Kaplan’s eigenfunction method~\cite{kaplan}. Variants of this strategy were subsequently developed for nonlocal semilinear models by Kavallaris and collaborators~\cite{K15,KY20}. Related eigenfunction-based arguments have also been employed in the presence of anomalous diffusion, including fractional Laplacian operators; see Wang~\cite{wang}.

For linear multiplicative noise that is purely time dependent, Dozzi and coauthors~\cite{doz2010,doz2013,car2013} analysed the long--time behaviour of semilinear parabolic SPDEs using a Doss--Sussmann type transformation. This device reduces the original SPDE to a pathwise random parabolic PDE, thereby allowing one to adapt deterministic blow--up techniques in the spirit of Fujita~\cite{fuji1966,Fuji1968} and Kaplan~\cite{kaplan} to obtain estimates for the blow--up time. In addition, in~\cite{doz2010,doz2013,car2013} lower and upper bounds for the blow--up probability are derived through estimates of exponential functionals of Brownian motion; see, for instance,~\cite{yor2005,revuz1999,yor2001}. Subsequently, related ideas were extended to coupled systems~\cite{Eug2017, li,smk1} and to nonlocal stochastic models~\cite{KY20,liang}. Concerning semilinear SPDEs driven by fractional Brownian motion, Dozzi and collaborators~\cite{dozfrac2013} established two--sided bounds for the explosion time of a semilinear equation perturbed by a one--dimensional fractional Brownian motion $B^{H}=(B^{H}(t))_{t\ge 0}$ with Hurst index $H>\tfrac12$. They also provided sufficient conditions ensuring either finite--time blow--up or global existence, together with quantitative estimates for the probability of finite--time explosion. More recently, Dung~\cite{dung} developed an alternative approach based on tail estimates for exponential functionals of fractional Brownian motion, leading to explicit upper and lower bounds on the finite--time blow--up probability in terms of the Hurst parameter $H\in(0,1)$. This approach was again extended to case of coupled systems~\cite{smk2} and nonlocal equations~\cite{smk3}.

To place the present work in context, we briefly recall several contributions on stochastic parabolic equations involving fractional diffusion. Chang and Lee \cite{chang2012} studied regularity properties for SPDEs driven by a fractional Laplacian, while Liu and Tudor \cite{liu2017} established existence and uniqueness of mild solutions for stochastic fractional heat equations. Related well--posedness results for space--time fractional SPDEs on bounded domains were obtained by Foondun, Mijena and Nane \cite{foon2016}. These papers provide a solid analytical foundation for stochastic evolution equations with nonlocal generators, but they are primarily concerned with existence, uniqueness, and regularity, rather than with sharp long--time scenarios such as extinction versus finite--time blow--up in the presence of nonlocal reactions and temporally correlated noise.

On the other hand, quantitative blow--up analysis has been developed mainly in settings with classical diffusion and Brownian forcing. In the Brownian case ($H=\tfrac12$) with the standard Laplacian ($\alpha=2$), Liang and Zhao \cite{liang} investigated finite--time blow--up for a nonlocal stochastic reaction--diffusion equation by means of a random transformation reducing the SPDE to a pathwise random PDE, from which bounds on the blow--up time and estimates of the blow--up probability can be derived. For related semilinear models with $\alpha=2$ and simplified reaction structures (e.g.\ particular choices of parameters such as $\delta=1$ and vanishing linear terms), blow--up criteria and probability estimates were obtained in \cite{wang}, and more recently in \cite{smk3} for variants with different nonlocal terms.

The above literature indicates a clear gap: a systematic blow--up/extinction theory for nonlocal reaction--diffusion equations combining fractional spatial dispersal ($0<\alpha<2$) with multiplicative fractional Brownian forcing ($H>\tfrac12$) is still largely missing. The model \eqref{b1} is tailored to this regime and brings together three strongly interacting features: $(i)$ anomalous L\'evy--type diffusion via the fractional Laplacian, whose analysis relies on Dirichlet forms, heat-kernel and spectral estimates, and nonlocal maximum/comparison principles \cite{BogdanGrzywnyRyznar2010,BucurValdinoci2016,CaffarelliSilvestre2007,CaffarelliSilvestre2009,ChenKimSong2010,DiNezzaPalatucciValdinoci2012,Kwasnicki2017}; (ii) a genuinely nonlocal reaction term through the spatial integral, \cite{KS18}, which induces global coupling and modifies classical Kaplan--Fujita blow--up mechanisms and their stochastic extensions \cite{chow09,chow11,fuji1966,Fuji1968,kaplan,K15,KS18,KY20,liang,wang}; and (iii) long--memory multiplicative noise driven by fBm, requiring a pathwise (Young/rough) integration framework and producing persistent excursions that may shift stability thresholds and alter blow--up times and probabilities \cite{MaslowskiNualart2003,Mishura2008,Nualart2006,NualartRascanu2002,TindelTudorViens2003,Zahle1998}. Consequently, \eqref{b1} serves as a natural and challenging testbed at the intersection of modern nonlocal PDE theory and stochastic analysis, and the goal of this paper is to develop a detailed long--time dynamical picture-global existence versus finite--time blow--up, two--sided blow--up time bounds, and quantitative blow--up probability estimates-thereby extending the Brownian/$\alpha=2$ framework of \cite{liang} to the fractional diffusion/fractional noise setting.

The paper is organised as follows. In Section~\ref{sec2} we discuss the
biological motivation for the fractional nonlocal stochastic model
\eqref{b1}. Section~\ref{sec3} collects the analytical preliminaries and the
functional framework. In Section~\ref{sec:noise-general} we study the long--time dynamics
under general multiplicative fractional noise for $\frac{1}{2}<H<1$, identifying parameter regimes leading to
global-in-time solutions or finite--time blow--up (Theorem~\ref{thm:SF-gb1-fBm}),
as well as almost sure exponential decay in a strongly dissipative setting
(Theorem~\ref{cor:sto-exp-fBm-gb1}). Section~\ref{sec5} focuses on linear
multiplicative noise: using a Doss--Sussmann type transformation we reduce
\eqref{b1} to a random PDE and derive two-sided bounds for the blow--up time
(Theorems~\ref{t2} and~\ref{thm5.1}), estimates on the blow--up rate
(Theorem~\ref{thm:two-sided-rate}), and quantitative lower bounds for the
blow--up probability (Theorem~\ref{thm5.2}). In Section~\ref{sec6} we address
the special case $H\in(\tfrac34,1)$, where $B^H$ is equivalent in law to a
standard Brownian motion, leading to simplified probability bounds
(Theorems~\ref{thm6.1} and~\ref{thm6.2}). The paper concludes with a numerical
study of \eqref{b1} in Section~\ref{nss} based on a finite-difference discretisation, illustrating
the interplay between anomalous diffusion, fractional noise, and the nonlocal
reaction term in the resulting long--time dynamics.

\section{Motivation}\label{sec2} 
Apart from its intrinsic mathematical significance, equation~\eqref{b1} is also of interest due to its connection with cancer dynamics, as explained in detail below. In this framework, \(u(t,x)\) denotes the tumour cell density at time \(t > 0\) and position \(x \in D\) within a biological tissue, and model~\eqref{b1} describes its evolution under specific biological conditions
\begin{align*}
\partial_t u
= -\,\annotopc{(-\Delta)^{\alpha/2}u}{L\'evy migration}
   + \annubc{\delta\!\int_{D} u^{q}(y,t)\,dy}{nonlocal signalling}
   + \annubc{\gamma u}{proliferation}
   - \annotopc{\beta u^{p}}{crowding / therapy}
   + \annubc{\sigma(u)\,\mathrm{d}B^H(t)}{fBm noise}.
\end{align*}
The fractional Laplacian $(-\Delta)^{\alpha/2}$ represents long-range L\'evy-type tumour cell migration and invasion along heterogeneous tissue structures, in particular for metastatic cancer cells, consistent with experimental and modelling evidence for anomalous motility and nonlocal cancer cell dispersal \cite{chen2019_nonlocal_migration, H18, tsai2012_fractal_tumour,volpert2025_reaction_diffusion_waves}, see also  Fig.~\ref{fig:cells}.
The nonlocal integral term $\delta\int_D u^q(x)dx$ models domain-wide signalling or systemic feedback such as cytokines, growth factors or immune mediators, in line with nonlocal reaction–diffusion descriptions of biomedical signalling and immunological interactions \cite{bessonov2022_nonlocal_biomed,nonlocal_tumor_growth2025, kazmierczak2025_tissue_growth}.
Local reaction dynamics $(\gamma u - \beta u^p)$ describe proliferation constrained by crowding, competition, or therapy, as in classical and nonlocal tumour reaction–diffusion models \cite{mathematical_oncology_review2025, volpert2025_reaction_diffusion_waves}.
The multiplicative fractional Brownian motion noise $\sigma(u)\,\mathrm{d}B^H$ captures temporally correlated fluctuations in vascularisation, oxygenation, or immune pressure, generalising earlier stochastic cancer models driven by (white) Brownian noise to the long-memory setting \cite{fahimi2020_chaos_stoch_cancer,fritz2025_stoch_cahn_hilliard, kwossek2025_sde_fbm,sweilam2025_frac_stoch_rd_cancer}. Finally,  the Dirichlet exterior condition encodes a hostile
environment outside $D$ (e.g.\ resection boundary or lethal habitat).

This formulation integrates \emph{two cutting-edge ingredients that, to our knowledge, have not been combined in a single continuum tumour model}:
(i) nonlocal spatial and mean-field couplings (capturing systemic signalling and long-range migration), extending the growing body of nonlocal reaction-diffusion models in biomedical applications and cancer invasion \cite{bessonov2022_nonlocal_biomed, chen2019_nonlocal_migration,nonlocal_tumor_growth2025};
and (ii) stochastic forcing with long-range temporal correlations (fractional Brownian noise), which goes beyond the classical white-noise perturbations considered in tumour SPDEs and stochastic ODE models \cite{fahimi2020_chaos_stoch_cancer,fritz2025_stoch_cahn_hilliard,kwossek2025_sde_fbm,sweilam2025_frac_stoch_rd_cancer}.
Their interplay raises fundamentally new questions on tumour dynamics that go beyond classical local PDE models without memory, nonlocality or correlated noise.

Mathematically, the model contributes to the emerging theory of nonlocal SPDEs with memory and fractional noise, sitting at the intersection of spatial fractional diffusion limits \cite{meerschaert_scheffler2017_ctrw}, nonlocal reaction–diffusion equations in biomedicine \cite{bessonov2022_nonlocal_biomed,nonlocal_tumor_growth2025}, and stochastic tumour growth models \cite{fritz2025_stoch_cahn_hilliard,sweilam2025_frac_stoch_rd_cancer}.
Biologically, it clarifies conditions under which tumour growth remains bounded (therapy success, strong competition or inhibitory signalling) versus regimes leading to blow--up (uncontrolled systemic proliferation), complementing recent fractional and stochastic extensions of tumour and tumour–immune models \cite{alinei_poiana2023_frac_oncology,gundogdu2025_time_frac_cancer,mathematical_oncology_review2025, west2022_fractal_tapestry}.
This dichotomy provides a rigorous framework to interpret the onset of aggressive cancer dynamics and the potential stabilising role of spatial nonlocality, memory and correlated fluctuations.

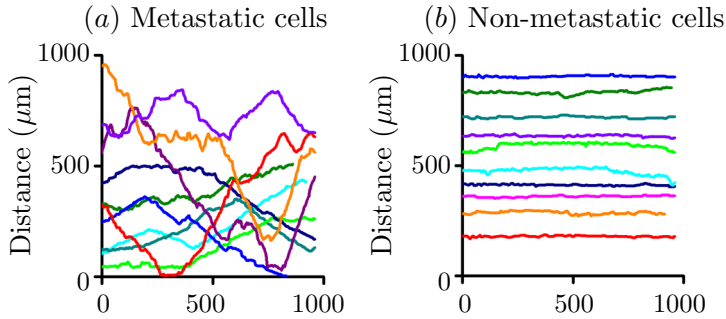
\begin{figure}[h!]
\begin{minipage}[m]{0.63\textwidth}
\begin{tikzpicture}
\draw (-3.76,-1.6) node[below] {\small 0};
\draw (-2.4,-1.6) node[below] {\small 500};
\draw (-0.9,-1.6) node[below] {\small 1000};
\draw (1.05,-1.6) node[below] {\small 0};
\draw (2.5,-1.6) node[below] {\small 500};
\draw (3.9,-1.6) node[below] {\small 1000};

\draw (-3.8,-1.55) node[left] {\small 0};
\draw (-3.8,0) node[left] {\small 500};
\draw (-3.8,1.5) node[left] {\small 1000};
\draw (1.0,-1.55) node[left] {\small 0};
\draw (1.0,0) node[left] {\small 500};
\draw (1.0,1.5) node[left] {\small 1000};

\draw (-4.5,0) node[above,rotate=90] {Distance ($\mu$m)};
\draw (0.3,0) node[above,rotate=90] {Distance ($\mu$m)};

\draw (-2.4,1.6) node[above] {$(a)$ Metastatic cells};
\draw (2.5,1.6) node[above] {$(b)$ Non-metastatic cells};

\clip (-3.9,-1.6) -- (-0.5,-1.6) -- (-0.5,1.5) -- (0.8,1.5) -- (0.8,-1.6) -- (4,-1.6) -- (4,1.5) -- (-3.9,1.5) -- cycle;
\draw (0,0) node {\includegraphics[scale=1,page=4,trim=40mm 220mm 80mm 20mm,clip]{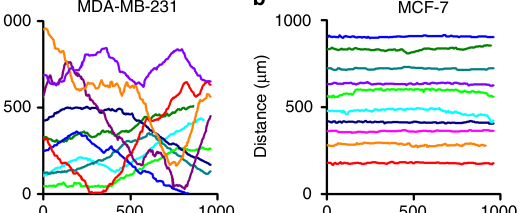}};
\end{tikzpicture}
\end{minipage}
\hfill
\begin{minipage}[m]{0.35\textwidth}
\caption{$\hspace{-3pt}(a)$ Metastatic cell movements: short bursts interspersed with long unidirectional journeys (local and nonlocal diffusion). $(b)$ Non-metastatic cell movements: more random or ``jiggly'' paths. Image adapted from Huda et al.~\cite{H18}.}
\label{fig:cells}
\end{minipage}
\end{figure}

\section{Mathematical Preliminaries}\label{sec3}
\subsection*{Notation and function spaces}	

We work in the Hilbert space
$
\mathcal{H} := L^2(D),
$
equipped with the norm
\[
\|u\|_2:= \|u\|_{L^2(D)}= \left( \int_D |u(x)|^2\,\mathrm{d}x \right)^{1/2},
\]
and inner product
\[
\langle u, \varphi \rangle_\mathcal{H} = \int_D u(x)\,v(x)\,\mathrm{d}x, \qquad u,v \in \mathcal{H}=L^2(D).
\]
Also $\|u\|_p$ stands for the norm of $L^p(D)$ for any $1\leq p\leq \infty.$ 

The canonical duality pairing between $H$ and its dual is denoted by
$\langle \cdot,\cdot \rangle$. For $u \in \mathcal{H}$ and $\varphi \in\mathcal{H}$ we write
\[
u(t,\varphi) := \langle u(t), \varphi \rangle_\mathcal{H}
= \int_D u(t,x)\,\varphi(x)\,\mathrm{d}x.
\]
We denote by
$
V := H^{\alpha/2}_0(D)$
the energy space associated with the fractional Laplacian $\Delta_{\alpha}$ with Dirichlet exterior boundary conditions. The space $V$ is a Hilbert space endowed with the norm
\[
\|u\|_{H^{\alpha/2}(D)}^2
:= \|u\|_{L^2(D)}^2
   + \int_D \int_D 
     \frac{|u(x) - u(y)|^2}{|x-y|^{d+\alpha}} \,\mathrm{d}x\,\mathrm{d}y,
\qquad u \in H^{\alpha/2}(D),
\]
and corresponding inner product $(\cdot,\cdot)_V$.

Throughout this work we will also use the Banach space
\[
\mathcal{B}^{\nu,2}\big([0,t], \mathcal{H}\big),
\]
which consists of all measurable functions $u : [0,t] \to L^2(D)$ for which the norm
\[
\|u\|_{\nu,2}^2
= \left( \operatorname*{ess\,sup}_{s \in [0,t]} \|u(\cdot,s)\|_2 \right)^2
  + \int_0^t \left( \int_0^s \frac{\|u(\cdot,s) - u(\cdot,r)\|_2}{(s-r)^{\nu+1}}\,\mathrm{d}r \right)^2 \mathrm{d}s
\]
is finite, i.e.\ $\|u\|_{\nu,2} < +\infty$, for $0 < \nu < \tfrac{1}{2}$, cf.\ \cite{ Zahle}.

\subsection*{Fractional Laplacian and relevant structures}
\label{subsec:frac-Lap}

Let $D\subset\mathbb{R}^d$, $d\ge1$, be a bounded domain with $C^{2}$ boundary and
let $\alpha\in(0,2)$. Throughout the paper we denote by
$
\Delta_\alpha := (-\Delta)^{\alpha/2}
$
the  fractional Laplacian of order $\alpha$ with exterior Dirichlet
condition. For sufficiently regular functions $u:\mathbb{R}^d\to\mathbb{R}$ we use
the pointwise singular integral representation
\begin{equation}\label{eq:frac-lap-def}
(-\Delta)^{\alpha/2} u(x)
=
C_{d,\alpha}\,\mathrm{P.V.}\!\int_{\mathbb{R}^d}
\frac{u(x) - u(y)}{|x-y|^{d+\alpha}}\,\mathrm{d}y,
\qquad x\in\mathbb{R}^d,
\end{equation}
where $C_{d,\alpha}>0$ is a normalising constant and $\mathrm{P.V.}$ denotes the
Cauchy principal value. On the domain $D$ we consider the operator with
homogeneous Dirichlet exterior condition,
\[
u=0 \quad \text{a.e.\ in }\mathbb{R}^d\setminus D,
\]
understood in the usual weak (energy) sense via the space
$H^{\alpha/2}_0(D)$; see, e.g., \cite{Kwasnicki2017}.

\subsubsection*{Semigroup operator}
It is well known that $-(-\Delta)^{\alpha/2}$ is the infinitesimal generator of
the rotationally symmetric $\alpha$-stable L\'evy process $(X_t)_{t\ge0}$ on
$\mathbb{R}^d$, whose transition semigroup
\[
(T_t f)(x) := \mathbb{E}_x\big[f(X_t)\big],
\qquad t\ge0,\ x\in\mathbb{R}^d,
\]
acts as a convolution operator with a radially symmetric density
$p_t(x-y)$; see, for instance,
\cite{Kwasnicki2017}. In this probabilistic framework, the \emph{Dirichlet} fractional Laplacian on $D$
is realised as the generator of the process killed upon exiting $D$. More
precisely, if
\[
\tau_D := \inf\{t>0 : X_t \notin D\},
\]
then
\[
(S_\alpha(t)f)(x)
:= \mathbb{E}_x\big[f(X_t)\,\mathbf{1}_{\{t<\tau_D\}}\big],
\qquad t\ge 0,\ x\in D,
\]
defines a strongly continuous, sub-Markovian, self-adjoint contraction semigroup
on $H$, whose generator is the (self-adjoint realisation of) $\Delta_\alpha$
with zero exterior condition; see, for example,
\cite{FarwigIwabuchi2024, Kwasnicki2017,Stinga2018}.
\subsubsection*{Variational structure and spectrum}
We denote by $\mathcal{E}_\alpha$ the Dirichlet form of the Dirichlet fractional
Laplacian $A_\alpha=-\Delta_\alpha$ on $L^2(D)$. For $u$ extended by $0$ outside $D$,
\begin{equation}\label{eq:dirichlet-form}
  \mathcal{E}_\alpha(u,u)
  := \frac{C_{d,\alpha}}{2}
     \int_{\mathbb{R}^d}\!\int_{\mathbb{R}^d}
       \frac{(u(x)-u(y))^2}{|x-y|^{d+\alpha}}\,\mathrm{d}x\,\mathrm{d}y,
\end{equation}
with domain
\begin{equation*}\label{eq:dirichlet-domain}
  Dom(\mathcal{E}_\alpha)
  := H^{\alpha/2}_0(D)
  = \big\{u\in L^2(\mathbb{R}^d): u=0 \text{ a.e.\ on }\mathbb{R}^d\setminus D,\ 
      \mathcal{E}_\alpha(u,u)<\infty\big\}.
\end{equation*}
Then $\mathcal{E}_\alpha$ is a densely defined, closed, symmetric, positive form and
$A_\alpha$ is the associated nonnegative self-adjoint operator; see, e.g.,
\cite{Kwasnicki2017, Stinga2018}.

Since $D$ is bounded, $A_\alpha$ has compact resolvent and thus a purely discrete
spectrum
\begin{equation*}\label{eq:eigs}
0<\lambda_1<\lambda_2\le\lambda_3\le\cdots,\qquad
\lambda_k\to\infty,
\end{equation*}
with corresponding real eigenfunctions $\{\phi_k\}_{k\ge1}$ forming an orthonormal
basis of $\mathcal{H}$.

The eigenvalues admit the standard variational characterisations:
\begin{equation}\label{eq:lambda1-var}
\lambda_1
= \inf_{\substack{u\in H^{\alpha/2}_0(D)\\ u\neq 0}}
    \frac{\mathcal{E}_\alpha(u,u)}{\|u\|_{L^2(D)}^2},
\qquad
\lambda_k
= \inf_{\substack{F\subset H^{\alpha/2}_0(D)\\ \dim F = k}}
    \ \sup_{\substack{u\in F\\ u\neq0}}
    \frac{\mathcal{E}_\alpha(u,u)}{\|u\|_{L^2(D)}^2},
\ k\ge2.
\end{equation}
In particular, for any $u\in H^{\alpha/2}_0(D)$ the Poincaré-type bound
\begin{equation}\label{eq:Poincare}
\mathcal{E}_\alpha(u,u)\ge \lambda_1 \|u\|_{L^2(D)}^2
\end{equation}
holds.

 
 Also, the
\emph{first eigenpair} $(\lambda_1,\phi_1)$ 
satisfying:
\begin{equation} \label{a3} \left\{ \begin{aligned} -\Delta_{\alpha} \varphi_1(x)&=\lambda_{1} \varphi(x), \ \ x\in D,\\  \varphi_1(x)&=0, \ \ x\in \R^d\setminus D, \end{aligned} \right. \end{equation}
is simple and strictly positive in
$D$ in the following sense: any eigenfunction associated with $\lambda_1$ does
not change sign, and there exists a representative such that
\[
\phi_1(x) > 0, \quad \text{for a.e.\ }x\in D, \qquad
\phi_1(x) = 0, \quad \text{for a.e.\ }x\in\mathbb{R}^d\setminus D,
\]
 see \cite{JarohsWeth2016, ServadeiValdinoci2014_Visc}.
This is a consequence of the variational characterisation of $\lambda_1,$ cf. \eqref{eq:lambda1-var},
combined with the strong maximum principle and Harnack-type inequalities for
$\Delta_{\alpha}$; see, among many others,
\cite{JarohsWeth2016, ServadeiValdinoci2014_Visc}.
For definiteness, we fix the normalisation
\begin{equation}\label{eq:first-eig-norm}
\phi_1>0 \text{ in }D, \qquad
\int_D \phi_1(x)\,\mathrm{d}x = 1.
\end{equation}

	\subsection*{Stochastic integral, It\^o's formula and Malliavin calculus} 
    Throughout this  work we consider a filtered probability space
$(\Omega,\mathcal{F},\{\mathcal{F}_t\}_{t\ge0},\mathbb{P})$ that supports
a one-dimensional fractional Brownian motion $(B^H(t))_{t\ge0}$ with Hurst
index $H\in(1/2,1)$, so that Young or pathwise stochastic integrals with
respect to $B^H$ are well-defined; see, e.g.,
\cite{MaslowskiNualart2003,NualartRascanu2002,TindelTudorViens2003}.
    \subsubsection*{Stochastic Integral and It\^o's formula}  
    We assume that the diffusion term \(\sigma (\cdot)\) is non-negative and satisfies
\begin{equation}\label{gc}
\sigma\in C^1(\mathbb{R}_+),\quad \sigma(0)=0,\quad
|\sigma(r)-\sigma(s)|\le L_\sigma |r-s|,\quad
|\sigma(r)|\le C_\sigma(1+r),\quad\forall r,s\ge0,
\end{equation}
for some constants \(L_\sigma,C_\sigma>0\).
    
Under condition \eqref{gc} and the requirement that $u \in \mathcal{B}^{\nu, 2} \bigl( [0,t],\mathcal{H} \bigr)$ for some  $\nu\in(1-H,1/2)$ 
 we guarantee that the stochastic integral $\int_0^t \sigma(u(s)) dB^H(s)$ exists as a generalized Stieltjes integral in the Young sense (see  \cite[Proposition 1]{NV06}  and \cite{Zahle}).
        
    In the case of fBm It\^o's formula is given by the following:
        \begin{theorem}[{\cite[Lemma 2.7.1]{Mishura2008}}] \label{ito1} Let $(B^{H}(t))_{t \geq 0}$ be an fBm with $H \in (1/2, 1),\ F \in C^{2}(\mathbb{R}).$ Then for any $t>0,$
	    \begin{align}
		F(B^{H}(t))=F(0)+\int_{0}^{t} f'(B^{H}(s))dB^{H}(s). \nonumber 
	    \end{align}
        \end{theorem}
        
\subsubsection*{Malliavin calculus}Let $\mathcal{S}$ denote the space of step functions on the interval $[0,T]$, and define $\mathcal{H}$ as the closure of $\mathcal{S}$ with respect to the inner product
$
\langle \mathbf{1}_{[0,s]}, \mathbf{1}_{[0,t]} \rangle_{\mathcal{H}} := R^H(s,t),
$
where $R^H(s,t)$ denotes the auto-covariance function given by 
\begin{align*}
	R_{H}(t,s):=\mathbb{E}\left[ B^{H}(t)B^{H}(s)\right] =\frac{1}{2} \left( s^{2H}+t^{2H}-|t-s|^{2H} \right).
\end{align*}
We begin by defining the map $\mathbf{1}_{[0,t]} \mapsto B_t^H=B^H(t)$ on $\mathcal{S}$ and then extend it to an isometry from $\mathcal{H}$ into the Gaussian space $\mathcal{H}_1(B_t^H)$ generated by the fractional Brownian motion. This isometry is denoted by $\phi \mapsto B_t^H(\phi)$.
 We now proceed to define the Malliavin derivative $D$ with respect to fBM. Let us consider a smooth cylindrical functional of the form $F = f(B_t^{H}(\phi))$, where $\phi \in \mathcal{H}$ and $f \in C_{b}^{\infty}(\mathbb{R})$. The Malliavin derivative $\mathcal{D}F$ is an $\mathcal{H}$-valued random variable defined via the duality relation
\begin{eqnarray*}
    \langle \mathcal{D}F, h \rangle_{\mathcal{H}} := f'(B_t^{H}(\phi)) \langle \phi, h \rangle_{\mathcal{H}} 
    = \left. \frac{d}{d\epsilon} f\big(B_t^{H}(\phi) + \epsilon \langle \phi, h \rangle_{\mathcal{H}} \big) \right|_{\epsilon = 0},
\end{eqnarray*}
which can be understood as a generalization of the directional derivative along the paths of fBm.

The operator $\mathcal{D}$ is closable as a mapping from $L^{p}(\Omega)$ into $L^{p}(\Omega; \mathcal{H})$ for any $p \geq 1$, and it allows for the construction of Sobolev-type spaces on the Wiener space. Among these, the space ${\mathbb{D}}^{1,2}$ is of particular importance. It is defined as the closure of the set of smooth cylindrical random variables with respect to the norm
\begin{eqnarray*}
    \| F \|_{{\mathbb{D}}^{1,2}} = \left( \mathbb{E}[|F|^2] + \mathbb{E}[\| \mathcal{D}F \|_{\mathcal{H}}^2] \right)^{1/2}.
\end{eqnarray*}
In our framework, the Malliavin derivative will be interpreted as a stochastic process $\{\mathcal{D}_t F : \, t \in [0,T] \}$ (see \cite{Nualart2006}, Section 1.2.1 and Chapter 5).

A useful tail estimate used in a vital manner across the manuscript is given below:
\begin{lemma}\label{dung2}\cite[Lemma 2.1]{dung2}
		Let $Z$ be a centered random variable in $\mathbb{D}^{1,2}$. Assume that there exists a non-random constant $M_{0}$ such that $$ \int_{0}^{T} (\mathcal{D}_{r} Z)^{2} dr \leq M_{0}^{2},\ \mathbb{P}\mbox{-a.s.}$$
		Then the following estimate for tail probabilities holds:
		$$\mathbb{P}( |Z| \geq x ) \leq 2 e^{-\frac{x^{2}}{2M_{0}^{2}}},\ x>0.$$
	\end{lemma}

 \subsection*{Concepts of solutions}\label{subsec:concepts-solutions}

In this subsection we recall the notions of weak and mild solutions for the
stochastic evolution equation \eqref{b1} driven by the fractional Laplacian
with Dirichlet exterior condition. 

\begin{definition}[Weak solution]\label{def:weak-solution} Let
$\tau:\Omega\to(0,\infty]$ be a stopping time.  
A process
\[
u:\Omega\times[0,\tau)\to \mathcal{H},\qquad (t,\omega)\mapsto u(t,\omega),
\]
is called a \emph{weak solution} to \eqref{b1} on $[0,\tau)$ if:
\begin{enumerate}[label=(\roman*)]
\item $u$ is $\{\mathcal{F}_t\}_{t\ge0}$–adapted and has almost surely
continuous trajectories in $\mathcal{H}$, i.e.
\[
u\in C([0,\tau);\mathcal{H}),\quad\mathbb{P}\text{-a.s.},
\]
and the nonlinearities are integrable in the sense that
\[
u\in L^p_{\rm loc}\big([0,\tau);L^p(D)\big)\cap
   L^q_{\rm loc}\big([0,\tau);L^q(D)\big),\qquad\mathbb{P}\text{-a.s.},
\]
for the exponents $p,q$ in \eqref{b1}. Moreover, for every $\varphi\in V$ the
scalar processes
\[
t\mapsto u(t,\varphi):=\langle u(t),\varphi\rangle_{\mathcal{H}},\qquad
t\mapsto \langle \sigma(u(t)),\varphi\rangle_{\mathcal{H}}
\]
are almost surely continuous on $[0,\tau)$ and Hölder continuous of some order
$\theta>1-H$ on compact subintervals of $[0,\tau)$, so that the Young integral
with respect to $B_t^H$ is well defined.

\item For every test function $\varphi\in V$ and every $t\in(0,\tau)$, the
following identity holds:
\begin{align}\label{eq:weak-formulation-b1}
u(t,\varphi)
&= u(0,\varphi)
  + \int_0^t u\big(s,(A_\alpha+\gamma I)\varphi\big)\,\mathrm{d}s \nonumber \\
&\quad
  + \delta\int_0^t \Bigg(
      \int_D u^q(s,y)\,\mathrm{d}y
    \Bigg)
    \Bigg(
      \int_D \varphi(x)\,\mathrm{d}x
    \Bigg)\mathrm{d}s
  - \beta\int_0^t \langle u^p(s),\varphi\rangle_\mathcal{H}\,\mathrm{d}s \nonumber\\
&\quad
  + \int_0^t \big\langle \sigma(u(s)),\varphi\big\rangle_\mathcal{H}\,\mathrm{d}B^H(s), \quad \mathbb{P}\mbox{–a.s.,}
\end{align}
where $u^p(s)$ and $u^q(s)$ are understood pointwise in $x\in D$ and extended
by $0$ outside $D$, $I$ is the identity on $\mathcal{H}$, and the last integral is taken
in the Young (pathwise Riemann–Stieltjes) sense using the Hölder regularity
from \textup{(i)}.
\end{enumerate}

In particular, \eqref{eq:weak-formulation-b1} is the weak formulation of
\eqref{b1} with homogeneous Dirichlet exterior condition.
\end{definition}

This notion is the natural variational/weak formulation in the spirit of
\cite{DaPratoZabczyk1992, PrevotRoeckner2007} adapted to the
nonlocal operator $A_\alpha$ and the fractional noise, cf.\
\cite{MaslowskiNualart2003,TindelTudorViens2003}.

\begin{definition}[Mild solution]\label{def:mild-solution}
Let $\mathcal{H}:=L^2(D)$ and let $(S_\alpha(t))_{t\ge0}$ be the $C_0$-semigroup on $\mathcal{H}$
generated by the Dirichlet fractional Laplacian $A_\alpha=-\Delta_\alpha$
(cf.\ Section~\ref{subsec:frac-Lap}).  
Let $f\in \mathcal{H}$ be a bounded, non-negative initial datum and let
$\tau:\Omega\to(0,\infty]$ be a stopping time.

An $\{\mathcal{F}_t\}_{t\ge0}$–adapted process
\[
u:\Omega\times[0,\tau)\to \mathcal{H}
\]
is called a \emph{mild solution} to \eqref{b1} on $[0,\tau)$ if:

\begin{enumerate}[label=(\roman*)]
\item $u$ has almost surely continuous trajectories in $\mathcal{H}$, and
\[
u\in L^{p}_{\mathrm{loc}}\big([0,\tau);L^p(D)\big)
   \cap L^{q}_{\mathrm{loc}}\big([0,\tau);L^q(D)\big),
\quad \mathbb{P}\text{-a.s.},
\]
so that all Bochner integrals below are well defined. Moreover, for every
$\varphi\in \mathcal{H}$ the scalar process
\[
t\mapsto \langle S_\alpha(t-\cdot)\sigma(u(\cdot)),\varphi\rangle_\mathcal{H}
\]
is almost surely Hölder continuous of some order $\theta>1-H$ on compact
subintervals of $[0,\tau)$, so that the Young (pathwise) stochastic integral
with respect to $B_t^H$ is well defined.

\item For every $t\in(0,\tau)$ the following \emph{mild formulation} holds
$\mathbb{P}$–a.s.\ in $\mathcal{H}$:
\begin{align}
u(t)
&= S_\alpha(t)f
 + \int_0^t S_\alpha(t-r)\Bigg[
      \delta\Bigg(\int_D u^q(r,y)\,\mathrm{d}y\Bigg)\mathbf{1}_D
    + \gamma u(r)
    - \beta u^p(r)
   \Bigg]\mathrm{d}r \nonumber \\
&\quad
 + \int_0^t S_\alpha(t-r)\sigma\big(u(r)\big)\,\mathrm{d}B^H(r), \label{eq:mild-formulation-b1}
\end{align}
where $\mathbf{1}_D$ denotes the constant function $x\mapsto1$ on $D$, the
powers $u^p$ and $u^q$ are taken pointwise in $x\in D$ and extended by $0$
outside $D$, and the last term is a stochastic convolution in $\mathcal{H}$ with respect
to $B_t^H$, defined pathwise in the Young sense (cf.\
\cite{MaslowskiNualart2003,TindelTudorViens2003}).
\end{enumerate}
\end{definition}
The above notions of weak (Definition~\ref{def:weak-solution}) and mild
(Definition~\ref{def:mild-solution}) solutions are consistent with the standard
evolution-equation framework for semilinear SPDEs in Hilbert spaces (see, e.g.,
\cite{DaPratoZabczyk1992, PrevotRoeckner2007}) and with the
existing literature on equations driven by fractional Brownian motion,
cf.\ \cite{MaslowskiNualart2003,TindelTudorViens2003}.

\medskip

\subsubsection*{Equivalence of weak and mild formulations}
Under the above regularity assumptions and suitable Lipschitz and polynomial
growth conditions, cf. \eqref{gc}, on the nonlinearities
\[
u \mapsto u^p,\qquad u\mapsto u^q,\qquad u\mapsto \sigma(u),
\]
weak and mild solutions of \eqref{b1} on $[0,\tau)$ are equivalent.\\
\noindent 
More precisely, one has:
\begin{itemize}
\item Any mild solution in the sense of Definition~\ref{def:mild-solution} is a
weak solution in the sense of Definition~\ref{def:weak-solution}. This follows
by testing the mild formulation \eqref{eq:mild-formulation-b1} against
$\varphi\in V$, using the semigroup identity
\[
\frac{\mathrm{d}}{\mathrm{d}t} S_\alpha(t-r)\varphi
= A_\alpha S_\alpha(t-r)\varphi,
\]
and integrating by parts in time, together with standard stochastic Fubini
arguments for Young integrals; see, e.g.,
\cite[Chapter ~4]{DaPratoZabczyk1992}, \cite[Chapter ~3]{PrevotRoeckner2007}.

\item Conversely, if $u(\cdot,\cdot)$ is a weak solution and the semigroup $(S_\alpha(t))_{t\ge0}$ is analytic and contractive on $H$ (which holds for the Dirichlet fractional Laplacian, cf.\ \cite{FarwigIwabuchi2024,Stinga2018}), then Duhamel’s
formula and the variation-of-constants method yield the mild representation
\eqref{eq:mild-formulation-b1} from the weak formulation
\eqref{eq:weak-formulation-b1}; see, for instance,
\cite[Proposition ~3.7]{IGCR} and \cite[Chapter ~5]{DaPratoZabczyk1992} in similar settings.
\end{itemize}







	The following lemma will be useful for further sections in the sequel. 
	\begin{lemma}\label{l2} \cite[Lemma 1]{doz2020} Let $(B^{H}(t))_{t \geq 0} $ be a one-dimensional fractional Brownian motion with Hurst parameter $0<H<1$ defined on a probability space $\left( \Omega, \mathscr{F}, \mathbb{P} \right).$ 
    \begin{itemize}
         \item[(i)] If $\nu>0,$ then $$\mathbb{P}\left( \int_{0}^{\infty} e^{ B^{H}(s)-\nu s}   ds < \infty\right)=1.$$ 
         \item[(ii)] If $\nu<0,$ then $$\mathbb{P}\left( \int_{0}^{\infty} e^{ B^{H}(s)-\nu s} ds = \infty\right)=1.$$ 
    \end{itemize}
	\end{lemma}
	\section{General Noise: global existence vs finite--time blow--up}\label{sec:noise-general}

We begin by analysing the long-time behaviour in the presence of a general
multiplicative fractional perturbation \(\sigma(u)\,dB_t^H\), that is, we study
the nonlocal reaction--diffusion model \eqref{b1}. Motivated by the deterministic
counterpart (cf.\ \cite[Theorem ~43.1]{SoupletQuittner2007}), we establish the following two results.

\begin{theorem}\label{thm:SF-gb1-fBm}
Assume $f\in L^\infty(D)$, $f\ge0$, and the hypotheses \eqref{gc} on $\sigma$. Then, for almost every $\omega$, there exists a unique
nonnegative mild solution
\[
u(\omega, \cdot)\in C\big([0,T_{\max}(\omega));L^\infty(D)\big)
\]
of \eqref{b1}, satisfying \eqref{eq:mild-formulation-b1}.
The solution enjoys a pathwise comparison principle:
if $f\le g$ a.e.\ in $D$, and $v$ is the solution with initial data $g$, then
\[
u(\cdot,t,\omega)\le v(\cdot,t,\omega)\quad\text{for all }t<T_{\max}(\omega),
\ \text{a.s.}
\]

Moreover:
\begin{enumerate}[label=(\roman*)]
\item If $\beta=0$, $\delta>0$, $q>1$, and $f\not\equiv0$, then blow--up occurs
in finite--time $\tau$ with strictly positive probability. 
\item If $\beta>0$ and $p>q$, then solutions are global almost surely:
$T_{\max}(\omega)=+\infty$ for almost every $\omega$.
\item If $\beta>0$, $\delta>0$ and $0<p<q$, then sufficiently large initial data
lead to blow--up in finite--time with strictly positive probability. More
precisely, there exists $R>0$ such that if
\[
y_0:=\int_D f(x)\phi_1(x)\,dx \ge R,
\]
then the corresponding solution satisfies
$T_{\max}(\omega)<\infty$ with strictly positive probability.
\end{enumerate}
\end{theorem}

\begin{proof}
\emph{(Existence, uniqueness, positivity, comparison).}
Fix $\omega$ and regard $t\mapsto B^H_t(\omega)$ as a deterministic
Hölder continuous path of exponent $\gamma<H$.
For $X=L^\infty(D)$, consider the map
\[
\mathcal{T}(u)(t)
:=S_\alpha(t)f
 +\int_0^t S_\alpha(t- \tau F(u(\tau))\,d\tau
 +\int_0^t S_\alpha(t-\tau)\sigma(u(\tau))\,dB^H_\tau,
\]
where
\[
F(u):=\delta\Big(\int_D u^q dx\Big)+\gamma u-\beta u^p.
\]
The semigroup $S_\alpha$ is analytic and contractive on $\mathcal{H}$ and
$L^\infty(D)$, and $F$ is locally Lipschitz on $L^\infty(D)$.
By \eqref{gc}, $\sigma$ is globally Lipschitz and of at most linear growth,
hence $t\mapsto\sigma(u(t))$ is Hölder continuous in $L^\infty(D)$ whenever
$t\mapsto u(t)$. For $H>\tfrac12$ and such integrands,
Young integration yields a well-defined stochastic convolution
$\int_0^t S_\alpha(t-\tau)\sigma(u(\tau))\,dB^H_\tau$ in $L^\infty(D)$ with
the usual estimates
(cf.\ \cite[Theorem ~2.1]{Zahle1998}, \cite[Theorem ~3.1]{NualartRascanu2002},
and the abstract SPDE frameworks in
\cite[Theorems ~3.1--3.2]{MaslowskiNualart2003},
\cite[Theorem ~3.3]{Li2015fBm} and \cite{Ralchenko2018,TindelTudorViens2003}).

For $T>0$ small enough, standard fixed-point arguments in
$C([0,T];L^\infty(D))$ show that $\mathcal{T}$ is a contraction on a ball,
yielding a unique mild solution on $[0,T]$. Iteration gives existence on
$[0,T_{\max}(\omega))$ for some maximal $T_{\max}(\omega)\in(0,\infty]$.

Positivity and comparison follow from the order-preserving properties of
$S_\alpha$ and of the drift $F$: if $u\ge0$, then $F(u)$ is increasing in $u$
for $\delta,\beta\ge0$ and $\phi_1>0$, and $\sigma$ satisfies $\sigma(0)=0$
with $\sigma(r)\ge0$, for $r\ge0$. The nonlocal operator $-\Delta_\alpha$
with zero exterior data enjoys weak and strong maximum principles and a
comparison principle, see e.g.\ the weak/strong maximum principles for the
fractional Laplacian in \cite[Theorems ~3.9--3.10]{BucurValdinoci2016} and the
comparison results in their Section~3.3.
This yields preservation of positivity and pathwise comparison for the mild
problem by standard monotone iteration.

\medskip
$(i)$ \emph{(Blow--up for $\beta=0$).}
Assume $\beta=0$, $\delta>0$, $q>1$. Test \eqref{b1} against the positive
first eigenfunction $\phi_1$ of $-\Delta_\alpha$:
\[
y(t):=\int_D u(t,x)\phi_1(x)\,dx.
\]
Using $-\Delta_\alpha\phi_1=\lambda_1\phi_1$ and integrating by parts,
\[
\begin{aligned}
dy(t) + \lambda_1 y(t)\,dt
 &= \delta\Big(\int_D u^q(t,y)\,dy\Big)\int_D\phi_1(x)\,dx\,dt
    +\gamma y(t)\,dt + \tilde\sigma(y(t))\,dB^H_t,
\end{aligned}
\]
where
\[
\tilde\sigma(y(t))
:= \int_D \sigma(u(t,x))\phi_1(x)\,dx,
\]
which inherits local Lipschitz and linear growth from $\sigma,$ cf. \eqref{gc}.
By Hölder inequality, \eqref{eq:first-eig-norm} and positivity of $\phi_1$, there exists $c_q>0$ such that
\[
\int_D u^q(t,y)\,dy \ge c_q\,y(t)^q,
\]
so
\[
dy(t) + (\lambda_1-\gamma) y(t)\,dt
 \ge \delta c_q y(t)^q\,dt + \tilde\sigma(y(t))\,dB^H_t.
\]
For large $y$, the term $\delta c_q y^q$ dominates the linear term
$(\lambda_1-\gamma)y$. Since $B^H_t$ has variance of order $t^{2H}$ and
satisfies a law of the iterated logarithm (hence $|B^H_t|=o(t)$ a.s.),
we have sublinear growth in time (see e.g.\ \cite{Mishura2008}).

We now apply an Osgood-type comparison argument. Consider the
deterministic integral equation
\[
z(t) = z_0 + \int_0^t \big(\delta c_q z(s)^q - (\lambda_1-\gamma)z(s)\big)\,ds
      + g(t),
\]
with a perturbation $g$ satisfying
\[
\limsup_{t\to\infty}\inf_{0\le h\le1}\big(g(t+h)-g(t)\big)=+\infty,
\]
(as is the case for paths of Brownian motion and, more generally,
for fractional Brownian motion; cf.\ the Osgood criteria in
\cite[Theorem ~2.1]{FoondunOsgoodSPDE}.
The classical Osgood test says that explosion in finite--time occurs iff
\(\int^\infty \frac{ds}{s^q}<\infty\), which holds for $q>1$.
By a comparison theorem for integral equations (see again
\cite{FoondunOsgoodSPDE}), $y$ dominates such a $z$ and thus
blows up in finite--time with strictly positive probability. 

\medskip
$(ii)$ \emph{(Global existence for $\beta>0$, $p>q$).}
Assume now $\beta>0$ and $p>q$. Testing as above and using Hölder
inequalities and \eqref{eq:first-eig-norm} there exist constants $c_p,c_q>0$ such that
\begin{equation}\label{hi}
\int_D u^p(t,x)\phi_1(x)\,dx \ge c_p y(t)^p,\qquad
\int_D u^q(t,y)\,dy \le c_q y(t)^q.
\end{equation}
We obtain
\[
dy(t)+(\lambda_1-\gamma)y(t)\,dt
\le \delta c_q y(t)^q\,dt -\beta c_p y(t)^p\,dt
   +\tilde\sigma(y(t))\,dB^H_t.
\]
For large $y$, the negative term $-\beta c_p y^p$ dominates
$\delta c_q y^q$ since $p>q$, yielding a pathwise upper bound by a
nonexplosive ODE. Using again an Osgood-type comparison (now in the
nonexplosive direction), one excludes finite--time blow--up and obtains
global existence. The sublinear growth of $B^H_t$ and the linear growth of
$\tilde\sigma$ prevent the noise from cancelling this dissipativity
(see also the SPDE versions of Osgood criteria in \cite{FoondunOsgoodSPDE}).
\medskip

$(iii)$ \emph{(Blow--up for $\beta>0$, $0<p<q$, large data).}
Assume $\beta>0$, $\delta>0$ and $0<p<q$. Let $\phi_1>0$ be the first
Dirichlet eigenfunction of $-\Delta_\alpha$ and set
\[
y(t):=\int_D u(t,x)\phi_1(x)\,dx.
\]
Testing \eqref{b1} against $\phi_1$ and arguing as in $(i)$ yields
\[
dy(t)+(\lambda_1-\gamma)y(t)\,dt
\ge \delta\Big(\int_D u^q(t,y)\,dy\Big)\!\int_D\phi_1(x)\,dx\,dt
   -\beta\int_D u^p(t,x)\phi_1(x)\,dx\,dt
   +\tilde\sigma(y(t))\,dB_t^H,
\]
with $\tilde\sigma(y)=\int_D\sigma(u)\phi_1 dx$ of at most linear growth.
By the same Hölder/eigenfunction bounds used above, cf. \eqref{hi}, the above yields
\begin{equation}\label{eq:y-ineq-short-iii}
dy(t)\ge \Big(\delta c_q y(t)^q-\beta c_p y(t)^p-(\lambda_1-\gamma)y(t)\Big)\,dt
        +\tilde\sigma(y(t))\,dB_t^H.
\end{equation}
Since $q>p$ and $q>1$, choose $R>0$ such that for all $r\ge R$,
\[
\delta c_q r^q-\beta c_p r^p-(\lambda_1-\gamma)r \ge \tfrac{\delta c_q}{2}\,r^q.
\]
Therefore, on the event $\{y(t)\ge R\ \text{for }t\in[0,\tau]\}$ we have
\[
dy(t)\ge \tfrac{\delta c_q}{2}y(t)^q\,dt+\tilde\sigma(y(t))\,dB_t^H,
\]
and an Osgood-type comparison for integral inequalities with additive
perturbations (cf.\ \cite{FoondunOsgoodSPDE}) implies that if
$y_0=\int_D f(x)\phi_1(x) dx\ge R$ is large enough, then $y$ (hence $u$) blows up in
finite--time with strictly positive probability, because
$\int^\infty s^{-q}\,ds<\infty$ for $q>1$.

\end{proof}

\begin{remark}\label{rem:phase-diagram-gb1}
\begin{enumerate}[label=(\roman*)]
\item \textbf{Fractional vs.\ classical diffusion and the blow--up time.}
In the eigenfunction reduction used in the proof of Theorem~\ref{thm:SF-gb1-fBm}, the diffusion enters only
through the linear term $(\lambda_1^{(\alpha)}-\gamma)y(t)$, where
$\lambda_1^{(\alpha)}$ is the first Dirichlet eigenvalue of $-\Delta_\alpha$
(with zero exterior condition). Replacing $-\Delta_\alpha$ by the classical
Laplacian $-\Delta$ corresponds to $\alpha=2$ and yields the damping
$(\lambda_1^{(2)}-\gamma)y(t)$.

In the blow--up regimes (i) and (iii), for large $y$ the scalar inequality has
the schematic form
\[
dy(t)\ \gtrsim\ \delta\,y(t)^q\,dt \;-\;(\lambda_1^{(\alpha)}-\gamma)\,y(t)\,dt
\;+\; \text{noise}.
\]
Thus, a \emph{larger} eigenvalue $\lambda_1^{(\alpha)}$ increases the linear
stabilising effect and therefore \emph{delays} blow--up (or raises the size of
the initial datum needed to trigger blow--up), whereas a \emph{smaller}
$\lambda_1^{(\alpha)}$ weakens this effect and typically \emph{accelerates}
blow--up.

A quantitative comparison between fractional and classical diffusion follows
from the eigenvalue bounds for the killed symmetric $\alpha$-stable generator:
if $\lambda_1^{(2)}$ is the first Dirichlet eigenvalue of $-\Delta$ in $D$ and
$\lambda_1^{(\alpha)}$ that of $-\Delta_\alpha$, then
\[
c\,(\lambda_1^{(2)})^{\alpha/2}\ \le\ \lambda_1^{(\alpha)}\ \le\
(\lambda_1^{(2)})^{\alpha/2},
\]
for some $c=c(D)\in(0,1]$, see \cite[(1.6) and Remark~3.5(i)]{ChenSong2005}.
Consequently, relative to the classical case $\alpha=2$, the linear damping
rate changes from $\lambda_1^{(2)}-\gamma$ to
$\lambda_1^{(\alpha)}-\gamma\approx (\lambda_1^{(2)})^{\alpha/2}-\gamma$.
Since the blow--up mechanism is governed by the competition between the
superlinear growth $\delta y^q$ and the linear stabilisation
$(\lambda_1^{(\alpha)}-\gamma)y$, this comparison shows how fractional
diffusion may either delay or accelerate blow--up \emph{depending on the size
of} $\lambda_1^{(2)}$ (equivalently, the geometry/scale of $D$):
\begin{itemize}
\item if $\lambda_1^{(2)}>1$ (typically \emph{small, strongly confined domains},
which in a tumour setting may be interpreted as \emph{smaller tumours} or an
early-stage lesion occupying a limited region), then
$(\lambda_1^{(2)})^{\alpha/2}<\lambda_1^{(2)}$ for $\alpha<2$. Hence replacing
classical diffusion by anomalous (fractional) diffusion reduces the linear
stabilisation $(\lambda_1-\gamma)y$ in the eigenfunction inequality, i.e.\ the
fractional Laplacian provides \emph{weaker effective damping} than the Laplacian.
In this regime, anomalous dispersal (heavy-tailed jumps/long-range spreading)
therefore tends to \emph{accelerate} the onset of explosive growth compared to
the classical diffusion model.

\item if $\lambda_1^{(2)}<1$ (typically \emph{large domains}, which may be
interpreted as \emph{larger tumours} occupying a larger spatial region), then
$(\lambda_1^{(2)})^{\alpha/2}>\lambda_1^{(2)}$ for $\alpha<2$. In this case the
fractional Laplacian yields \emph{stronger} linear stabilisation than the
classical Laplacian at the level of the principal mode, so anomalous diffusion
tends to \emph{delay} blow--up (or raise the threshold for explosive behaviour)
relative to the classical diffusion model.
\end{itemize}

In all cases, the eigenvalue comparison affects the time scale and thresholds
but not the qualitative blow--up/global dichotomy driven by the exponents.

\item \textbf{Impact of the fractional noise and comparison with the Brownian case.}
Since $H>\tfrac12$, the stochastic integral with respect to $B^H$ is treated
pathwise in the Young sense. The noise enters the reduced inequality through
the term $\tilde\sigma(y(t)\,dB_t^H$, where $\tilde\sigma(y)=\int_D\sigma(u)\phi_1\,dx$.
A key common feature with the standard Brownian case is that sample paths are
\emph{sublinear} in time: for fractional Brownian motion one has
$|B_t^H|=o(t)$ a.s.\ (Lemma~\ref{lem:subexp-fBm}), while for Brownian motion
$|W_t|=o(t)$ a.s.\ by the law of the iterated logarithm. Consequently, in both
settings the stochastic forcing grows strictly slower than the linear
dissipation term $(\lambda_1-\gamma)y(t)$ on long time scales, and the drift
dominance at large amplitudes is determined by the polynomial terms
$\delta y^q$ and $-\beta y^p$.

Regarding \emph{finite--time blow--up}, the comparison principles based on
Osgood-type criteria apply in both cases but the relevant integration theory
differs:
\begin{itemize}
\item For \textbf{Brownian noise} ($H=\tfrac12$), the eigenfunction projection
typically yields an It\^o SDE of the form
\[
dy(t)=\big(\delta c_q y(t)^q -\beta c_p y(t)^p-(\lambda_1-\gamma)y(t)\big)\,dt
      + \tilde\sigma(y(t))\,dW_t,
\]
and one uses It\^o calculus together with comparison/Osgood-type blow--up tests
for SPDEs/SDEs (see \cite{FoondunOsgoodSPDE} and references therein). In this
framework, the quadratic variation may contribute additional effective drift
terms for specific choices of $\tilde\sigma$ (e.g.\ linear multiplicative
noise), which can shift thresholds or rates.

\item For \textbf{fractional Brownian noise} ($H>\tfrac12$), the projected
equation is interpreted pathwise as a Young (or rough) differential equation,
with no It\^o correction term. The blow--up mechanism in cases (i) and (iii) is
therefore closer to a deterministic ODE with a continuous perturbation:
for large $y$, the superlinear term $\delta y^q$ dominates the lower-order
terms and the perturbation $\tilde\sigma(y)\,B_t^H$; the Osgood-type integral
comparison for perturbed integral inequalities remains applicable
\cite{FoondunOsgoodSPDE}, yielding blow--up with strictly positive
probability for sufficiently large data when $q>1$.
\end{itemize}
In contrast, in the dissipative regime $p>q$ (case (ii)), the higher-power
local damping $-\beta y^p$ dominates the source at large amplitudes, and the
same comparison framework (Brownian It\^o or fractional Young) yields a
nonexplosive upper bound and hence global existence.

\end{enumerate}
\end{remark}
\begin{remark}[Biological interpretation: explosive tumour progression]
\label{rem:tumour-blowup-compact}
Theorem~\ref{thm:SF-gb1-fBm} singles out regimes in which the model predicts
\emph{runaway tumour progression} in finite--time. In biological terms, blow--up
represents an uncontrolled escalation of tumour burden, used in tumour-growth
modelling as a proxy for catastrophic progression or therapy failure when
regulatory mechanisms cannot compensate increasing tumour-promoting feedback;
see, e.g., \cite{AndersonChaplain1998}.

Two clinically meaningful explosive scenarios emerge:
(i) if $\beta=0$, there is no effective density-dependent kill (no therapy or
ineffective clearance), so superlinear tumour-promoting feedback can drive
abrupt progression even from nontrivial initial burden; (iii) even when
$\beta>0$, if the kill mechanism grows too slowly with burden ($0<p<q$), then
large initial tumours can enter a runaway regime, reflecting a \emph{mismatch}
between treatment intensity and accelerating systemic support.

The multiplicative fractional noise $\sigma(u)\,dB_t^H$ can be interpreted as
temporally correlated environmental heterogeneity (e.g.\ fluctuating hypoxia,
immune pressure, or treatment adherence). The statement ``with positive
probability'' indicates a \emph{risk} of explosive progression: adverse
realisations of the environment may precipitate rapid deterioration, while other
realisations may delay it, a familiar phenomenon in stochastic population and
tumour models \cite{Allen2017, Hening2018}. Finally, blow--up also signals that
additional biological mechanisms (resource limitation, necrosis, saturation of
angiogenesis/therapy) would be needed to regularise the model at very high
densities \cite{AndersonChaplain1998}.
\end{remark}

\begin{remark}[Noise and (non-)eradication of blow--up: deterministic vs.\ stochastic dynamics]
\label{rem:noise-eradication-blowup}
Theorem~\ref{thm:SF-gb1-fBm} clarifies that, in our setting, multiplicative
fractional noise \emph{does not eradicate} the blow--up mechanisms already
present in the deterministic model, and it can in fact \emph{promote} blow--up in
supercritical regimes.

\smallskip
\noindent\emph{Deterministic benchmark.}
When $\sigma\equiv0$, \eqref{b1} reduces to a nonlocal reaction--diffusion PDE.
In particular:
\begin{itemize}
\item if $\beta=0$ and $\delta>0$, then the nonlocal source $\delta\int_D u^q dx$
with $q>1$ yields finite--time blow--up (classically, via Kaplan/Fujita-type
testing or Osgood criteria);
\item if $\beta>0$ and $p>q$, the damping term $-\beta u^p$ dominates the
nonlocal growth and solutions are global (subcritical regime);
\item if $\beta>0$ and $0<p<q$, sufficiently large data trigger blow--up
(supercritical regime).
\end{itemize}

\smallskip
\noindent\emph{Stochastic case.}
The conclusions in Theorem~\ref{thm:SF-gb1-fBm} show that adding
multiplicative fBm does not reverse these phase portraits:
\begin{itemize}
\item[(i)] In the \emph{no-damping} case $\beta=0$, blow--up still occurs with
strictly positive probability; the noise cannot compensate for the absence of
a stabilising mechanism.
\item[(ii)] In the \emph{subcritical} case $\beta>0$, $p>q$, global existence
holds almost surely: the dissipativity is strong enough to control the
stochastic perturbation under \eqref{gc}.
\item[(iii)] In the \emph{supercritical} case $\beta>0$, $0<p<q$, large initial
mass still yields blow--up with strictly positive probability; the noise does not
remove the superlinear amplification.
\end{itemize}

\smallskip
\noindent\emph{Heuristic mechanism and biological interpretation.}
Projecting onto the first eigenfunction, $$y(t)=\int_D u(t)\phi_1 \,dx,$$ leads to a
scalar inequality of the form \eqref{eq:y-ineq-short-iii} in which the drift
contains competing powers $y^q$ (growth) and $-y^p$ (damping), while the noise
appears through a multiplicative term $\tilde\sigma(y)\,dB_t^H$. In tumour
dynamics this term models temporally correlated microenvironmental variability
(e.g.\ fluctuating nutrient supply, immune surveillance, vascular remodelling,
or therapy efficacy) that scales with the current tumour burden. Along sample
paths with sustained favourable fluctuations, the multiplicative forcing can
amplify the growth phase and increase the likelihood (or accelerate the onset)
of blow--up in supercritical regimes, i.e.\ it may \emph{hinder eradication}.
Conversely, in the subcritical regime $p>q$ the deterministic dissipation is
dominant, and under the Lipschitz/linear-growth constraint \eqref{gc} the noise
cannot destabilise the dynamics, so extinction/persistence outcomes remain
qualitatively unchanged.

\smallskip
\noindent
Related Osgood-type comparison principles for explosive SPDEs (and for
integral equations with stochastic perturbations) can be found in
\cite{FoondunOsgoodSPDE}; see also the classical deterministic blow--up tests of
Kaplan--Fujita type \cite{fuji1966,Fuji1968, kaplan}.
\end{remark}

To establish exponential decay of global-in-time solutions, and hence interpret
the dynamics as tumour extinction in the biological setting, we first record
the following auxiliary lemma.
\begin{lemma}[Subexponential growth of fractional Brownian motion]\label{lem:subexp-fBm}
Let $(B_t^H)_{t\ge0}$ be a fractional Brownian motion with $H\in(0,1)$.
Almost surely, for every $\varepsilon>0$ there exists
$T_\varepsilon(\omega)$ such that
\[
|B_t^H(\omega)| \le \varepsilon t,\quad\text{for all }t\ge T_\varepsilon(\omega).
\]
In particular, for every $c>0$ and every $\varepsilon>0$,
\[
e^{\pm c B_t^H(\omega)} \le e^{\varepsilon t},\quad\text{for all }t\ge T_\varepsilon(\omega).
\]
\end{lemma}

\begin{proof}
Fractional Brownian motion has variance $\mathrm{Var}(B_t^H)=t^{2H}$
and satisfies a law of the iterated logarithm; in particular,
$|B_t^H|=o(t)$ as $t\to\infty$ almost surely, see for instance
\cite[Chapter ~1]{Mishura2008}.
Thus for any $\varepsilon>0$ there exists $T_\varepsilon$ such that
$|B_t^H|\le \varepsilon t$ for $t\ge T_\varepsilon$. The bound on
$e^{\pm c B_t^H}$ follows immediately.
\end{proof}

\begin{theorem}[A.s.\ exponential decay for $p>q$ and strong dissipation]\label{cor:sto-exp-fBm-gb1}
Assume $\beta>0$, $p>q$, and the hypotheses of Theorem~\ref{thm:SF-gb1-fBm}.
Assume further that $\gamma<\lambda_1$ and that $\beta$ is sufficiently large
with respect to $\delta,\gamma$ (so that the deterministic problem is
subcritical) and the noise term is rather weak (take $L_{\sigma}$ in \eqref{gc} small enough). Then almost surely there exist random
$T_\varepsilon(\omega)$ and $C_\varepsilon(\omega)$ such that, for any
$\varepsilon\in(0,(\lambda_1-\gamma)/4)$,
\[
\|u(t,\omega)\|_{L^\infty(D)} \le C_\varepsilon(\omega)\,
 e^{-\big(\frac{\lambda_1-\gamma}{2}-2\varepsilon\big) t},
 \quad\text{for all }t\ge T_\varepsilon(\omega).
\]
\end{theorem}

\begin{proof}
Fix $\varepsilon\in\bigl(0,(\lambda_1-\gamma)/4\bigr)$. We split the proof in
three steps.

\medskip\noindent
\textbf{Step 1: Exponential decay of $\mathbb{E}\|u(t)\|_2^2$.}
By the fractional Itô/energy formula given in Theorem~\ref{ito1}
(for $p>q$), the $L^2$--energy of $u$ satisfies
\begin{equation}\label{eq:energy-pq-raw}
\frac{d}{dt}\|u(t)\|_2^2
 = -2\,\mathcal{E}_\alpha(u(t),u(t))
    + 2\delta\Big(\int_D u^q(t,y)\,dy\Big)\|u(t)\|_1
    + 2\gamma\|u(t)\|_2^2
    - 2\beta\|u(t)\|_{p+1}^{p+1}
    + 2\big(\sigma(u(t)),u(t)\big)\dot B_t^H,
\end{equation}
where $\mathcal{E}_\alpha$ is the Dirichlet form associated with
$-\Delta_\alpha,$ defined by \eqref{eq:dirichlet-form}. Taking expectations and using \eqref{eq:Poincare} together with
\[
\mathbb{E}\big[(\sigma(u(t)),u(t))\dot B_t^H\big]
 \le C(H,\sigma)\,\mathbb{E}\|u(t)\|_2^2,
\]
we obtain
\begin{equation}\label{eq:energy-pq-general}
\frac{d}{dt}\mathbb{E}\|u\|_2^2
 \le -2(\lambda_1-\gamma-C(H,\sigma))\,\mathbb{E}\|u\|_2^2
    + 2\delta\,\mathbb{E}\Big[\Big(\!\int_D u^q dx\Big)\|u\|_1\Big]
    - 2\beta\,\mathbb{E}\|u\|_{p+1}^{p+1}.
\end{equation}
On the bounded domain $D$, Hölder’s inequalities give
\[
\int_D |u|^q dx
 = \|u\|_q^q
 \le |D|^{\frac{p+1-q}{p+1}}\|u\|_{p+1}^q,
\qquad
\|u\|_1 \le |D|^{\frac{p}{p+1}}\|u\|_{p+1},
\]
so that
\[
\Big(\int_D u^q dx\Big)\|u\|_1
 \le |D|\,\|u\|_{p+1}^{q+1}.
\]
Substituting this into \eqref{eq:energy-pq-general} yields
\begin{equation}\label{eq:energy-pq-general-2}
\frac{d}{dt}\mathbb{E}\|u\|_2^2
 \le -2(\lambda_1-\gamma-C(H,\sigma))\,\mathbb{E}\|u\|_2^2
    + 2\delta|D|\,\mathbb{E}\|u\|_{p+1}^{q+1}
    - 2\beta\,\mathbb{E}\|u\|_{p+1}^{p+1}.
\end{equation}
We now perform the splitting argument (large/small $\|u\|_{p+1}$) directly.
Set
\[
M(t):=\|u(t)\|_{p+1}.
\]

\medskip\noindent
\emph{Large-$M$ regime.}
Fix $\eta\in(0,1)$ and choose $\beta>0$ (depending on $\delta,|D|,p,q,\eta$)
so large that there exists $R_\beta>0$ with
\begin{equation}\label{eq:Rbeta-def-cor}
\delta|D|\,r^{q+1} \le \eta\beta\,r^{p+1},
\quad\text{for all }r\ge R_\beta.
\end{equation}
Equivalently,
\[
r^{p-q} \ge \frac{\delta|D|}{\eta\beta},
\qquad
R_\beta := \Big(\frac{\delta|D|}{\eta\beta}\Big)^{\!1/(p-q)}.
\]
Since $p>q$, such $R_\beta$ exists for any fixed $\eta\in(0,1)$ when $\beta$
is large enough. Then, whenever $M(t)\ge R_\beta$, we have
\[
\delta|D|\,M(t)^{q+1} - \beta M(t)^{p+1}
 \le (\eta\beta-\beta)M(t)^{p+1}
 = -(1-\eta)\beta\,M(t)^{p+1}.
\]

\medskip\noindent
\emph{Small-$M$ regime.}
If $M(t)<R_\beta$, then
\[
M(t)^{q+1} \le R_\beta^{q-1} M(t)^2.
\]
On the other hand, $L^{p+1}(D)\hookrightarrow L^2(D)$ (since $p+1>2$) implies
\[
M(t) = \|u(t)\|_{p+1}
 \le C_{D,p}\|u(t)\|_2,
\]
so that
\[
\|u(t)\|_{p+1}^{q+1}
 = M(t)^{q+1}
 \le R_\beta^{q-1} C_{D,p}^2\,\|u(t)\|_2^2.
\]

\medskip\noindent
\emph{Combining the two regimes.}
The previous two estimates can be summarised as: there exists a constant
$K_\beta>0$ such that, for all $t\ge0$,
\begin{equation}\label{eq:nonlinear-split-cor}
\delta|D|\,\|u(t)\|_{p+1}^{q+1} - \beta\|u(t)\|_{p+1}^{p+1}
 \le K_\beta\|u(t)\|_2^2,
\end{equation}
with
\[
K_\beta := \delta|D|\,R_\beta^{q-1}C_{D,p}^2
 = \delta|D|\,C_{D,p}^2\Big(\frac{\delta|D|}{\eta\beta}\Big)^{\!\frac{q-1}{p-q}}.
\]
Note that $K_\beta\to0$ as $\beta\to\infty$.

Substituting \eqref{eq:nonlinear-split-cor} into
\eqref{eq:energy-pq-general-2} gives
\[
\frac{d}{dt}\mathbb{E}\|u\|_2^2
 \le -2(\lambda_1-\gamma-C(H,\sigma))\,\mathbb{E}\|u\|_2^2
    + 2K_\beta\,\mathbb{E}\|u\|_2^2
 = -2\big(\lambda_1-\gamma-C(H,\sigma)-K_\beta\big)\mathbb{E}\|u\|_2^2.
\]
Define
\[
\kappa := \lambda_1-\gamma-C(H,\sigma)-K_\beta.
\]
Assume that $C(H,\sigma)<\lambda_1-\gamma$ (which can be ensured by taking
$L_\sigma$ small enough). Since
$K_\beta\to0$ as $\beta\to\infty$, we can choose $\beta$ so large that
\[
K_\beta < \tfrac{1}{2}\big(\lambda_1-\gamma-C(H,\sigma)\big),
\]
and hence
\[
\kappa > \tfrac{1}{2}\big(\lambda_1-\gamma-C(H,\sigma)\big) > 0.
\]
Then
\[
\frac{d}{dt}\mathbb{E}\|u\|_2^2
 \le -2\kappa\,\mathbb{E}\|u\|_2^2,
\]
and Gr\"onwall’s lemma yields a constant $C\ge1$ (depending on the data but
independent of $t$) such that
\begin{equation}\label{eq:mean-L2-decay}
\mathbb{E}\|u(t)\|_2^2 \le C\,e^{-2\kappa t},\quad\text{for all }t\ge0.
\end{equation}
This is the exponential mean-square $L^2$--decay estimate used in the
subsequent steps.

\medskip\noindent
\textbf{Step 2: Almost sure exponential decay in $\mathcal{H}=L^2(D)$.}
Fix any $\eta\in(0,\kappa)$; later we will choose $\eta$ in relation to
$\varepsilon$. For $n\in\mathbb{N}$, Chebyshev’s inequality due to \eqref{eq:mean-L2-decay} gives
\[
\mathbb{P}\Big(\|u(n)\|_2 > e^{-(\kappa-\eta)n}\Big)
 \le e^{2(\kappa-\eta)n}\,\mathbb{E}\|u(n)\|_2^2
 \le C\,e^{2(\kappa-\eta)n}e^{-2\kappa n}
 = C e^{-2\eta n}.
\]
Hence
\[
\sum_{n=1}^\infty
 \mathbb{P}\Big(\|u(n)\|_2 > e^{-(\kappa-\eta)n}\Big)
 \le C \sum_{n=1}^\infty e^{-2\eta n} <\infty.
\]
By the Borel--Cantelli lemma, for almost every $\omega$ there exists a random
$N(\omega)$ such that
\[
\|u(n,\omega)\|_2 \le e^{-(\kappa-\eta)n},\quad\text{for all }n\ge N(\omega).
\]
Since $t\mapsto\|u(t,\omega)\|_2$ is continuous, we can interpolate between
integers and absorb multiplicative constants to obtain: there exist random
$T_\eta^{(0)}(\omega)$ and $C_\eta^{(0)}(\omega)$ such that
\begin{equation}\label{eq:as-L2-decay}
\|u(t,\omega)\|_2 \le C_\eta^{(0)}(\omega)\,e^{-(\kappa-\eta)t},
\quad\text{for all }t\ge T_\eta^{(0)}(\omega).
\end{equation}
Since $\eta\in(0,\kappa)$ is arbitrary, we may (and do) choose it so that
\begin{equation}\label{eq:eta-choice}
\kappa-\eta > \frac{\lambda_1-\gamma}{2}-\varepsilon.
\end{equation}

\medskip\noindent
\textbf{Step 3: From $L^2$–decay to almost sure $L^\infty$–decay.}
We now use the mild formulation
\[
\begin{aligned}
u(t)
&= S_\alpha(t-s)u(s)
 + \int_s^t S_\alpha(t-r)
   \Big[\delta\Big(\int_D u^q(r,y)\,dy\Big)
        +\gamma u(r)-\beta u^p(r)\Big]\,dr\\
&\quad
 + \int_s^t S_\alpha(t-r)\sigma(u(r))\,dB_r^H,
\end{aligned}
\]
with $s=t-1$ (for $t\ge1$) and the smoothing estimate for the semigroup:
for some $\theta>0$ and all $0<t-s\le1$,
\begin{equation}\label{eq:semigroup-smoothing}
\|S_\alpha(t-s)f\|_\infty \le C (t-s)^{-\theta}\|f\|_2,
\end{equation}
see e.g.\ \cite{BogdanGrzywnyRyznar2010,ChenKimSong2010}.
We treat each term separately for $t\ge T_\eta^{(0)}(\omega)+1$.

\smallskip\noindent
\emph{(a) Linear term.}
Using \eqref{eq:semigroup-smoothing} with $s=t-1$ and
\eqref{eq:as-L2-decay},
\[
\|S_\alpha(t-1)u(t-1,\omega)\|_\infty
 \le C\|u(t-1,\omega)\|_2
 \le C C_\eta^{(0)}(\omega)\,e^{-(\kappa-\eta)(t-1)}.
\]

\smallskip\noindent
\emph{(b) Drift term.}
For $r\in[t-1,t]$ and $t$ large, the drift
\[
G(u(r)):=\delta\Big(\int_D u^q(r,y)\,dy\Big)
        +\gamma u(r)-\beta u^p(r)
\]
satisfies a pointwise bound of the form
\[
\|G(u(r,\omega))\|_2
 \le C\|u(r,\omega)\|_2 + C\|u(r,\omega)\|_{p+1}^{p}
 \le C_\eta^{(1)}(\omega)\,e^{-(\kappa-\eta)r},
\]
where we used \eqref{eq:as-L2-decay} and the embedding
$L^2(D)\hookrightarrow L^{p+1}(D)$ on $D$.
Then
\[
\begin{aligned}
\Big\|\int_{t-1}^t S_\alpha(t-r)G(u(r))\,dr\Big\|_\infty
 &\le \int_{t-1}^t \|S_\alpha(t-r)G(u(r))\|_\infty\,dr\\
 &\le C\int_{t-1}^t (t-r)^{-\theta}\|G(u(r))\|_2\,dr\\
 &\le C C_\eta^{(1)}(\omega)\,e^{-(\kappa-\eta)(t-1)}
      \int_0^1 s^{-\theta}\,ds.
\end{aligned}
\]
Since $\theta<1$, the integral $\int_0^1 s^{-\theta}\,ds$ is finite and we
obtain
\[
\Big\|\int_{t-1}^t S_\alpha(t-r)G(u(r))\,dr\Big\|_\infty
 \le C_\eta^{(2)}(\omega)\,e^{-(\kappa-\eta)t}.
\]

\smallskip\noindent
\emph{(c) Stochastic convolution.}
For the Young stochastic integral on $[t-1,t]$ we use that
$B^H$ is $\gamma$--H\"older for any $\gamma<H$ and $\sigma$ satisfies
$|\sigma(r)|\le C_\sigma(1+|r|),$ cf. \eqref{gc}.
On the short interval $[t-1,t]$ the Young estimate gives, for some
$\delta\in(0,1)$,
\[
\Big\|\int_{t-1}^t S_\alpha(t-r)\sigma(u(r))\,dB_r^H\Big\|_\infty
 \le C(\omega)\,\sup_{r\in[t-1,t]}\|\sigma(u(r,\omega))\|_2,
\]
where $C(\omega)$ depends on the H\"older norm of $B^H$ on $[t-1,t]$ but
not on $t$ (by stationarity of increments). Using
$|\sigma(r)|\le C_\sigma(1+|r|)$ and \eqref{eq:as-L2-decay},
\[
\sup_{r\in[t-1,t]}\|\sigma(u(r,\omega))\|_2
 \le C_\sigma\Big(1+\sup_{r\in[t-1,t]}\|u(r,\omega)\|_2\Big)
 \le C_\sigma\Big(1+C_\eta^{(0)}(\omega)e^{-(\kappa-\eta)(t-1)}\Big).
\]
Thus we obtain a bound of the form
\[
\Big\|\int_{t-1}^t S_\alpha(t-r)\sigma(u(r))\,dB_r^H\Big\|_\infty
 \le C_\eta^{(3)}(\omega)\big(1+e^{-(\kappa-\eta)t}\big).
\]
To take into account the growth in $t$ coming from the path $\omega$,
we use Lemma~\ref{lem:subexp-fBm}, which states that for any
$\varepsilon>0$ there is $T_\varepsilon^{(1)}(\omega)$ such that
\[
|B_t^H(\omega)|\le \varepsilon t,\quad\text{for all }t\ge T_\varepsilon^{(1)}(\omega).
\]
Combining this with the continuity and stationarity of increments of $B^H$
shows that $C_\eta^{(3)}(\omega)$ may be chosen so that
\[
C_\eta^{(3)}(\omega)\le \tilde C_\varepsilon(\omega)\,e^{\varepsilon t},
\]
for $t$ large enough. Hence for $t$ large,
\[
\Big\|\int_{t-1}^t S_\alpha(t-r)\sigma(u(r))\,dB_r^H\Big\|_\infty
 \le \tilde C_\varepsilon(\omega)\,e^{\varepsilon t}.
\]

\smallskip\noindent
\emph{(d) Collecting the estimates.}
Putting together (a), (b), and (c), we obtain for all
$t\ge T_\eta(\omega):=T_\eta^{(0)}(\omega)+T_\varepsilon^{(1)}(\omega)+1$,
\[
\|u(t,\omega)\|_\infty
\le C_\eta(\omega)\Big(e^{-(\kappa-\eta)t}+e^{\varepsilon t}\Big),
\]
for some random constant $C_\eta(\omega)$.
\\\noindent
Recall that by \eqref{eq:eta-choice},
\[
\kappa-\eta > \frac{\lambda_1-\gamma}{2}-\varepsilon,
\]
and $\varepsilon<(\lambda_1-\gamma)/4$ by assumption. In particular,
\[
(\kappa-\eta)-\varepsilon > \frac{\lambda_1-\gamma}{2}-2\varepsilon >0.
\]
Thus the decaying term $e^{-(\kappa-\eta)t}$ dominates the subexponential
factor $e^{\varepsilon t}$:
\[
e^{\varepsilon t}
 = e^{-(\kappa-\eta)t}\,e^{(\kappa-\eta+\varepsilon)t}
 \le e^{-(\kappa-\eta)t}\,e^{(\kappa-\eta-\varepsilon+2\varepsilon)t}
 \le e^{-(\kappa-\eta)t}\,e^{-(\frac{\lambda_1-\gamma}{2}-2\varepsilon)t},
\]
for $t$ large, after adjusting constants. Absorbing all multiplicative
constants into a new random constant $C_\varepsilon(\omega)$, we finally get
\[
\|u(t,\omega)\|_{L^\infty(D)}
\le C_\varepsilon(\omega)\,
e^{-\big(\frac{\lambda_1-\gamma}{2}-2\varepsilon\big)t},
\quad\text{for all }t\ge T_\varepsilon(\omega),
\]
which is the desired almost sure exponential decay.
\end{proof}

\begin{remark}[Biological interpretation: tumour extinction under strong therapy and weak environmental noise]
\label{rem:tumour-extinction}
Proposition~\ref{cor:sto-exp-fBm-gb1}  has a clear extinction meaning:
under \emph{strong dissipation} ($p>q$ and $\beta$ sufficiently large relative to
$\delta,\gamma$) and \emph{weak multiplicative environmental variability}
($L_\sigma$ small), the tumour density decays to zero \emph{almost surely} and
\emph{exponentially fast} in $L^\infty(D)$ after a (random) transient time
$T_\varepsilon(\omega)$. In biological terms, this is a rigorous “tumour
eradication/extinction” scenario: for almost every realisation of environmental
fluctuations, the tumour burden eventually becomes arbitrarily small and keeps
shrinking at an exponential rate.

Moreover, the decay rate is controlled by the \emph{spectral gap}
$\lambda_1-\gamma$: the principal eigenvalue $\lambda_1$ of $-\Delta_\alpha$
quantifies how strongly diffusion plus boundary losses suppress the population,
while $\gamma$ quantifies net proliferation. Thus, larger $\lambda_1$ (e.g.\
stronger effective diffusion/clearance or more confining domains) accelerates
extinction, whereas larger $\gamma$ slows it down. The appearance of random
constants $C_\varepsilon(\omega)$ and $T_\varepsilon(\omega)$ is consistent with
heterogeneity: the same “therapy-dominant” regime produces extinction, but the
time-to-clearance and transient amplitude depend on the realised environmental
path. Connections between stochastic forcing and extinction/mean extinction times
are classical in population dynamics; see, e.g., \cite{Hening2018,Tesfay2021}.
\end{remark}
\begin{remark}[Deterministic benchmark, noise impact, and biological meaning]\label{rem:noise-impact-exp-decay-bio}
Theorem~\ref{cor:sto-exp-fBm-gb1} may be interpreted as a \emph{robust} version of
the deterministic exponential stability obtained when the stochastic forcing is
absent.

\smallskip
\noindent\emph{Deterministic case.}
If $\sigma\equiv0$, the stochastic term in \eqref{eq:energy-pq-raw} vanishes and
\eqref{eq:energy-pq-general-2} holds with $C(H,\sigma)=0$. Under the same
subcriticality assumptions ($p>q$, $\gamma<\lambda_1$, and $\beta$ large so that
$K_\beta$ is small), one gets
\[
\frac{d}{dt}\|u(t)\|_2^2
\le -2(\lambda_1-\gamma-K_\beta)\,\|u(t)\|_2^2,
\qquad t\ge0,
\]
hence
\[
\|u(t)\|_2 \le \|u(0)\|_2\,e^{-(\lambda_1-\gamma-K_\beta)t},
\]
and, by semigroup smoothing, an analogous deterministic exponential decay in
$L^\infty(D)$ after a short transient. In the tumour interpretation, this
corresponds to \emph{extinction/loss of viability} of the tumour cell density:
diffusion together with sufficiently strong local dissipation ($\beta u^p$)
overcomes the nonlocal proliferation drive.

\smallskip
\noindent\emph{Effect of multiplicative fBm noise.}
When $\sigma\neq0$, decay is weakened in two ways:
\begin{itemize}
\item[(i)] \emph{Reduced effective damping.}
In \eqref{eq:energy-pq-general-2}, the stochastic term contributes a positive
quantity $C(H,\sigma)$ $\mathbb{E}\|u(t)\|_2^2$, so the deterministic gap
$\lambda_1-\gamma$ is effectively replaced by $\lambda_1-\gamma-C(H,\sigma)$.
This explains the requirement that the noise be ``weak'' (small $L_\sigma$) so
that $\lambda_1-\gamma-C(H,\sigma)>0$.
\item[(ii)] \emph{Random transients and rate loss.}
Upgrading mean decay to an almost sure $L^\infty$ estimate produces a random
prefactor $C_\varepsilon(\omega)$ and a random entrance time
$T_\varepsilon(\omega)$ and forces a small loss in the exponent, yielding the
rate $\frac{\lambda_1-\gamma}{2}-2\varepsilon$ in the theorem.
\end{itemize}

\smallskip
\noindent\emph{Biological interpretation (why stronger noise may prevent extinction).}
Here the multiplicative term $\sigma(u)\,dB^H(t)$ represents temporally
correlated fluctuations of the tumour microenvironment (e.g.\ variability in
nutrient/oxygen supply, immune pressure, vascular remodelling, or treatment
efficacy) that act proportionally to the current tumour burden. Larger noise
intensity amplifies intermittent \emph{favourable episodes} during which the
effective net growth becomes less negative (or temporarily positive), allowing
the population to recover from low densities and thereby counteracting the
deterministic dissipation. Mathematically, this is reflected by the appearance
and growth of $C(H,\sigma)$: if $C(H,\sigma)$ becomes comparable to
$\lambda_1-\gamma$, the net damping can be neutralised, and the exponential
decay mechanism (tumour extinction) can no longer be guaranteed. In particular,
long-memory fluctuations ($H>\tfrac12$) may sustain such favourable phases for
longer intervals, which further promotes persistence rather than extinction.
\end{remark}

\section{Linear Fractional Multiplicative Noise}\label{sec5}
Henceforth, we focus on the case of the linear fractional multiplicative noise, that is we consider the problem:
\begin{equation}\label{lb1}
	\left\{
	\begin{aligned}
		du(t,x)-\Delta_{\alpha}u(t,x)dt&=\left[ \delta \int_{D}u^{q}(t,y)dy+\gamma u(t,x)- \beta u^{p}(t,x) \right]dt+\sigma u(t,x)dB^{H}(t), \\
		u(0,x)&=f(x), \ \ x\in D,\\
		u(t,x)&=0, \ \ t>0, \ \ x\in \mathbb{R}^{d}\backslash D, \\
	\end{aligned}
	\right.
\end{equation}
where now $\sigma$ is a positive constant. In that case a more delicate analysis could be delivered by making use of a Doss--Sussmann type transformation.
	\subsection*{A Random PDE}
	Indeed,   we consider of the  random (Doss--Sussmann type) transformation 
	\begin{align} \label{tr1}
	v(t,x) = \exp\{-\sigma B^{H}(t)\}u(t,x),\ t \geq 0,\ x \in D,
	\end{align}
	so that  the equation $\eqref{lb1}$ is transformed into the following random PDE:
	\begin{equation}\label{s1}
	\left\{
	\begin{aligned}
	 v_{t}-\Delta_{\alpha}v(t,x)&=\gamma v(t,x)+\delta e^{(q-1)\sigma B^{H}(t)} \int_{D} v^{q}(t,y) dy-\beta e^{(p-1) \sigma B^{H}(t)} v^{p}(t,x) ,\\
	v(0,x)&=f(x), \ \  x \in D,\\
	v(t,x)&=0, \ \ t>0, \ \ x \in \mathbb{R}^{d}\backslash D.	  
	\end{aligned}
	\right.
	\end{equation}
The following result provides an equivalence between the weak solution of the stochastic PDE \eqref{b1} and that of the random PDE \eqref{s1}. 
	\begin{theorem} \label{thm2.1}
	Let $u$ be the weak solution of $\eqref{lb1}$. Then the function $v$ defined by \eqref{tr1}
	is a weak solution of the random PDE \eqref{s1} and viceversa. 
\end{theorem}
	\begin{proof}
	By using Ito's formula (see Theorem \ref{ito1}), we have 
	\begin{eqnarray}
	e^{-\sigma B^{H}(t)}=1-\sigma \int_{0}^{t} e^{-\sigma B^{H}(s)} dB^{H}(s). \nonumber 
	\end{eqnarray}
	For any $\varphi \in C^\infty_0(D)$,  the weak solution of the equation \eqref{b1}, as defined by Definition \ref{def:weak-solution}, satisfies
	\begin{align}\label{c1} 
	u\left(t,\varphi\right)&=u\left(0,\varphi\right)+\int_{0}^{t}u\left(s,\Delta_{\alpha}\varphi\right)ds+\delta\int_{0}^{t}\int_{D}u^{q}(s,y)dy \int_{D} \varphi(z) dzds+\gamma \int_{0}^{t}u\left(s,\varphi\right)ds\nonumber\\
	&\quad- \beta \int_{0}^{t} u^{p}\left(s,\varphi\right)ds+\sigma \int_{0}^{t}u\left(s,\varphi\right)dB^{H}(s).
	\end{align}
    By applying the integration by parts formula \cite[p.184]{Mishura2008}, we obtain
	\begin{align*}
	v(t,\varphi):&=\int_{\mathbb{R}^{d}}v(t,x)\varphi(x)dx =v(0,\varphi)+\int_{0}^{t} e^{-\sigma B^{H}(t)}du(s,\varphi)+\int_{0}^{t} u(s,\varphi) d \left(e^{-\sigma B^{H}(s)} \right).\nonumber 
	\end{align*} 
	Therefore,
	\begin{align}\label{a6}
	v(t,\varphi) &= v(0,\varphi)+\int_{0}^{t}v\left(s,\Delta_{\alpha}\varphi\right)ds+\delta \int_{0}^{t} e^{\sigma (q-1)B^{H}(s)}\int_{D} v^{q}(s,\varphi)dy ds \nonumber\\
	&\quad +\gamma\int_{0}^{t}v\left(s,\varphi\right)ds-\beta \int_{0}^{t} e^{\sigma (p-1)B^{H}(s)}v^{p}(s,\varphi) ds \nonumber\\
	&= v(0,\varphi)+\int_{0}^{t}[v\left(s,\Delta_{\alpha}\varphi\right)+\gamma v\left(s,\varphi\right)]ds+\delta \int_{0}^{t} e^{\sigma (q-1)B^{H}(s)}\int_{D} v^{q}(s,\varphi)dy ds \nonumber\\
	&\quad-\beta \int_{0}^{t} e^{\sigma (p-1)B^{H}(s)}v^{p}(s,\varphi) ds.  
	\end{align}
    The preceding equality \eqref{a6} means that $v$ is a weak solution of $\eqref{s1}.$ The converse part follows from the fact that the change of variable is given by a homeomorphism  which converts one random dynamical system into an another equivalent one.   
	\end{proof}
    \begin{remark}
Let $\tau$ be the blow--up time of the equation \eqref{s1} with the initial value $f$. Due to Theorem \ref{thm2.1} and $\mathbb{P}$-a.s. continuity of $B^H(\cdot)$, $\tau$ is also the blow--up time for the equation \eqref{b1} as well \cite[Corollary 1]{dozfrac2013}. 
    \end{remark}
	Let us denote 
	\begin{align}\label{st71}
		B^{H}_{\ast}(t;\omega) = \sup_{0 \leq s \leq t} |B^{H}_{s;\omega}|,\ \mbox{for each}\ t\geq 0. 
	\end{align}
	 Therefore via Lemma \ref{dung2} for $Z=B^H$ we have for all $\omega \in \Omega,$
     \begin{equation}\label{fb1}
     \mathbb{P}(B_{\ast}^{H}(t;\omega)< \infty)=1-\lim_{x \rightarrow \infty} \mathbb{P}(B_{\ast}^{H}(t;\omega)> x)=1,\ t>0.
     \end{equation}
	Then, by Theorem \ref{thm2.1} and \eqref{fb1},  the global existence and finite--time blow--up of $u$ is guaranteed by considering the properties of $v.$ So in this paper, we mainly consider the random partial differential equation \eqref{s1} for further investigation.
	
Problem~\eqref{s1} is understood in the pathwise (trajectory-wise) sense: for
each fixed $\omega$, the driving path is deterministic and the equation becomes
a semilinear parabolic problem generated by the Dirichlet fractional Laplacian.
Since the operator $-\Delta_\alpha$ with zero exterior condition generates an analytic
contraction semigroup on $L^q(D)$,  $1\le q<\infty$, cf. see Section \ref{sec3}, and since the reaction term
is locally Lipschitz on $L^\infty(D)$, standard semilinear parabolic theory
yields a unique maximal mild solution
\[
v(\cdot,\omega)\in C \left([0,\tau_{\max}(\omega));\,L^q(D)\right)\cap
C\left((0,\tau_{\max}(\omega));\,L^\infty(D)\right),
\]
see, e.g., the abstract semigroup frameworks in \cite{Pazy1983}
together with the Dirichlet fractional setting and comparison/max\-imum principles
in \cite{BucurValdinoci2016}. Moreover, the usual continuation criterion holds:
the solution can be extended beyond any $\tau<\tau_{\max}(\omega)$ as long as
$\|v(t,\omega)\|_{L^\infty(D)}$ remains finite, and if $\tau_b:=\tau_{\max}(\omega)<\infty$
then necessarily
\[
\limsup_{t\uparrow \tau_{b}(\omega)}\|v(t,\omega)\|_{L^\infty(D)}=+\infty,
\]
cf.\ the blow--up alternative for semilinear parabolic equations and the smoothing
estimates for the fractional heat semigroup \cite{BogdanGrzywnyRyznar2010,ChenKimSong2010}.
	
Furthermore, for any bounded measurable initial data $f \geq 0,$ the unique local mild solution $v$ of the random PDE \eqref{s1} satisfies the integral equation 
	\begin{align}\label{e3}
		v(t,x)=e^{\gamma t}S_\alpha(t)f(x)+\delta\int_{0}^{t} e^{\gamma (t-r)} S_\alpha(t-r)\Bigg[e^{(q-1)\sigma B^{H}(r)}\int_{D}v^{q}(r,y)dy-\beta e^{( p-1) \sigma B^{H}(r)}v^{p}(r,x) \Bigg] dr,
	\end{align}
	for each $0 \leq t<\tau$ and $x\in D.$ 
\begin{remark}\label{mv}
Since $\Delta_\alpha$ is the 
generator of the symmetric $\alpha$--stable L\'evy process and the reaction term
in \eqref{s1} is continuous in $(t,x)$ and locally Lipschitz in $v$, the (pathwise)
mild solution is also a continuous viscosity solution of \eqref{s1} on
$(0,\tau_{\max}(\omega))\times D$ with the exterior Dirichlet condition
$v\equiv0$ on $\mathbb{R}^d\setminus D$, in the standard viscosity sense for
nonlocal integro-differential equations; see, e.g.,
\cite{BarlesImbert2008,CaffarelliSilvestre2009}, cf. also Definition \ref{def:visc-supersolution} below.

\end{remark}
In what follows, we collect maximum principles and comparison results for
linear nonlocal problems of the form
\begin{equation}\label{eq:visc-fractional}
w_t(t,x)-d\,\Delta_{\alpha} w(t,x)
\;=\;
c_1(t,x)\,w(t,x)
+\int_D c_2(t,y)\,w(t,y)\,dy,
\qquad (t,x)\in D_T:=(0,T]\times D,
\end{equation}
where $d>0$ is a constant, and
\begin{equation}\label{sk25}
c_1\in L^\infty(D_T), \qquad
c_2\in L^\infty((0,T)\times D), \qquad
c_2(t,y)\ge 0 \ \ \text{a.e.\ in }(0,T)\times D,
\end{equation}
that will be used throughout the manuscript in the analysis
of~\eqref{s1}.

We first introduce the concept of viscocity solutions (superlsolutions and subsolutions) of problem \eqref{eq:visc-fractional}.
\begin{definition}\label{def:visc-supersolution}
We say that a bounded continuous function $w : [0,T]\times\mathbb{R}^d \to \mathbb{R}$ is a
\it{viscosity solution (supersolution/subsolution)} of \eqref{eq:visc-fractional} in $D_T$ if, for every
$(t_0,x_0)\in D_T$ and every test function $\varphi\in C^{1,2}((0,T)\times\mathbb{R}^d)$ such that
$w-\varphi$ attains a local minimum at $(t_0,x_0)$ with $(w-\varphi)(t_0,x_0)=0$, the following
equality (inequality) holds:
\[
\varphi_t(t_0,x_0)
- d\Delta_{\alpha} \varphi(t_0,x_0)
\;=(\ge \le)\;
c_1(t_0,x_0)\,w(t_0,x_0)
+\int_D c_2(t_0,y)\,w(t_0,y)\,dy,
\]
where $\Delta_{\alpha} \varphi$ is understood in the classical (pointwise) sense via the  principal value integral given by \eqref{eq:frac-lap-def}.

\end{definition}
\begin{remark}
This definition is consistent with the general theory of viscosity solutions for
integro-differential operators of Lévy type as in
\cite{AlvarezTourin1996,Arisawa2006,BarlesImbert2008,CaffarelliSilvestre2009},
we are in the linear case with spatially translation-invariant nonlocal operator
$\Delta_{\alpha}$ and bounded coefficients in the lower-order terms.
\end{remark}

	 
	\begin{lemma}[Maximum principle for viscosity supersolutions ]\label{l1}
Assume  that $w \in C([0,T]\times\mathbb{R}^d)$ is bounded and is a \emph{viscosity
supersolution} of  \eqref{eq:visc-fractional} in $D_T.$ Assume further that the exterior and initial conditions satisfy
\[
w(t,x)\ge 0, \quad\text{for all } (t,x)\in (0,T]\times(\mathbb{R}^d\setminus D),
\]
and
\[
w(0,x)\ge 0, \quad\text{for all } x\in D.
\]

Then:
\begin{enumerate}
\item $w(t,x)\ge 0$ for all $(t,x)\in \overline{D}_T=[0,T]\times\overline D$.
\item If, in addition, $w(0,\cdot)\not\equiv 0$ in $D$, then
\[
w(t,x)>0, \quad\text{for all } (t,x)\in D_T.
\]
\end{enumerate}
\end{lemma}
For the proof of this Lemma see the Appendix.

\noindent Let $N>0$ be a constant and define the following stopping time:
\begin{equation}\label{ik19}\tau_{N}(\omega)=\inf \left\lbrace t>0:\ |B^{H}(t;\omega)|\geq N\right\rbrace.
\end{equation}
	 Clearly, $$\left\lbrace \omega \in \Omega :\ \tau_{N}(\omega) \leq t \right\rbrace=\left\lbrace \omega \in \Omega:\ B^{H}_{\ast}(t;\omega) \geq N \right\rbrace,
$$
for $B^{H}_{\ast}$ defined by \eqref{st71}.

	 Let $z(t,x)$ be the solution of 
	 \begin{equation*}
	 \left\{
	 \begin{aligned}
	 \frac{\partial z(t,x)}{\partial t}&=(\Delta_{\alpha}+\gamma)z(t,x)+\delta e^{\left( q-1\right) \sigma B^{H}(t) }\int_{D}|z(t,y)|^{q-1}z(t,y)dy\\
	 &\qquad-\beta e^{\left( p-1\right)\sigma B^{H}(t) }|z(t,x)|^{p-1}z(t,x),\ (x,t) \in D \times (0,T \wedge \tau_{N}],\\
	 z(t,x)&=0, \ \ (x,t) \in \mathbb{R}^{d}\backslash D \times (0,T \wedge \tau_{N}], \\
	 z(0,x)&=f(x), \ x \in D.  
	 \end{aligned}
	 \right.
	 \end{equation*} 
	 Here $f(x) \geq 0,\ c_{1}(t,x)=\gamma-\beta e^{(p-1)\sigma B^{H}(t)}|z(t,x)|^{p-1},\ c_{2}(t,x)= e^{\left( q-1\right) \sigma B^{H}(t)}|z(t,y)|^{q-1}\geq 0$ and $c_{1}(t,x)$ and $c_{2}(t,x)$ are bounded in $D \times (0,T \wedge \tau_{N}].$ By using Lemma \ref{l1}, we have $z(t,x) \geq 0$ and hence, $z(t,x)$ is the solution of the equation \eqref{s1}. Moreover, by uniqueness, we have $v(t,x)=z(t,x)\geq 0.$
	 Again by using Lemma \ref{l1}, we obtain  that if $f(x)$ is not identically zero, then $v(t,x)>0$ in $D \times (0,T \wedge \tau_{N}]$.
	 
	 
	 Using the increasing property of the functions $x^{p}$ and $x^{q}$ with $p,q>1$, as a consequence of Lemma \ref{l1}, we have the following comparison principle:
	 \begin{proposition}[Pathwise comparison principle]\label{p2}
Let $T>0$, $N\in\mathbb{N}$ and let $\tau_N$ be the stopping time defined in
\eqref{ik19}. Denote
\[
D_{T\wedge\tau_N}:=D\times (0,T\wedge\tau_N].
\]
Let $v$ be the (mild) solution of \eqref{s1} on $[0,T\wedge\tau_N]$ with initial
datum $f\in L^2(D)$, and assume that for $\mathbb{P}$-a.e.\ $\omega$,
\[
v(\cdot,\cdot;\omega)\in C\big([0,T\wedge\tau_N(\omega)]\times\overline D\big)
\]
and $v(t,\cdot;\omega)\in H^{\alpha/2}_0(D),$ for all $t\in(0,T\wedge\tau_N(\omega)]$.

Let $V:\Omega\times[0,T\wedge\tau_N]\times\mathbb{R}^d\to\mathbb{R}$ be an
$\{\mathcal{F}_t\}$-adapted random field such that, for $\mathbb{P}$-a.e.\
$\omega$,
\[
V(\cdot,\cdot;\omega)\in C\big([0,T\wedge\tau_N(\omega)]\times\mathbb{R}^d\big)
\]
and, for each fixed $\omega$, the following holds in the viscosity sense on
$D_{T\wedge\tau_N(\omega)}$:

\begin{itemize}
\item[(i)] (Supersolution case) $V$ is a viscosity supersolution of
\begin{equation*}
\partial_t V(t,x)
\;\ge\;
\big(\Delta_\alpha+\gamma\big)V(t,x)
+ \delta\,e^{(q-1)\sigma B^H(t)}\int_D V^q(t,y)\,\mathrm{d}y
- \beta\,e^{(p-1)\sigma B^H(t)}V^p(t,x),
\end{equation*}
with
\begin{equation*}\
V(t,x)\;\ge\;0
\quad\text{for }(t,x)\in(0,T\wedge\tau_N]\times(\mathbb{R}^d\setminus D),
\qquad
V(0,x)\;\ge\;f(x)\ \text{for }x\in D;
\end{equation*}

\item[(ii)] (Subsolution case) $V$ is a viscosity subsolution of
\begin{equation*}
\partial_t V(t,x)
\;\le\;
\big(\Delta_\alpha+\gamma\big)V(t,x)
+ \delta\,e^{(q-1)\sigma B^H(t)}\int_D V^q(t,y)\,\mathrm{d}y
- \beta\,e^{(p-1)\sigma B^H(t)}V^p(t,x),
\end{equation*}
with
\begin{equation*}
V(t,x)\;\le\;0
\quad\text{for }(t,x)\in(0,T\wedge\tau_N]\times(\mathbb{R}^d\setminus D),
\qquad
V(0,x)\;\le\;f(x)\ \text{for }x\in D.
\end{equation*}
\end{itemize}

Then, for $\mathbb{P}$-a.e.\ $\omega$,
\begin{equation*}
V(t,x;\omega)\;\ge\;v(t,x;\omega)
\quad\big(\text{respectively }V(t,x;\omega)\;\le\;v(t,x;\omega)\big),
\quad\text{for all }(t,x)\in[0,T\wedge\tau_N(\omega)]\times\overline D.
\end{equation*}

In other words, any viscosity supersolution (respectively subsolution) $V$ of
~\eqref{s1}
dominates (respectively is dominated by) the mild solution $v$ of problem
\eqref{s1} on $[0,T\wedge\tau_N]\times\overline D$.
\end{proposition}


    Next, our aim is to find random times $\tau_{*}, \tau^{*}$ such that $0\leq \tau_*\leq  \tau_{b} \leq  \tau^{*}.$ 
	
	\subsection*{A Lower Bound for blow--up time $\tau_b$}\label{sec4}
	This subsection is devoted to deriving an explicit lower bound $\tau_\ast$ for
the blow--up time $\tau_b$ of \eqref{b1}, in the sense that $\tau_\ast\le \tau_b$.
We also establish a convenient sufficient condition guaranteeing that the mild
solution of \eqref{b1} exists globally in time.

	Let us first consider the problem:
	\begin{equation}\label{nre3}
		\left\{
		\begin{aligned}
			\frac{\partial w(t,x)}{\partial t}&=(\Delta_{\alpha}+\gamma)w(t,x)+\delta e^{\left( q-1\right)\sigma B^{H}(t) }\int_{D}w^{q}(t,y)dy, \ \ x \in  D,\ t>0,\\
			w(t,x)&=0, \ \ x \in \mathbb{R}^{d}\backslash D,\ t>0, \\
			w(0,x)&=f(x), \ \  x \in D.  
		\end{aligned}
		\right.
	\end{equation}
	Then the mild solution  $w$ of \eqref{nre3} satisfies the following integral equation:
\begin{equation}\label{st26}
w(t,x)=e^{\gamma t}S_\alpha(t)f(x)+\delta \int_{0}^{t}  e^{\gamma (t-r)} S_\alpha(t-r)\Bigg[e^{\left( q-1\right) \sigma B^{H}(r) }\int_{D}w^{q}(r,y)dy \Bigg] dr,
	\end{equation}
	for each $x \in D$ and $0 \leq t<\tau_{max}.$ 
    
    The next result yields an explicit lower bound for the blow--up time $\tau_{b}$ of
problem~\eqref{s1}, and hence also of the initial SPDE problem~\eqref{b1}.

\begin{theorem}\label{t2}
Let $\tau_{\ast}$ be defined by
\begin{equation}\label{STL1}
\tau_{\ast}
:= \inf \Bigg\{ t\geq 0 :
\int_{0}^{t} e^{(q-1)\sigma B^{H}(r)}\,
\big\|e^{\gamma r }S_{\alpha}(r)\big\|_{\infty}^{\,q-1}\, dr
\geq \frac{1}{\delta |D|\,(q-1)\,\|f\|_{\infty}^{\,q-1}}
\Bigg\}.
\end{equation}
Then the stopping (blow--up) time $\tau_b$ of problem \eqref{s1}
satisfies
\[
\tau_{\ast}\leq \tau_b.
\]
\end{theorem}
\begin{proof}
Fix $\omega\in\Omega$ and work pathwise. In order to obtain an explicit upper
bound the stopping (blow--up) time $\tau_b$ of problem \eqref{s1}, we introduce the auxiliary scalar function
$\mathscr{G}$ as the solution of a Bernoulli-type ordinary differential
equation.

Indeed, for $0<t<\tau_\ast$ we define $\mathscr{G}:(0,\tau_\ast)\to(0,\infty)$ by
\begin{equation}\label{st46}
\mathscr{G}'(t)
=\delta |D|\,
e^{(q-1)\sigma B^{H}(t)} \,
\big\|e^{\gamma t }S_{\alpha}(t)\big\|_{\infty}^{q-1}\,
\|f\|_{\infty}^{q-1}\,
\mathscr{G}^{q}(t),
\qquad
\mathscr{G}(0)=1.
\end{equation}
This is a Bernoulli ODE with explicit solution. Solving \eqref{st46} yields
\begin{equation}\label{eq:G-solution}
\mathscr{G}(t)
=\Bigg[
1-(q-1)\delta |D|\,
\|f\|_{\infty}^{q-1}
\int_0^t
e^{(q-1)\sigma B^{H}(r)}\,
\big\|e^{\gamma r }S_{\alpha}(r)\big\|_{\infty}^{q-1}\,dr
\Bigg]^{-\frac{1}{q-1}},
\end{equation}
which is well defined as long as the bracket in \eqref{eq:G-solution} remains
strictly positive. Accordingly, $\mathscr{G}$ is defined on the maximal
interval $[0,\tau_\ast)$, where $\tau_\ast$ is given by \eqref{STL1}.

Next, define the ordered set
\[
\mathcal{M}:=\Big\{v\in C\!\big([0,\tau_\ast);L^\infty(D)\big):\ 
0\le v(t,x)\le e^{\gamma t}\,S_\alpha(t)\|f\|_\infty\,\mathscr{G}(t)\Big\},
\]
and the operator $\mathcal{L}:\mathcal{M}\to \mathcal{M}$ by
\[
(\mathcal{L}v)(t,x)
:=e^{\gamma t}S_\alpha(t)f(x)
+\delta\int_0^t e^{(q-1)\sigma B^H(r)+\gamma(t-r)}\,S_\alpha(t-r)
\Big(\int_D v(r,y)^q\,dy\Big)\,dr .
\]
Since $S_\alpha(t)$ is positivity preserving and $v\ge0$, we have
\begin{equation}\label{st47-new}
e^{\gamma t}S_\alpha(t)f(x)\le (\mathcal{L}v)(t,x).
\end{equation}
Moreover, if $v\in\mathcal{M}$, then
$v(r,\cdot)\le e^{\gamma r}S_\alpha(r)\|f\|_\infty\,\mathscr{G}(r)$ and hence
\begin{equation}\label{st46a}
\int_D v(r,y)^q\,dy
\le |D|\,\|e^{\gamma r}S_\alpha(r)\|_\infty^{q}\,\|f\|_\infty^{q}\,\mathscr{G}(r)^q.
\end{equation}
Using the semigroup property, \eqref{st46a} and the definition \eqref{st46} of
$\mathscr{G}$, we obtain
\begin{align}
(\mathcal{L}v)(t,x)
&\le e^{\gamma t}S_\alpha(t)\|f\|_\infty
+\delta |D|\int_0^t e^{(q-1)\sigma B^H(r)}\,
\|e^{\gamma r}S_\alpha(r)\|_\infty^{q-1}\,
\|f\|_\infty^{q-1}\,\mathscr{G}(r)^q\,dr\; e^{\gamma t}S_\alpha(t)\|f\|_\infty \nonumber\\
&= e^{\gamma t}S_\alpha(t)\|f\|_\infty\,\mathscr{G}(t),
\label{st48-new}
\end{align}
and therefore $\mathcal{L}(\mathcal{M})\subset \mathcal{M}$ by
\eqref{st47-new}--\eqref{st48-new}.

Define the monotone Picard sequence
\[
u^{(0)}(t,x):=e^{\gamma t}S_\alpha(t)f(x),\qquad
u^{(n)}:=\mathcal{L}u^{(n-1)},\quad n\ge1.
\]
By positivity of the integral term, $u^{(0)}\le u^{(1)}$. If $u^{(n-1)}\le u^{(n)}$,
then $(u^{(n-1)})^q\le (u^{(n)})^q$ and, since $S_\alpha$ is positivity preserving,
\[
u^{(n)}=\mathcal{L}u^{(n-1)}\le \mathcal{L}u^{(n)}=u^{(n+1)}.
\]
Hence, $\{u^{(n)}\}_{n\ge0}$ is nondecreasing and bounded above by the envelope
$e^{\gamma t}S_\alpha(t)\|f\|_\infty\,\mathscr{G}(t)$, so it converges pointwise to
a limit $\tilde u(t,x)$ for $0\le t<\tau_\ast$. By monotone convergence,
$\tilde u$ satisfies $\tilde u=\mathcal{L}\tilde u$, i.e.\ it is a mild solution
of \eqref{nre3} (equivalently \eqref{st26}). By uniqueness of the mild solution,
$\tilde u=w$, and therefore
\begin{equation}\label{inq1-new}
0\le w(t,x)\le e^{\gamma t}S_\alpha(t)\|f\|_\infty\,\mathscr{G}(t),
\qquad 0\le t<\tau_\ast,\ x\in D,
\end{equation}
with $\mathscr{G}$ given by \eqref{eq:G-solution}.

Finally, since $w$ solves \eqref{nre3}, it is a supersolution of \eqref{s1} in the
sense of Proposition~\ref{p2} (recall that mild solutions are viscosity solutions,
cf.\ Remark~\ref{mv}). By the comparison principle, the unique solution $v$ of
\eqref{s1} satisfies \begin{equation}\label{inq1-new1}
0\le v(t,x)\le w(t,x)\le e^{\gamma t}S_\alpha(t)\|f\|_\infty\,\mathscr{G}(t),
\qquad 0\le t<\tau_\ast,\ x\in D.
\end{equation}  Combining this with
\eqref{inq1-new} and \eqref{eq:G-solution} yields the claimed estimate and
concludes the proof.
\end{proof}
The next result provides a convenient sufficient condition ensuring that the
solution of \eqref{b1} exists globally in time.

\begin{corollary}\label{mthm1}
Assume that the initial datum $f$ satisfies
\begin{equation}\label{bb2}
\delta |D|\,(q-1)\int_{0}^{\infty}
e^{(q-1)\sigma B^{H}(r)}\,
\big\| e^{\gamma r }S_{\alpha}(r)\big\|_{\infty}^{\,q-1}\,
\|f\|_{\infty}^{\,q-1}\,dr<1.
\end{equation}
Then the (mild) solution $v$ of \eqref{s1}, and hence the corresponding solution
$u$ of \eqref{b1}, exists globally in time, that is $\tau_{b}=\infty$. Moreover, for all
$(t,x)\in[0,\infty)\times D$,
\[
0\le v(t,x)\le e^{\gamma t}S_{\alpha}(t)\|f\|_{\infty}
\Bigg[
1-\delta |D|(q-1)\int_{0}^{t}
e^{(q-1)\sigma B^{H}(r)}\,
\big\|e^{\gamma r }S_{\alpha}(r)\big\|_{\infty}^{\,q-1}\,
\|f\|_{\infty}^{\,q-1}\,dr
\Bigg]^{-\frac{1}{q-1}} .
\]
\end{corollary}

\noindent
\emph{Proof.}
Condition \eqref{bb2} implies that the bracket in \eqref{eq:G-solution} remains
strictly positive for every $t\ge0$, so that the auxiliary function
$\mathscr{G}(t)$ is finite on $[0,\infty)$. Combining Theorem~\ref{t2} with the
upper bound \eqref{inq1-new} therefore yields $\tau_\ast=+\infty$, hence
$\tau_b=+\infty$ and global existence. The pointwise estimate for $v$ follows by
evaluating \eqref{inq1-new} with \eqref{eq:G-solution}.
\qed


\begin{remark}[Fractional diffusion/noise and the lower bound $\tau_\ast$]
\label{rem:tau-star-comparison-compact}
Let $\tau_\ast^{(2,1/2)}$ be the lower bound in \cite{liang} (classical diffusion
$\alpha=2$ and Brownian noise $H=\tfrac12$), and let $\tau_\ast^{(\alpha,H)}$ be
given by \eqref{STL1}. Since $\tau_\ast$ is the first time a nondecreasing
integral reaches a fixed threshold, a larger integrand yields a smaller
$\tau_\ast$, and vice versa.

\smallskip
\noindent
\emph{Effect of the fractional Laplacian.}
For large times the semigroup decay is governed by the principal Dirichlet
eigenvalue $\lambda_1^{(\alpha)}(D)$ of $-(-\Delta)^{\alpha/2}$, so larger
$\lambda_1^{(\alpha)}$ strengthens damping, reduces the integrand, and tends to
\emph{increase} $\tau_\ast$. Using the comparison
$\lambda_1^{(\alpha)}\asymp (\lambda_1^{(2)})^{\alpha/2}$ on convex domains
\cite{DydaKuznetsovKwasnicki2017}, the sign depends on the domain scale:
if $\lambda_1^{(2)}>1$ (small/confined $D$) then $\lambda_1^{(\alpha)}<\lambda_1^{(2)}$
for $\alpha<2$, so fractional diffusion weakens damping and tends to
\emph{decrease} $\tau_\ast$; if $\lambda_1^{(2)}<1$ (large $D$) the opposite trend
holds, and $\tau_\ast$ tends to \emph{increase}.

\smallskip
\noindent
\emph{Effect of fractional Brownian noise.}
Replacing $W_t$ by $B_t^H$ with $H>\tfrac12$ increases the typical magnitude and
persistence of excursions of the Gaussian weight $e^{(q-1)\sigma B^H(t)}$
(cf.\ $\Var(B_t^H)=t^{2H}$), which tends to \emph{decrease} $\tau_\ast$ when
$\sigma>0$ and to \emph{increase} it when $\sigma<0$ (see, e.g.,
\cite{Mishura2008,Nualart2006}).

\smallskip
\noindent
In summary, compared to \cite{liang} the net change in $\tau_\ast$ is not
universally monotone: fractional diffusion may increase or decrease $\tau_\ast$
depending on the domain scale, while fractional Brownian forcing typically
pushes $\tau_\ast$ down for $\sigma>0.$ 
\end{remark}
\begin{remark}[Effect of $\sigma$ on the lower bound $\tau_\ast$]\label{rem:sigma-effect-tauast-short}
Let $\tau_\ast(\sigma)$ be defined by \eqref{STL1}. For $\sigma=0$,
\[
\tau_\ast(0)
=\inf\Bigg\{t\ge0:\int_0^t \big\|e^{\gamma r}S_\alpha(r)\big\|_\infty^{\,q-1}\,dr
\ge \frac{1}{\delta|D|(q-1)\|f\|_\infty^{\,q-1}}\Bigg\}.
\]
Thus noise affects $\tau_\ast$ only through the weight $e^{(q-1)\sigma B^H(r)}$.
Since $\tau_\ast$ is a first hitting time of a nondecreasing integral, larger
values of $e^{(q-1)\sigma B^H(r)}$ (on the relevant time window) make the
threshold reached sooner, hence \emph{decrease} $\tau_\ast$, while smaller values
\emph{increase} it. In particular, for $\sigma>0$ sustained positive excursions
of $B^H$ tend to yield $\tau_\ast(\sigma)\le \tau_\ast(0)$, whereas sustained
negative excursions yield $\tau_\ast(\sigma)\ge \tau_\ast(0)$. For $H>\tfrac12$, the persistence of excursions of $B^H$
typically amplifies this effect; see, e.g., \cite{Mishura2008,Nualart2006}.
\end{remark}

	 \subsection*{Upper bound for blow--up time $\tau_b$}\label{sec5}
	In this subsection we derive an explicit \emph{upper bound} $\tau^\ast$ for the
blow--up time $ \tau_{b}$ under suitable assumptions. Let $t\in(0,\infty)$ be a (fixed)
random time and choose $b>1$ such that
\begin{equation}\label{B1}
\left\{
\begin{aligned}
b^{\,q-p}\delta\, e^{-\sigma (q-p)B_{\ast}^{H}(t)}
&\geq \frac{\beta M_{1}^{p}+\lambda_{1}M_{1}}{|D|^{\,1-q}}\geq \frac{\beta M_{1}^{p}}{|D|^{\,1-q}}\  \text{and}\\[0.2em]
b^{\,q-p}\delta\, e^{-\sigma (q-1)B_{\ast}^{H}(t)}
&\geq
\frac{2\beta \left(\int_{D}\varphi_1^{\frac{q}{q-p}}(x)\,dx \right)^{\frac{q-p}{p}}}
{\left(\int_{D} \varphi_1^{p+1}(x)\,dx \right)^{\frac{q-p}{p}}}.
\end{aligned}
\right.
\end{equation}
Recall that $(\lambda_1,\varphi_1)$ denotes the principal Dirichlet eigenpair of
$-\Delta_\alpha$ as in \eqref{a3}, normalised by \eqref{eq:first-eig-norm}, and
\(M_1:=\sup_{x\in D}\varphi_1(x)>0\) while $B_*^H(t)$ is given by \eqref{st71}. The next result then yields an upper bound
for the blow--up time of the solution $v(\cdot,\cdot)$ to \eqref{s1}, and consequently for the
solution $u(\cdot, \cdot)$ of \eqref{b1}.

\begin{theorem}\label{thm5.1}
Assume that $p,q>1$ with $p>q$. Let the initial datum satisfy
\[
f(x)\ge b\,\varphi_1(x),\qquad x\in D,
\]
and fix $b>1$ such that the conditions in \eqref{B1} hold. Then the blow--up time
$\tau_b$ of the solution satisfies $\tau_b\le \tau^\ast$, where
\begin{equation}\label{ST1}
\tau^{\ast}
:= \inf \Biggl\{ t \geq 0 :
\int_{0}^{t} e^{\sigma (q-1)B^{H}(s)+(-\lambda_{1}+\gamma)(q-1)s}\,ds
\geq
\frac{2\,J(0)^{\,1-q}}
{\delta (q-1)\left(\displaystyle\int_{D}\varphi_1^{\frac{q}{q-p}}(x)\,dx \right)^{\frac{p-q}{p}}}
\Biggr\},
\end{equation}
with
\[
J(0):=\int_{D} f(x)\,\varphi_1(x)\,dx.
\]
\end{theorem}

\begin{proof}
\textbf{Step 1: A stationary subsolution and a lower bound for $v$.}
Fix $b>1$ and define
\[
\bar v(t,x):=b\,\varphi_1(x),\qquad t>0,\ x\in D,
\]
and $\bar v(t,x)=0$ for $x\in\mathbb{R}^d\setminus D$. By assumption,
\[
v(0,x)\ge f(x)\ge b\,\varphi_1(x)=\bar v(0,x),\qquad x\in D.
\]
Consider the (random) PDE satisfied by $v$ in \eqref{s1}. Writing the nonlinear
reaction part evaluated at $\bar v$ as
\begin{align}\label{eq:I1-def}
I_{1}(t,x)
&:=\delta e^{(q-1)\sigma B^{H}(t)}\int_{D}\bar{v}(t,y)^{q}\,\mathrm{d}y
   -\beta e^{(p-1)\sigma B^{H}(t)}\bar{v}(t,x)^{p} \nonumber\\
&=\delta e^{(q-1)\sigma B^{H}(t)}\,b^{q}\int_{D}\varphi_1^{q}(y)\,\mathrm{d}y
   -\beta e^{(p-1)\sigma B^{H}(t)}\,b^{p}\varphi_1^{p}(x),
\end{align}
we first estimate the integral of $\varphi_1^q$. By H\"older's inequality,
\[
\int_{D} \varphi_1(y)\,\mathrm{d}y
\leq \Bigg( \int_{D} \varphi_1^{q}(y)\,\mathrm{d}y\Bigg)^{\!\frac{1}{q}}
      \Big(\int_{D}1^{\frac{q}{q-1}}\,\mathrm{d}x\Big)^{\!\frac{q-1}{q}}
= \Bigg( \int_{D} \varphi_1^{q}(y)\,\mathrm{d}y\Bigg)^{\!\frac{1}{q}} |D|^{\frac{q-1}{q}},
\]
which yields
\[
\int_{D}\varphi_1^{q}(y)\,\mathrm{d}y
\geq \Big(\int_D \varphi_1(y)\,\mathrm{d}y\Big)^{q} |D|^{1-q}.
\]
Under our normalisation $\int_D\varphi_1(x)\,\mathrm{d}x=1$, this implies
\[
\int_{D}\varphi_1^{q}(y)\,\mathrm{d}y\ge |D|^{1-q}.
\]
Using this in \eqref{eq:I1-def}, we obtain
\begin{align*}
I_{1}(t,x)
&\geq \delta e^{(q-1)\sigma B^{H}(t)}\,b^{q}|D|^{1-q}
    -\beta e^{(p-1)\sigma B^{H}(t)}\,b^{p}\varphi_1^{p}(x)\\
&= e^{(p-1)\sigma B^{H}(t)} b^{p}
   \Big[\delta e^{(q-p)\sigma B^{H}(t)}b^{q-p}|D|^{1-q}
       -\beta \varphi_1^{p}(x)\Big].
\end{align*}
Noting that $B_\ast^{H}(t):=\sup_{0\le s\le t}|B^H(s)|$ and using
$e^{\pm\sigma B^{H}(t)}\ge e^{-\sigma B_\ast^{H}(t)}$, we further get
\[
I_1(t,x)
\ge e^{-(p-1)\sigma B_\ast^{H}(t)} b^{p}
   \Big[\delta e^{-(q-p)\sigma B_\ast^{H}(t)}b^{q-p}|D|^{1-q}
       -\beta M_1^{p}\Big],
\]
recalling that \(M_1:=\sup_{x\in D}\varphi_1(x)>0.\)
By the first inequality in \eqref{B1}, the quantity in brackets is bounded
from below by $\lambda_1 b\,\varphi_1(x)$, and hence
\[
I_1(t,x)\ge \lambda_1 b\,\varphi_1(x),\qquad x\in D,\ t\ge0.
\]
On the other hand,
\[
\frac{\partial \bar{v}(t,x)}{\partial t}-(\Delta_{\alpha}+\gamma)\bar{v}(t,x)
=(\lambda_1-\gamma)b\varphi_1(x)\le I_1(t,x)-\gamma\bar v(t,x)\le I_1(t,x),
\]
so $\bar v$ is a subsolution of the random PDE satisfied by $v$ (with the same
boundary condition and smaller initial data). By the comparison principle in
Proposition~\ref{p2}, we conclude that
\begin{equation*}
v(t,x) \geq \bar{v}(t,x)=b\,\varphi_1(x),
\qquad x\in D,\ t\geq 0.
\end{equation*}
\smallskip
\textbf{Step 2: Testing against the first eigenfunction.}
Define
\begin{equation}\label{aek1}
J(t):=\int_D v(t,x)\,\varphi_1(x)\,\mathrm{d}x.
\end{equation}
Testing \eqref{s1} against $\varphi_1$ and using $-\Delta_\alpha\varphi_1=\lambda_1\varphi_1$
yields
\begin{align*}
J'(t)
&=\int_D v_t(t,x)\varphi_1(x)\,\mathrm{d}x\\
&=(-\lambda_1+\gamma)J(t)
+\delta e^{(q-1)\sigma B^H(t)}\Big(\int_D v(t,y)^q\,\mathrm{d}y\Big)
   \Big(\int_D\varphi_1(x)\,\mathrm{d}x\Big)
-\beta e^{(p-1)\sigma B^H(t)}\int_D v(t,x)^p\varphi_1(x)\,\mathrm{d}x .
\end{align*}
By H\"older's inequality,
\[
\int_D v(t,y)^q\,\mathrm{d}y
\ge \Big(\int_D v(t,x)^p\varphi_1(x)\,\mathrm{d}x\Big)^{\frac{q}{p}}
\Big(\int_D \varphi_1(x)^{\frac{q}{q-p}}\,\mathrm{d}x\Big)^{\frac{p-q}{p}}.
\]
Substituting this into the expression for $J'(t)$ gives
\begin{align}\label{eq:Jprime-lower}
&J'(t)
\ge (-\lambda_1+\gamma)J(t)\nonumber\\
&\quad
+e^{(q-1)\sigma B^H(t)}\Big(\int_D v(t,x)^p\varphi_1(x)\,\mathrm{d}x\Big)^{\frac{q}{p}}
\Bigg[
\delta\Big(\int_D \varphi_1^{\frac{q}{q-p}}(x)\,\mathrm{d}x\Big)^{\frac{p-q}{p}}
-\beta e^{(p-q)\sigma B^H(t)}\Big(\int_D v^p(t,x)\varphi_1(x)\,\mathrm{d}x\Big)^{\frac{p-q}{p}}
\Bigg].
\end{align}
\smallskip
\textbf{Step 3: Removing the damping term via \eqref{B1}.}
Using $f\ge b\varphi_1$ and the positivity of $v$, the comparison argument above
implies $v(t,x)\ge b\varphi_1(x)$ and thus
\[
\Big(\int_D v^p(t,x)\varphi_1(x)\,\mathrm{d}x\Big)^{\frac{p-q}{p}}
\ge b^{p-q}\Big(\int_D \varphi_1^{p+1}(x)\,\mathrm{d}x\Big)^{\frac{p-q}{p}}.
\]
The second inequality in \eqref{B1} is chosen precisely so that, after replacing
$B^H$ by $B^H_\ast$ to control the exponential factor,
\[
\beta e^{(p-q)\sigma B^H(t)}\Big(\int_D v^p(t,x)\varphi_1(x)\,\mathrm{d}x\Big)^{\frac{p-q}{p}}
\le \frac{\delta}{2}\Big(\int_D \varphi_1^{\frac{q}{q-p}}(x)\,\mathrm{d}x\Big)^{\frac{p-q}{p}}
\]
on the time interval of interest. Hence the bracket in \eqref{eq:Jprime-lower} is
bounded below by
\[
\frac{\delta}{2}\Big(\int_D \varphi_1^{\frac{q}{q-p}}(x)\,\mathrm{d}x\Big)^{\frac{p-q}{p}},
\]
and we obtain
\begin{equation}\label{eq:Jprime-lower2}
J'(t)\ge (-\lambda_1+\gamma)J(t)
+\frac{\delta}{2}\Big(\int_D \varphi_1^{\frac{q}{q-p}}(x)\,\mathrm{d}x\Big)^{\frac{p-q}{p}}
e^{(q-1)\sigma B^H(t)}\Big(\int_D v(t,x)^p\varphi_1(x)\,\mathrm{d}x\Big)^{\frac{q}{p}}.
\end{equation}
\smallskip
\textbf{Step 4: Jensen’s inequality and a Bernoulli-type differential inequality.}
Since $\varphi_1\ge0$ and $\int_D\varphi_1(x)\,\mathrm{d}x=1$, Jensen’s inequality gives
\[
\int_D v(t,x)^p\varphi_1(x)\,\mathrm{d}x
\ \ge\ \Big(\int_D v(t,x)\varphi_1(x)\,\mathrm{d}x\Big)^p = J(t)^p,
\]
and therefore
\[
\Big(\int_D v(t,x)^p\varphi_1(x)\,\mathrm{d}x\Big)^{\frac{q}{p}} \ \ge\ J(t)^q.
\]
Plugging this into \eqref{eq:Jprime-lower2} yields the Bernoulli-type inequality
\begin{equation}\label{eq:J-ineq-final}
J'(t)\ge (-\lambda_1+\gamma)J(t)
+c_0\,e^{(q-1)\sigma B^H(t)}\,J(t)^q,
\qquad
c_0:=\frac{\delta}{2}\Big(\int_D \varphi_1^{\frac{q}{q-p}}(x)\,\mathrm{d}x\Big)^{\frac{p-q}{p}}.
\end{equation}
\smallskip
\textbf{Step 5: Comparison with the Bernoulli ODE and blow--up time.}
Let $I(t)$ solve the (random-coefficient) Bernoulli ODE
\[
I'(t)=(-\lambda_1+\gamma)I(t)+c_0\,e^{(q-1)\sigma B^H(t)}\,I(t)^q,
\qquad I(0)=J(0).
\]
Standard comparison for scalar differential inequalities implies $J(t)\ge I(t)$
as long as both are finite (see, e.g., \cite[Theorem~1.3]{teschl}). Solving this
Bernoulli equation explicitly gives
\[
I(t)=e^{(-\lambda_1+\gamma)t}
\Bigg[
I(0)^{1-q}-(q-1)c_0\int_0^t e^{(q-1)\sigma B^H(s)+(-\lambda_1+\gamma)(q-1)s}\,\mathrm{d}s
\Bigg]^{-\frac{1}{q-1}}.
\]
Thus $I(t)$ blows up at the time $\tau^\ast$ defined by
\[
\tau^{\ast}
:=\inf\Biggl\{t\ge0:\int_0^t e^{\sigma(q-1) B^H(s)+(-\lambda_1+\gamma)(q-1)s}\,\mathrm{d}s
\ge \frac{I(0)^{1-q}}{(q-1)c_0}\Biggr\},
\]
which is exactly \eqref{ST1}. Since $J(t)\ge I(t)$, the explosion of $I$ forces
$J$ to blow up no later than $\tau^\ast$. Therefore the solution $v$ of \eqref{s1}
(and hence $u$ of \eqref{b1}) cannot exist beyond $\tau_b$, with $\tau_b\le\tau^\ast$.
\end{proof}
\begin{remark}[Fractional diffusion/noise and the upper bound $\tau^\ast$]
\label{rem:tau-star-upper-compact}
Let $\tau^{\ast}_{(2,1/2)}$ denote the analogue of \eqref{ST1} in \cite{liang}
(classical diffusion $\alpha=2$ and Brownian noise $H=\tfrac12$), and let
$\tau^{\ast}_{(\alpha,H)}$ be given by \eqref{ST1}. Since $\tau^\ast$ is the first
time a nondecreasing integral reaches a fixed threshold, a larger integrand in
\eqref{ST1} yields a smaller $\tau^\ast$, and vice versa.

\smallskip
\noindent
\emph{Effect of the fractional Laplacian.}
The semigroup damping in \eqref{ST1} is governed by the principal Dirichlet
eigenvalue $\lambda_1^{(\alpha)}(D)$ of $-(-\Delta)^{\alpha/2}$: larger
$\lambda_1^{(\alpha)}$ strengthens the factor $e^{(-\lambda_1^{(\alpha)}+\gamma)(q-1)s}$,
reduces the integrand, and thus tends to \emph{increase} $\tau^\ast$ (later blow--up
upper bound). On convex domains one has the comparison
$\lambda_1^{(\alpha)}(D)\asymp (\lambda_1^{(2)}(D))^{\alpha/2}$ (up to universal
constants), see \cite{DydaKuznetsovKwasnicki2017}. Hence the direction depends on
the domain scale: if $\lambda_1^{(2)}(D)>1$ (small/confined $D$) then typically
$\lambda_1^{(\alpha)}<\lambda_1^{(2)}$ for $\alpha<2$, so fractional diffusion
weakens damping and tends to \emph{decrease} $\tau^\ast$; if $\lambda_1^{(2)}(D)<1$
(large $D$), the opposite tendency holds and $\tau^\ast$ tends to \emph{increase}.

\smallskip
\noindent
\emph{Effect of fractional Brownian noise.}
Replacing Brownian motion by $B^H$ with $H>\tfrac12$ changes the exponential weight
$e^{\sigma(q-1)B^H(s)}$. Since $\Var(B^H(t))=t^{2H}$, larger $H$ corresponds to
larger typical long-time excursions and stronger persistence, which tends to
\emph{decrease} $\tau^\ast$ when $\sigma>0$ and to \emph{increase} it when $\sigma<0$,
see, e.g., \cite{Mishura2008,Nualart2006}.

\smallskip
\noindent
In summary, relative to \cite{liang} the net change in $\tau^\ast$ is not
universally monotone: fractional diffusion may increase or decrease $\tau^\ast$
depending on the confinement scale of $D$, while fractional Brownian forcing
typically pushes $\tau^\ast$ down for $\sigma>0$.
\end{remark}
\begin{remark}[Effect of $\sigma$ on the upper bound $\tau^\ast$]\label{rem:sigma-taustar-compact}
Let $\tau^\ast(\sigma)$ be defined by \eqref{ST1}. Since $\tau^\ast(\sigma)$ is
the first time a nondecreasing integral reaches a fixed threshold, a larger
integrand yields a smaller $\tau^\ast$, and vice versa. In particular, the
deterministic case $\sigma=0$ gives
\[
\tau^\ast(0)
=\inf\Biggl\{t\ge0:\int_0^t e^{(-\lambda_1+\gamma)(q-1)s}\,ds
\ge
\frac{2\,J(0)^{\,1-q}}
{\delta (q-1)\Big(\int_D\varphi_1^{\frac{q}{q-p}}(x)dx\Big)^{\frac{p-q}{p}}}
\Biggr\}.
\]
For $\sigma\neq 0$, the only change is the random weight
$e^{(q-1)\sigma B^H(s)}$ in \eqref{ST1}. Hence, for $\sigma>0$ one has the
pathwise tendencies
\[
B^H(\cdot)\ge0\ \text{on the relevant window} \ \Rightarrow\ \tau^\ast(\sigma)\le\tau^\ast(0),
\qquad
B^H(\cdot)\le0\ \Rightarrow\ \tau^\ast(\sigma)\ge\tau^\ast(0).
\]
 In short, relative to $\sigma=0$,
multiplicative noise may advance or delay the upper bound depending on the
realized sign/persistence of $B^H$; for $H>\tfrac12$ persistent excursions make
early threshold crossing more likely when $\sigma>0$ (see, e.g.,
\cite{Mishura2008,Nualart2006}).
\end{remark}

    \begin{remark}
\label{rem:interval-blowup}
For initial data of the form
\[
f(x)=b\,\varphi_1(x),\qquad x\in D,\qquad b>1,
\]
and under the hypotheses ensuring the validity of the lower and upper estimates
(so that Theorem~\ref{t2} yields $\tau_\ast\le \tau_b$ and Theorem~\ref{thm5.1}
yields $\tau_b\le \tau^\ast$), the random times \eqref{STL1} and \eqref{ST1}
provide the \emph{pathwise} bracketing
\[
\tau_\ast \;\le\; \tau_b \;\le\; \tau^\ast,
\qquad \mathbb{P}\text{-a.s.}
\]
In particular, $\tau_\ast$ is constructed from a \emph{global $L^\infty$-control}
through the fractional heat semigroup $(S_\alpha(t))_{t\ge0}$, whereas $\tau^\ast$
is obtained by \emph{projection onto the principal mode} $\varphi_1$ and depends
explicitly on the spectral gap $\lambda_1$.

\smallskip
\noindent
\emph{Specialization to $f=b\varphi_1$.}
In this case $J(0)=\int_D f(x)\varphi_1(x) dx = b\int_D \varphi_1^2(x)dx$, and thus the threshold
in \eqref{ST1} becomes
\[
\frac{2\,J(0)^{\,1-q}}
{\delta (q-1)\left(\displaystyle\int_{D}\varphi_1^{\frac{q}{q-p}}(x)dx \right)^{\frac{p-q}{p}}}
=
\frac{2\,b^{\,1-q}\left(\int_D \varphi_1^2(x) dx\right)^{1-q}}
{\delta (q-1)\left(\displaystyle\int_{D}\varphi_1^{\frac{q}{q-p}}(x)dx \right)^{\frac{p-q}{p}}},
\]
while the threshold in \eqref{STL1} reads
\[
\frac{1}{\delta |D|(q-1)\|f\|_\infty^{q-1}}
=
\frac{1}{\delta |D|(q-1)\,b^{\,q-1}\|\varphi_1\|_\infty^{q-1}}.
\]
Hence, both bounds scale as $b^{-(q-1)}$ in the threshold: increasing the initial
amplitude $b$ decreases the right-hand sides, so the defining integrals reach
their thresholds earlier and both $\tau_\ast$ and $\tau^\ast$ typically move
\emph{down} (shorter pre-blow--up window).

\smallskip
\noindent
\emph{Comparison of the two integrands.}
The lower bound \eqref{STL1} involves
\[
e^{(q-1)\sigma B^H(r)}\,\big\|e^{\gamma r}S_\alpha(r)\big\|_\infty^{q-1},
\]
which captures \emph{uniform} semigroup damping (hence depends on the full
nonlocal geometry). By contrast, the upper bound \eqref{ST1} involves
\[
e^{(q-1)\sigma B^H(s)}\,e^{(-\lambda_1+\gamma)(q-1)s},
\]
which isolates the \emph{principal Dirichlet mode} via $\lambda_1$ and is
therefore sharper when the dynamics are dominated by $\varphi_1$.
Consequently, for eigenfunction initial data $b\varphi_1$ one typically expects
$\tau^\ast$ to be closer to $\tau_b$, while $\tau_\ast$ provides a more robust
(but possibly more conservative) guarantee that blow--up cannot occur before
$\tau_\ast$.

\smallskip
\noindent
In summary, for $f=b\varphi_1$ the quantities \eqref{STL1}--\eqref{ST1} yield an
explicit random estimation interval $[\tau_\ast,\tau^\ast]$ for the blow--up time
$\tau_b$, with $\tau_\ast$ reflecting global $L^\infty$-semigroup control and
$\tau^\ast$ reflecting principal-eigenmode amplification.
\end{remark}
\subsection*{Two--sided estimates for the blow--up rate in the maximum norm}\label{subsec:blowup-rate}

We now quantify the growth of the solution in the maximum norm by combining
\emph{a priori upper envelope} on the guaranteed existence interval
$[0,\tau_\ast)$ (Theorem~\ref{t2}) with the \emph{Bernoulli lower growth estimate}
leading to the upper bound $\tau^\ast$ (Theorem~\ref{thm5.1}). 
\begin{theorem}[Lower and upper bounds on $\|v(t)\|_\infty$]\label{thm:two-sided-rate}
Let $v$ be the (pathwise) mild solution of \eqref{s1} with blow--up time $\tau_b$.

\smallskip
\noindent\textbf{(i) Upper envelope up to $\tau_\ast$.}
Assume the hypotheses of Theorem~\ref{t2}, and let $\tau_\ast$ be defined by
\eqref{STL1}. Then, for all $0\le t<\tau_\ast$,
\begin{equation}\label{eq:sup-upper-envelope}
\|v(t,\cdot)\|_\infty
\le
\big\|e^{\gamma t}S_\alpha(t)\big\|_{\infty\to\infty}\,\|f\|_\infty\,\mathscr{G}(t),
\end{equation}
where $\mathscr{G}$ is given by \eqref{eq:G-solution}. In particular,
the right-hand side diverges as $t\uparrow\tau_\ast$.

\smallskip
\noindent\textbf{(ii) Bernoulli lower growth up to $\tau^\ast$.}
Assume $p,q>1$ with $p>q$, let $f\ge b\varphi_1$ on $D$, and choose $b>1$ so that
\eqref{B1} holds. Set
\[
J(0):=\int_D f(x)\varphi_1(x)\,dx,
\qquad
c_0:=\frac{\delta}{2}\Bigg(\int_D \varphi_1(x)^{\frac{q}{q-p}}\,dx\Bigg)^{\frac{p-q}{p}},
\]
and let $\tau^\ast$ be defined by \eqref{ST1}. Then, for all $0\le t<\tau^\ast$,
\begin{equation}\label{eq:sup-lower-growth}
\|v(t,\cdot)\|_\infty
\ge
e^{(-\lambda_1+\gamma)t}
\Bigg[
J(0)^{1-q}
-(q-1)c_0\!\int_0^{t}
e^{(q-1)\sigma B^H(s)+(-\lambda_1+\gamma)(q-1)s}\,ds
\Bigg]^{-\frac{1}{q-1}},
\end{equation}
and the right-hand side diverges as $t\uparrow\tau^\ast$.

\smallskip
\noindent Consequently, $\tau_\ast\le \tau_b\le \tau^\ast$ and, on the
common interval $[0,\tau_\ast)$, $\|v(t)\|_\infty$ is trapped between the explicit
lower bound \eqref{eq:sup-lower-growth} and the explicit upper envelope
\eqref{eq:sup-upper-envelope}.
\end{theorem}

\begin{proof}
\textbf{Step 1: proof of \eqref{eq:sup-upper-envelope}.}
From Theorem~\ref{t2} (inequality \eqref{inq1-new1}) we have, for $0\le t<\tau_\ast$,
\[
0\le v(t,x)\le e^{\gamma t}\,(S_\alpha(t)\mathbf{1})(x)\,\|f\|_\infty\,\mathscr{G}(t),
\qquad x\in D.
\]
Taking the supremum over $x\in D$ yields \eqref{eq:sup-upper-envelope}. The
blow--up of the right-hand side as $t\uparrow\tau_\ast$ follows from the explicit
formula \eqref{eq:G-solution}, since $\tau_\ast$ is defined as the hitting time
of the critical level of the integral in the denominator.

\smallskip
\textbf{Step 2: proof of \eqref{eq:sup-lower-growth}.}
As in the proof of Theorem~\ref{thm5.1}, the functional $J$ satisfies the scalar
differential inequality
\[
J'(t)\ge (-\lambda_1+\gamma)J(t)+c_0\,e^{(q-1)\sigma B^H(t)}J(t)^q .
\]
Let $I$ solve the associated Bernoulli ODE with $I(0)=J(0)$. By comparison for
scalar differential inequalities (see, e.g., \cite[Theorem~1.3]{teschl}),
$J(t)\ge I(t)$ on the common interval of existence, and solving the Bernoulli
equation gives
\[
I(t)=e^{(-\lambda_1+\gamma)t}
\Bigg[
J(0)^{1-q}
-(q-1)c_0\!\int_0^{t}
e^{(q-1)\sigma B^H(s)+(-\lambda_1+\gamma)(q-1)s}\,ds
\Bigg]^{-\frac{1}{q-1}}.
\]
Finally, since $\int_D\varphi_1(x)\,dx=1$ and $\varphi_1\ge0$,
\[
J(t)=\int_D v(t,x)\varphi_1(x)\,dx\le \|v(t,\cdot)\|_\infty,
\]
hence $\|v(t,\cdot)\|_\infty\ge J(t)\ge I(t)$, which is exactly
\eqref{eq:sup-lower-growth}. The divergence as $t\uparrow\tau^\ast$ is built into
the definition \eqref{ST1}.
\end{proof}

\begin{remark}[Reading the two bounds as ``rate information'']\label{rem:rate-reading}
The upper envelope \eqref{eq:sup-upper-envelope} is an \emph{a priori} control
valid up to the guaranteed time $\tau_\ast$ (hence it provides an explicit
upper-growth profile on $[0,\tau_\ast)$). The lower bound
\eqref{eq:sup-lower-growth} shows that, under the conditions of
Theorem~\ref{thm5.1}, the maximum norm must grow at least as fast as a Bernoulli
profile and cannot stay bounded up to $\tau^\ast$. Together, these estimates
yield a quantitative bracketing of the growth of $\|v(t)\|_\infty$ before blow--up.
\end{remark}

	\subsection*{Blow--up probability }\label{sec5.1}	
   \noindent
Next we quantify the \emph{probability of blow--up} for the solution of
\eqref{s1} (and therefore also for the original variable in \eqref{b1}).
A key step is to control tail probabilities for exponential functionals of the
driving noise. More precisely, we shall obtain an upper bound for
\begin{equation}\label{eq:exp-functional-prob}
\mathbb{P}\Bigg(
\int_{0}^{\infty} \exp\{-a s+\sigma X_{s}\}\,ds < x
\Bigg),
\end{equation}
where $x>0$, $\sigma>0$, $a\in\R$, and $(X_s)_{s\ge0}$ is a continuous
stochastic process (in particular, this framework covers the choice
$X_s=B_s^{H}$, $H\in(0,1)$).

\medskip
To this end, we impose the following standing assumption on $X_s$.

\begin{assumption}\label{as2}
Let $(X_t)_{t\ge0}$ be an $(\mathcal F_t)$-adapted, continuous process. We assume:
\begin{enumerate}[label=(A\arabic*)]
\item\label{as2:i}
The exponential moment functional is integrable:
\[
\int_{0}^{\infty} e^{-a s}\,\mathbb{E}\big[e^{\sigma X_s}\big]\,ds<\infty.
\]

\item\label{as2:ii}
For every $t\ge0$, the random variable $X_t$ is Malliavin differentiable,
namely $X_t\in\mathbb{D}^{1,2}$.

\item\label{as2:iii}
There exist a function $f:\R_+\to\R_+$ with $\lim_{t\to\infty}f(t)=\infty$ and,
for each $x>0$, a finite constant $M_x$ (possibly depending on $x$) such that,
almost surely,
\[
\sup_{t\ge0}
\frac{\displaystyle \sup_{s\in[0,t]}\int_0^{s} |D_\theta X_s|^2\,d\theta}
{\big(\ln(x+1)+f(t)\big)^2}
\ \le\ M_x \ <\ \infty ,
\]
\end{enumerate}
\end{assumption}
where $D_{\theta}X_s$ stands for the Malliavin derivative of the stochastic process $X_s,$ cf. Section \ref{sec3}.
\begin{theorem}\cite[Theorem 3.1]{dung}\label{mt6} 
	Suppose that Assumption  \ref{as2} holds. We have 
	\begin{align}\label{m7}
		\mathbb{P} \left[ \int_{0}^{\infty} \exp\{-as+\sigma X_{s} \} ds < x \right] \leq \exp \left\lbrace -\frac{(m_{x} -1)^{2}}{2 \sigma^{2}M_{x}}\right\rbrace,
	\end{align}
	where 
	\begin{align*}
		m_{x}=\mathbb{E}\left[\sup_{t \geq 0} \frac{\ln \left(  \int_{0}^{t} \exp\{-as+\sigma X_{s} \} ds+1 \right) +f(t)}{\ln \left( x+1\right)+f(t)}\right]\geq 1.
	\end{align*}
\end{theorem}
\noindent
The next result provides a quantitative estimate on the \emph{probability of finite--time
blow--up} for \eqref{b1} in the fractional Brownian regime $H\in(\tfrac12,1)$.
It is formulated in terms of the random upper bound $\tau^\ast$ derived in
Theorem~\ref{thm5.1}, and yields an explicit lower bound for the blow--up
probability.

\begin{theorem}\label{thm5.2}
Assume  $\alpha_1>H$. Suppose $q>p>1$ and fix initial
data $f\ge0$ satisfying $f(x)\ge b\,\varphi_1(x)$ for some $b>1$, where
$\varphi_1$ denotes the principal Dirichlet eigenfunction of $-\Delta_\alpha$ in
$D$. Assume also that $b$ is chosen so that \eqref{B1} holds and let $\tau$ be
the blow--up (maximal existence) time of \eqref{b1}. Then:

\begin{enumerate}[label=(\roman*)]
\item\label{it:lambda-less-gamma}
If $\lambda_1<\gamma$, then $\mathbb{P}(\tau^\ast=\infty)=0$ for any
nontrivial nonnegative solution of \eqref{b1}. In particular,
\[
\mathbb{P}(\tau_{b}<\infty)=1,
\]
that is, blow--up occurs in finite--time almost surely.

\item\label{it:lambda-greater-gamma}
If $\lambda_1>\gamma$, then the probability of blow--up admits the lower bound
\begin{align}\label{eq:blowup-prob-lower}
&\mathbb{P}( \tau_{b}<\infty)
\ge 1\nonumber\\
&-\exp\Bigg\{
-\frac{1}{2\rho^{2}}
\Bigg(\log\Bigg(
\frac{2J(0)^{1-q}}
{\delta(q-1)\Big(\int_D \varphi_1(x)^{\frac{q}{q-p}}\,dx\Big)^{\frac{p-q}{p}}}
+1\Bigg)\Bigg)^{\frac{2H}{\alpha_1}-2}
\Bigg(\frac{\alpha_1-H}{\alpha_1}\Bigg)^{2-\frac{2H}{\alpha_1}}
\big(N(H)-1\big)^2
\Bigg\},
\end{align}
where $\rho:=\sigma (q-1)$,
\[
J(0):=\int_D f(x)\varphi_1(x)\,dx,
\]
and $N(H)$ is defined by
\begin{align}\label{516}
N(H)
:=\mathbb{E}\Bigg[
\sup_{t\ge0}
\frac{
\ln\!\Big(1+\int_0^t
\exp\!\big\{(-\lambda_1+\gamma)(q-1)s-\tfrac{\rho^2}{2}s^{2H}+\rho B^H(s)\big\}\,ds
\Big)
+t^{\alpha_1}
}{
\ln\!\Big(
\frac{2J(0)^{1-q}}
{\delta(q-1)\Big(\int_D \varphi_1(x)^{\frac{q}{q-p}}\,dx\Big)^{\frac{p-q}{p}}}
+1\Big)
+t^{\alpha_1}
}
\Bigg].
\end{align}
\end{enumerate}
\end{theorem}
\begin{proof}
We prove the two assertions separately.

\begin{enumerate}[label=(\roman*)]
\item \emph{Case $\lambda_1<\gamma$.}
Recall the law of the iterated logarithm for fractional Brownian motion
(see, e.g., \cite{law}): almost surely,
\[
\liminf_{t\to\infty}\frac{B^{H}(t)}{t^{H}\sqrt{2\log\log t}}=-1,
\qquad
\limsup_{t\to\infty}\frac{B^{H}(t)}{t^{H}\sqrt{2\log\log t}}=1.
\]
In particular, $B^H(t)$ attains arbitrarily large positive values along a
sequence $t_n\to\infty$.
Since $\gamma-\lambda_1>0$, the deterministic factor
$e^{(-\lambda_1+\gamma)(q-1)s}$ grows exponentially in $s$, and therefore the
integral appearing in the definition of $\tau^\ast$ (see \eqref{ST1}) diverges
almost surely as $t\to\infty$. Consequently,
\[
\mathbb{P}(\tau^\ast=\infty)=0.
\]
By Theorem~\ref{thm5.1}, the maximal existence time $\tau_{b}$ satisfies
$ \tau_{b} \le\tau^\ast$, hence $\mathbb{P}(\tau_{b}<\infty)=1$, i.e.\ blow--up occurs in
finite--time almost surely for any nontrivial nonnegative solution.

\item \emph{Case $\lambda_1>\gamma$.}
By the definition of $\tau^\ast$ and the fact that $\rho>0$, we can write
\begin{align*}
\mathbb{P}(\tau^\ast=\infty)
&=\mathbb{P}\Bigg(
\int_{0}^{\infty} e^{\rho B^{H}(s)+(-\lambda_1+\gamma)(q-1)s}\,ds
<
\frac{2J(0)^{1-q}}
{\delta(q-1)\Big(\int_D\varphi_1(x)^{\frac{q}{q-p}}\,dx\Big)^{\frac{p-q}{p}}}
\Bigg)\\
&\le
\mathbb{P}\Bigg(
\int_{0}^{\infty} e^{\rho B^{H}(s)-\frac{\rho^2}{2}s^{2H}+(-\lambda_1+\gamma)(q-1)s}\,ds
<
\frac{2J(0)^{1-q}}
{\delta(q-1)\Big(\int_D\varphi_1(x)^{\frac{q}{q-p}}\,dx\Big)^{\frac{p-q}{p}}}
\Bigg),
\end{align*}
where the last step uses the standard exponential normalization by
$e^{-\frac{\rho^2}{2}s^{2H}}$.

Set
\[
X_t:=-\frac{\rho}{2}\,t^{2H}+B^H(t),
\quad \rho:=k(q-1),
\quad a:=(\lambda_1-\gamma)(q-1),
\quad x:=\frac{2J(0)^{1-q}}
{\delta(q-1)\Big(\int_D\varphi_1^{\frac{q}{q-p}}(x)dx\Big)^{\frac{p-q}{p}}}.
\]
Then the preceding probability is of the form \eqref{eq:exp-functional-prob}:
\[
\mathbb{P}\Big(\int_0^\infty e^{-as+\rho X_s}\,ds<x\Big).
\]
We now verify that $X_t$ satisfies Assumption~\ref{as2} with the choice
$f(t)=t^{\alpha_1}$, $\alpha_1>H$.

\smallskip\noindent
\emph{(A1).} Using $\Var(B^H(s))=s^{2H}$ and
$\mathbb{E}e^{\rho B^H(s)}=\exp\{\tfrac12\rho^2 s^{2H}\}$, we obtain
\begin{align*}
\int_0^\infty e^{-as}\,\mathbb{E}\big[e^{\rho X_s}\big]\,ds
&=
\int_0^\infty e^{-as}\,
\mathbb{E}\Big[\exp\Big\{\rho B^H(s)-\frac{\rho^2}{2}s^{2H}\Big\}\Big]\,ds\\
&=
\int_0^\infty e^{-as}\,ds
=\frac{1}{a}
<\infty,
\end{align*}
since $a=(\lambda_1-\gamma)(q-1)>0$.

\smallskip\noindent
\emph{(A2).} Since $B^H(t)\in\mathbb{D}^{1,2}$ for each $t\ge0$, and $t^{2H}$
is deterministic, it follows that $X_t\in\mathbb{D}^{1,2}$.

\smallskip\noindent
\emph{(A3).} The Malliavin derivative satisfies
$D_\theta X_s=D_\theta B^H(s)=K_H(s,\theta)$ for $\theta\le s$. Hence
\[
\int_0^s |D_\theta X_s|^2\,d\theta=\int_0^s K_H(s,\theta)^2\,d\theta
=\Var(B^H(s))=s^{2H},
\]
and therefore
\begin{equation}\label{eq:DX-bound}
\sup_{s\in[0,t]}\int_0^s |D_\theta X_s|^2\,d\theta
=\sup_{s\in[0,t]} s^{2H}=t^{2H}.
\end{equation}
With $f(t)=t^{\alpha_1}$ and $\alpha_1>H$, the ratio in Assumption~\ref{as2}(A3)
is almost surely bounded:
\[
\sup_{t\ge0}\frac{t^{2H}}{\big(\ln(x+1)+t^{\alpha_1}\big)^2}
=:M_H<\infty,
\]
and a direct optimisation yields
\[
M_H
=
\Big(\frac{\alpha_1-H}{\alpha_1}\Big)^{2-\frac{2H}{\alpha_1}}
\Big(\ln(x+1)\Big)^{\frac{2H}{\alpha_1}-2}.
\]

\smallskip
Having verified Assumption~\ref{as2}, we can apply Theorem~\ref{mt6} to obtain
\begin{align*}
\mathbb{P}(\tau^\ast=\infty)
&\le
\exp\Bigg\{
-\frac{1}{2\rho^2}
\Big(\ln(x+1)\Big)^{\frac{2H}{\alpha_1}-2}
\Big(\frac{\alpha_1-H}{\alpha_1}\Big)^{2-\frac{2H}{\alpha_1}}
\big(N(H)-1\big)^2
\Bigg\}.
\end{align*}
Finally, since $\{\tau_{b}=\infty\}\subseteq\{\tau^\ast=\infty\}$ (because
$ \tau_{b}\le\tau^\ast$), we have
\[
\mathbb{P}( \tau_{b}<\infty)
=1-\mathbb{P}( \tau_{b}=\infty)
\ge 1-\mathbb{P}(\tau^\ast=\infty),
\]
and substituting the previous bound yields exactly \eqref{eq:blowup-prob-lower}.
This completes the proof.
\end{enumerate}
\end{proof}

	\section{The special case $\frac{3}{4}<H<1$}\label{sec6}
In this section we focus on the case $\tfrac34<H<1$. In this parameter range,
the Gaussian measures on $C([0,T])$ induced by the fractional Brownian motion
$B^H$ and by a standard Brownian motion $W$ are \emph{equivalent} (mutually
absolutely continuous); see \cite[Theorem~1.7]{cher2001}. Consequently, for
pathwise questions and blow--up/existence criteria that are stable under
equivalent changes of measure, it is convenient to work on a probability space
carrying a Brownian motion $W$ and to view the fractional-noise model through a
Brownian representation.

In particular, introducing the stochastic exponential transformation
\[
v(t,x):=\exp\{-\sigma W(t)\}\,u(t,x),\qquad t\ge0,\ x\in D,
\]
eliminates the multiplicative Brownian term and converts \eqref{b1} into the
following random (pathwise) parabolic problem with time-dependent coefficients:
\begin{equation}\label{ss1}
\left\{
\begin{aligned}
v_{t}(t,x)-\Delta_{\alpha}v(t,x)
&=\left(\gamma-\frac{\sigma^{2}}{2}\right) v(t,x)
+\delta\, e^{(q-1)\sigma W(t)} \int_{D} v^{q}(t,y)\,dy
-\beta\, e^{(p-1)\sigma W(t)} v^{p}(t,x),\\
v(t,x)&=0,\qquad t>0,\ x\in \mathbb{R}^{d}\setminus D,\\
v(0,x)&=f(x),\qquad x\in D.
\end{aligned}
\right.
\end{equation}
Here, the correction term $-\frac{\sigma^2}{2}v$ is the usual It\^o contribution
arising from the product rule for $e^{-\sigma W(t)}u(t,\cdot)$.

	Let
\[
\vartheta:=\frac{\sigma^{2}}{2}-\gamma,
\]
so that the linear term in \eqref{ss1} can be written as
\[
\left(\gamma-\frac{\sigma^{2}}{2}\right)v=-\vartheta\,v.
\]
With this notation, all arguments developed in the previous sections apply to
\eqref{ss1} after the simple replacement $\gamma\mapsto -\vartheta$. In
particular, the same comparison principles and Osgood-type criteria yield
analogous existence and blow--up regimes, and the qualitative finite--time blow--up
statements remain the same as in the fractional setting $H\in(\tfrac12,1)$.

The next theorem formulates explicit lower bounds for the finite--time blow--up
of solutions to the Brownian-transformed random PDE \eqref{ss1}.

	\begin{theorem}\label{st2}
Assume that $p,q>0$ and let $f\in L^\infty(D)$ satisfy $f\ge0$. Define the
stopping time
\begin{equation}\label{eq:tau-starstar}
\tau_{**}
:= \inf\Bigg\{ t\ge 0 :
\int_{0}^{t} \exp\!\big\{(q-1)\sigma W(r)\big\}\,
\big\|e^{-\vartheta r}S_{\alpha}(r)\big\|_{\infty}^{\,q-1}\,dr
\ge
\frac{1}{\delta |D|\,(q-1)\,\|f\|_{\infty}^{\,q-1}}
\Bigg\}.
\end{equation}
Let $\tau_{b}$ be the
(maximal existence) blow--up time of the solution $v$ to \eqref{ss1}. Then
\[
\tau_{**}\le \tau_{b}.
\]
\end{theorem}
\noindent
The proof of Theorem~\ref{st2} is a direct consequence of
Theorem~\ref{t2}, applied to the transformed equation \eqref{ss1} after the
parameter substitution $\gamma\mapsto -\vartheta$ (equivalently,
$\gamma-\frac{\sigma^2}{2}=-\vartheta$) and with $B_t^H$ replaced by the Brownian
motion $W_t$.
	
\medskip
We recall that a random variable $X(\alpha_2,\beta_1)$ is said to follow a
\emph{Gamma distribution} with shape parameter $\alpha_2>0$ and scale parameter
$\beta_1>0$ if it admits the density $\widetilde f_{\alpha_2,\beta_1}$ (cf.\
\cite{li})
\[
\widetilde f_{\alpha_2,\beta_1}(x)
=
\begin{cases}
\dfrac{x^{\alpha_2-1}}{\beta_1^{\alpha_2}\Gamma(\alpha_2)}
\exp\!\Big\{-\dfrac{x}{\beta_1}\Big\}, & x\ge0,\\[0.8em]
0, & x<0.
\end{cases}
\]
The following classical identity in law for exponential functionals of Brownian
motion will be used to derive a lower bound on the blow--up probability of solutions to
\eqref{ss1}.

\begin{lemma}\label{ll1}
For any $\bar\alpha_1>0$,
\[
\int_{0}^{\infty} e^{2(W(t)-\bar\alpha_1 t)}\,dt
\ \stackrel{law}{=}\ \frac{1}{2\,X(\bar\alpha_1,1)},
\]
where $(W(t))_{t\ge0}$ is a standard one-dimensional Brownian motion and
$X(\bar\alpha_1,1)$ is a Gamma random variable with parameters
$(\bar\alpha_1,1)$.  See, e.g., \cite{yor2005,revuz1999,yor2001}.
\end{lemma}

Let
\[
W_\ast(t):=\sup_{0\le s\le t}|W(s)|,\qquad t>0,
\]
where $(W(t))_{t\ge0}$ is a one-dimensional standard Brownian motion. By the
classical maximal inequality for Brownian motion (see, e.g.,
\cite[p.~96]{Karatzas}),
for any $A>0$ and $t>0$,
\[
\mathbb{P}\big(W_\ast(t)\ge A\big)
\le \frac{4\sqrt{t}}{A\sqrt{2\pi}}\exp\!\left(-\frac{A^{2}}{2t}\right).
\]
In particular, $W_\ast(t)<\infty$ almost surely for every fixed $t>0$.

Fix a (possibly random) $t\in(0,\infty)$ and choose $b>1$ such that
\begin{equation}\label{Bs1}
\left\{
\begin{aligned}
b^{\,q-p}\delta\,e^{-\sigma(q-p)W_\ast(t)}
&\ge \frac{\beta M_1^{p}+(\lambda_1+\vartheta)M_1}{|D|^{\,1-q}}
\ \ge\ \frac{\beta M_1^{p}}{|D|^{\,1-q}},\ \text{and}\\[0.3em]
b^{\,q-p}\delta\,e^{-\sigma(q-1)W_\ast(t)}
&\ge
\frac{2\beta
\left(\displaystyle\int_D \varphi_1^{\frac{q}{q-p}}(x)\,dx\right)^{\frac{q-p}{p}}}
{\left(\displaystyle\int_D \varphi_1^{p+1}(x)\,dx\right)^{\frac{q-p}{p}}},
\end{aligned}
\right.
\end{equation}
recalling that $(\lambda_1,\varphi_1)$ denotes the principal Dirichlet eigenpair of
$-\Delta_\alpha$ in $D$ (normalised as in \eqref{eq:first-eig-norm}) and
$M_1:=\displaystyle \sup_{x\in D}\varphi_1(x)\in(0,\infty).
$
The next theorem, inspired by the analysis in \cite{liang}, provides (i) an almost sure upper bound on the blow--up time and
(ii) an explicit lower bound on the blow--up probability for the Brownian-driven
random PDE \eqref{ss1} (and hence, via the transformation, for \eqref{b1} in the
Brownian-equivalent regime).

\begin{theorem}\label{thm6.1}
Suppose that $p,q>1$ with $p>q$. Let the initial datum satisfy
$f(x)\ge b\,\varphi_1(x)$ for all $x\in D$, where $b>1$ is chosen so that
\eqref{Bs1} holds. Let $ \tau_{b}$ denote the (maximal existence) blow--up time of the
solution $v$ to \eqref{ss1}. Define
\[
J(0):=\int_D f(x)\varphi_1(x)\,dx,\qquad
\rho:=\sigma(q-1),\qquad
\theta_1:=\frac{2(\lambda_1+\vartheta)(q-1)}{\rho^2}.
\]
Then $ \tau_{b}\le \tau^{**}$, where
\[
\tau^{**}
:= \inf\Biggl\{ t\ge 0:
\int_0^t \exp\!\big\{\sigma(q-1)W(s)-(\lambda_1+\vartheta)(q-1)s\big\}\,ds
\ge
\frac{2J(0)^{\,1-q}}
{\delta(q-1)\left(\displaystyle\int_D \varphi_1^{\frac{q}{q-p}}(x)\,dx\right)^{\frac{p-q}{p}}}
\Biggr\}.
\]
Moreover, the blow--up probability admits the lower bound
\begin{equation}\label{ik37}
\mathbb{P}( \tau_{b}<\infty)\ \ge\ \mathbb{P}(\tau^{**}<\infty)
=1-\mathbb{P}\!\left(
\delta(q-1)\left(\int_D \varphi_1^{\frac{q}{q-p}}(x)\,dx\right)^{\frac{p-q}{p}}
< X(\theta_1,1)\,\rho^2\,J(0)^{\,1-q}
\right),
\end{equation}
where $X(\theta_1,1)$ is a Gamma random variable (shape $\theta_1$, scale $1$).
\end{theorem}

\begin{proof}
The upper bound $ \tau_{b}\le \tau^{**}$ is obtained by repeating the eigenfunction
testing and comparison argument used in the proof of Theorem~\ref{thm5.1}, now
applied to \eqref{ss1} with the linear coefficient $-(\lambda_1+\vartheta)$.
This yields the threshold functional $\tau^{**}$ above.

To estimate $\mathbb{P}(\tau^{**}=\infty)$, note that
\[
\mathbb{P}(\tau^{**}=\infty)
=\mathbb{P}\!\left(
\int_0^\infty e^{\rho W(s)-(\lambda_1+\vartheta)(q-1)s}\,ds
<
\frac{2J(0)^{\,1-q}}
{\delta(q-1)\left(\int_D \varphi_1^{\frac{q}{q-p}}(x)\,dx\right)^{\frac{p-q}{p}}}
\right).
\]
Using Brownian scaling,
$\,2W(\tfrac{\rho^2}{4}s)\overset{law}{=}\rho W(s)$, and the change of variables
$s=\tfrac{4t}{\rho^2}$, we obtain
\[
\int_0^\infty e^{\rho W(s)-(\lambda_1+\vartheta)(q-1)s}\,ds
\ \overset{law}{=}\
\frac{4}{\rho^2}\int_0^\infty e^{2(W(t)-\theta_1 t)}\,dt,
\quad
\theta_1=\frac{2(\lambda_1+\vartheta)(q-1)}{\rho^2}.
\]
Lemma~\ref{ll1} then gives
\[
\int_0^\infty e^{2(W(t)-\theta_1 t)}\,dt
\ \overset{law}{=}\ \frac{1}{2X(\theta_1,1)}.
\]
Substituting this identity yields
\[
\mathbb{P}(\tau^{**}=\infty)
=\mathbb{P}\!\left(
\delta(q-1)\left(\int_D \varphi_1^{\frac{q}{q-p}}(x)\,dx\right)^{\frac{p-q}{p}}
<
X(\theta_1,1)\,\rho^2\,J(0)^{\,1-q}
\right),
\]
and therefore
$\mathbb{P}( \tau_{b}<\infty)\ge \mathbb{P}(\tau^{**}<\infty)
=1-\mathbb{P}(\tau^{**}=\infty)$ leading to the desired estimate \eqref{ik37}.
\end{proof}

The next theorem provides an explicit upper bound on the blow--up probability of
solutions to \eqref{ss1}.

\begin{theorem}\label{thm6.2}
Assume $p,q>1$ with $p>q$ and let $f\in L^\infty(D)$ satisfy $f\ge0$. Then
\[
\mathbb{P}\{\tau_{b}<\infty\}
\le \int_{\widetilde N}^{\infty} h(y)\,dy,
\]
where
\[
\rho:=\sigma(q-1),\quad
\widetilde N:=\frac{1}{\delta|D|(q-1)\|f\|_\infty^{\,q-1}},
\]
and
\[
h(y)
=
\frac{\left(\frac{2}{\rho^2 y}\right)^{\frac{2\vartheta(q-1)}{\rho^2}}}
{y\,\Gamma\!\left(\frac{2\vartheta(q-1)}{\rho^2}\right)}
\exp\!\left(-\frac{2}{\rho^2 y}\right).
\]
\end{theorem}

\begin{proof}
By the definition of the lower-bound stopping functional in the Brownian case,
\[
\mathbb{P}\{ \tau_{b}<\infty\}\le \mathbb{P}(\tau_{**}<\infty)
= \mathbb{P}\!\left(\int_0^\infty e^{\rho W(s)-\vartheta(q-1)s}\,ds\ge \widetilde N\right).
\]
Writing $W^{(\hat\alpha)}(s):=W(s)-\hat\alpha s$ with
$\hat\alpha=\vartheta(q-1)/\rho$, we have
\[
\int_0^\infty e^{\rho W(s)-\vartheta(q-1)s}\,ds
=\int_0^\infty e^{\rho W^{(\hat\alpha)}(s)}\,ds
\ \overset{law}{=}\
\frac{4}{\rho^2}\int_0^\infty e^{2W^{(\nu)}(t)}\,dt,
\quad
\nu:=\frac{2\hat\alpha}{\rho}=\frac{2\vartheta(q-1)}{\rho^2}.
\]
The exponential functional identity (see \cite[Ch.~6, Cor.~1.2]{yor2001})
yields
\[
\int_0^\infty e^{2W^{(\nu)}(t)}\,dt\ \overset{law}{=}\ \frac{1}{2Z_\nu},
\]
where $Z_\nu$ is Gamma distributed with density
$\frac{1}{\Gamma(\nu)}e^{-y}y^{\nu-1}\,dy$. Therefore,
\[
\mathbb{P}\{\tau_{b}<\infty\}
\le \mathbb{P}\!\left(\frac{2}{\rho^2 Z_\nu}\ge \widetilde N\right)
=\mathbb{P}\!\left(Z_\nu\le \frac{2}{\rho^2 \widetilde N}\right)
=\int_{\widetilde N}^{\infty} h(y)\,dy,
\]
with $h$ as stated.
\end{proof}


  \section{Numerical Results}\label{nss}
 In this section we present a numerical study of \eqref{b1} in one space
dimension ($d=1$). Throughout, we impose homogeneous Dirichlet conditions on
the complement of the interval $D=[-1,1]$.

\subsection*{Finite-difference approximation}
To approximate \eqref{b1} numerically we employ a finite-difference
discretisation in space combined with a semi-implicit Euler scheme in time,
cf.~\cite{Duo18,Lord}. Related strategies have also been used in~\cite{kava}.

We discretise $[0,T]\times[-1,1]$ by setting $t_n=n\Delta t$ with
$\Delta t=T/N$ ($n=0,1,\dots,N$), and $x_j=-1+j\Delta x$ with $\Delta x=2/M$
($j=0,1,\dots,M$). Let $U_h^n=(u_0^n,u_1^n,\dots,u_M^n)$ denote the grid
approximation of $u(t_n,\cdot)$, and write
\[
u_h^n := (u_1^n,\dots,u_{M-1}^n)^{\top}\in\mathbb{R}^{M-1},
\]
since the exterior Dirichlet condition yields $u_0^n=u_M^n=0$.

\medskip
\noindent\textbf{(1) Spatial discretisation of the fractional Laplacian.}
We approximate the fractional Laplacian
\[
\Delta_\alpha u = -(-\Delta)^{\alpha/2}u,\qquad 0<\alpha\le 2,
\]
by the finite-difference method of~\cite{Duo18}, which leads to a matrix
representation
\[
(\Delta_\alpha u)(t_n,\cdot)\ \approx\ A\,u_h^n,
\]
where $A\in\mathbb{R}^{(M-1)\times(M-1)}$ has entries
\[
A_{i,j}=C_{1,\alpha}\left\{
\begin{array}{ll}
\displaystyle
\sum_{k=2}^{M-1}\frac{(k+1)^\chi-(k-1)^\chi}{k^\rho}
+\frac{(M+1)^\chi-(M-1)^\chi}{M^{\rho}}
+\Big(2^\chi+\kappa_\rho-1\Big)+\frac{\chi}{\alpha M^{2\alpha}},
& i=j,\\[1.1em]
\displaystyle
-\frac{\big(|j-i|+1\big)^\chi-\big(|j-i|-1\big)^\chi}{2\,|j-i|^\rho},
& \Big(j\neq i,\\ &j\neq i\pm 1\Big),\\[1.1em]
\displaystyle
-\frac12\Big(2^\chi+\kappa_\rho-1\Big),
& j=i\pm 1,
\end{array}
\right.
\]
with
\[
\chi=\rho-2\alpha,\qquad
\kappa_\rho=
\begin{cases}
1+2\alpha, & 2\alpha\in(1,2),\\
1, & 2\alpha=1,
\end{cases}
\qquad
\rho\in(2\alpha,2].
\]
In our simulations we take $\rho=1+\alpha$. We refer to~\cite{Duo18} for the
derivation, admissible parameter ranges, and numerical considerations.

\medskip
\noindent\textbf{(2) Semi-implicit Euler time stepping.}
Let the drift term \[
F(u):=\delta\Big(\int_D u^q dx\Big)+\gamma u-\beta u^p.
\]at time $t_n$ be approximated by
\[
F(u_h^n)=\delta\,\mathcal{I}\!\big[(u_h^n)^q\big]+\gamma u_h^n-\beta (u_h^n)^p,
\]
where $\mathcal{I}[(u_h^n)^q]$ denotes the (replicated) numerical approximation
of the scalar integral $\int_{-1}^1 u(t_n,y)^q\,dy$, evaluated for instance via
Simpson’s rule and then placed into $\mathbb{R}^{M-1}$ as a constant vector
(multiplying the all-ones vector).

For the stochastic forcing we write
\[
B_h(t_n):=\sigma(u_h^n)\,\big(k\,b_s^H(t_n)\big)\in\mathbb{R}^{M-1},
\]
where $b_s^H(t_n)$ denotes the discrete fractional Gaussian noise increment and
$\sigma(\cdot)$ is the chosen multiplicative coefficient (e.g.\ $\sigma(u)=\sigma u$).

Using a semi-implicit treatment of the diffusion (implicit in $A$ and explicit
in the nonlinearities), the update reads
\[
\frac{u_h^{n+1}-u_h^n}{\Delta t}
= A\,u_h^{n+1}+F(u_h^n)+B_h(t_n),
\]
equivalently,
\[
(I-\Delta t\,A)\,u_h^{n+1}
= u_h^n+\Delta t\,F(u_h^n)+\Delta t\,B_h(t_n).
\]

\medskip
\noindent\textbf{(3) Sampling fractional Brownian motion.}
Fractional Brownian motion paths are generated via the circulant embedding
method; see, e.g.,~\cite[Chapter~6]{Lord}. In particular, for $H=\tfrac12$ the
increments reduce to those of a standard Brownian motion, and the discrete
noise can be represented by
\[
b_s^{1/2}(t_n)=\frac{\xi(t_{n+1})-\xi(t_n)}{\sqrt{\Delta t}},\qquad
b_s(t_n)\sim N(0,1)\ \text{i.i.d.},
\]
where $\xi$ is a standard Brownian motion.

 \subsection*{Simulations}

We begin by illustrating a representative sample path of the numerical
solution of~\eqref{b1}; see Figure~\ref{Fig_sim1}(a). In this experiment we take
\[
\delta=1,\qquad \gamma=1,\qquad \beta=1,\qquad \sigma=0.1,\qquad q=2,\qquad p=2,
\]
with fractional diffusion exponent $\alpha=1.2$ and Hurst index $H=0.6$.
The initial datum is chosen as
\[
u(0,x)=c\,(1-x^2)+V_1(x),\qquad c=0.1,
\]
where $V_1$ denotes the numerical approximation of the eigenfunction
corresponding to the principal Dirichlet eigenvalue of the fractional Laplacian
on $D=[-1,1]$ with homogeneous exterior condition. For this particular
realisation the onset of blow--up is clearly visible: the solution profile
steepens rapidly and its amplitude grows sharply over a short time window.

For comparison, Figure~\ref{Fig_sim1}(b) shows $\|u(\cdot,t)\|_\infty$ computed
from five independent realisations with the same parameter set. All five
trajectories exhibit a qualitatively similar explosive behaviour, indicating
that (for these parameters and this time horizon) blow--up is not a rare event
and that the maximal amplitude is a robust diagnostic of the instability.

\begin{figure}[!htb]\vspace{0cm}\hspace{0cm}
   \begin{minipage}{0.48\textwidth}
     \centering    
     \includegraphics[width=.8\linewidth]{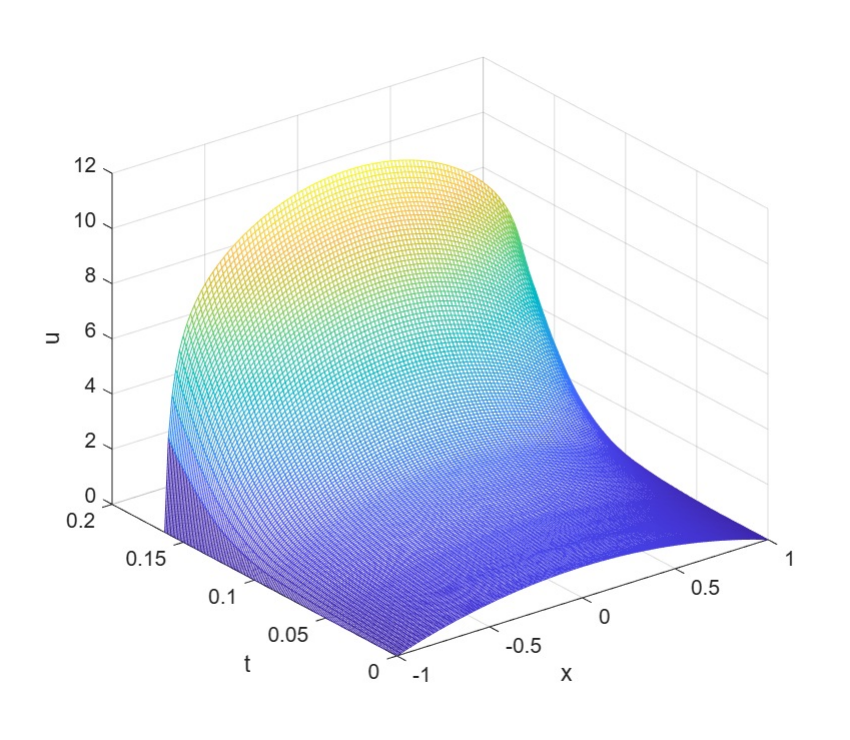}
   \end{minipage}\hfill
   \begin{minipage}{0.48\textwidth}\hspace{0cm}
     \centering
     \includegraphics[width=.8\linewidth]{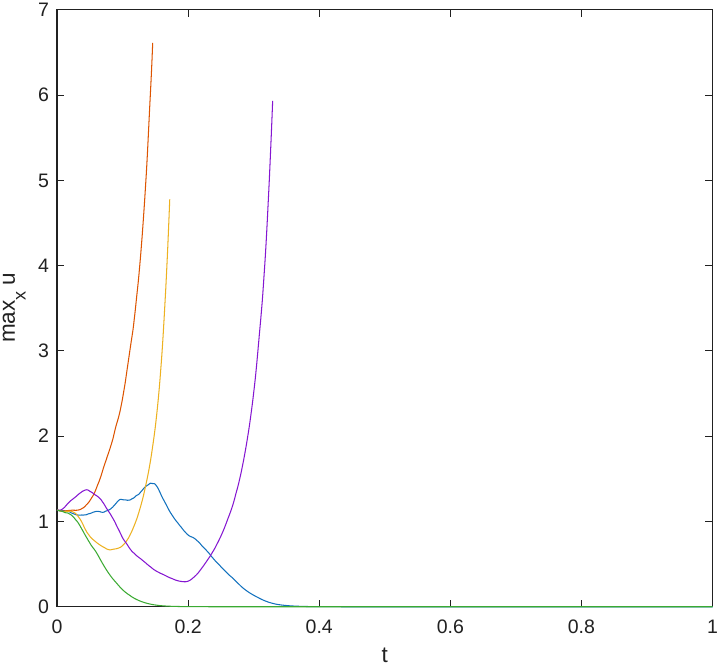}
   \end{minipage}\vspace{0cm}
   \caption{(a) A realisation of the numerical solution of~\eqref{b1} with
   $\alpha=1.2$, $H=0.6$, $\delta=1$, $\gamma=1$, $\beta=1$, $\sigma=0.1$,
   $M=101$, $N=10^4$, and initial condition $u(0,x)=c(1-x^2)+V_1(x)$ with
   $c=0.1$. (b) The evolution of $\|u(\cdot,t)\|_\infty$ for five independent
   realisations with the same parameters.}\label{Fig_sim1}
\end{figure}

\medskip
A natural quantitative question is how to estimate, for a prescribed horizon
$T>0$, the probability that blow--up occurs before $T$, i.e.\ $\mathbb{P}(\tau_b<T)$,
and how this probability depends on key parameters. In the deterministic
counterpart with classical diffusion (namely $\alpha=2$ and $\sigma\equiv 0$)
it is well known that sufficiently large $\delta$, sufficiently large $q$, or
sufficiently large initial data can enforce finite--time blow--up; see, e.g.,
\cite[Thmeorem ~43.1]{SoupletQuittner2007}. From an applications perspective, it is therefore
informative to report Monte--Carlo estimates of $\mathbb{P}(\tau_b<T)$ and of
basic statistics of the blow--up time.

In Tables~\ref{T1}--\ref{T5} we report results based on $N_R=10^4$ independent
realisations on the time interval $[0,1]$. The numerical blow--up time $\tau_b$
is recorded as
\[
\tau_b \approx t_m,
\qquad
t_m:=\max\Big\{t_n:\ \max_j u(t_n,x_j)\ge M_b\Big\},
\]
where $M_b$ is a fixed large threshold; in the computations below we take
$M_b=4.5036\times 10^{15}$. For each parameter configuration we estimate
\[
\widehat{\mathbb{P}}(\tau_b<1)
=\frac{\#\{\text{realisations with blow--up before }1\}}{N_R},
\]
together with the empirical mean $m(\tau_b)$ and empirical variance
$\Var(\tau_b)$ computed over the realisations that blow--up within $[0,1]$.

Moreover, for the following simulations, we choose  as indicative  values of the parameters to be
  $\alpha=1.2$, $H=0.6$, $\delta=1$, $\gamma=.1$, $\beta=1$, 
  $p=2$, $q=2$, $\sigma=0.1$ initial condition $u(0,x)=c(1-x^2)+V_1(x)$ with $c=0.01$
  and in each set of simulations we vary the parameter of interest, while for  the rest of them we keep the aforementioned values.


\begin{table}[!htb]
\centering
\caption{ Realisations of the numerical solution of
problem~\eqref{b1} for $N_R=10000$ on $[0,1]$. Variation of the exponent $q$
(recall that $p=2$, $\sigma=0.1$).}
\label{T1}
\bigskip
\begin{tabular}{||c||c|c|c||}\hline
$q$ & Blow--up Probability & $m(\tau_b)$ & $\Var(\tau_b)$\\ \hline
1.5 & 0      & --     & --     \\
2   & 0.4857 & 0.4434 & 0.0390 \\
2.5 & 0.5169 & 0.3979 & 0.0424 \\
3   & 0.5332 & 0.3766 & 0.0426 \\
3.5 & 0.5347 & 0.3632 & 0.0430 \\
4   & 0.5375 & 0.3539 & 0.0421 \\
4.5 & 0.5439 & 0.3490 & 0.0444 \\
5   & 0.5451 & 0.3449 & 0.0432 \\
10  & 0.5542 & 0.3241 & 0.0436 \\
\hline
\end{tabular}
\end{table}

\begin{table}[!htb]
\centering
\caption{ Realisations of the numerical solution of
problem~\eqref{b1} for $N_R=10000$ on $[0,1]$. Variation of the Hurst index $H$.
}
\label{T2}
\bigskip
\begin{tabular}{||c||c|c|c||}\hline
$H$ & Blow--up Probability & $m(\tau_b)$ & $\Var(\tau_b)$\\ \hline
0.5  & 0.4646 & 0.4948 & 0.0375 \\
0.6  & 0.4956 & 0.4441 & 0.0384 \\
0.65 & 0.5059 & 0.4194 & 0.0386 \\
0.7  & 0.5287 & 0.3938 & 0.0386 \\
0.75 & 0.5270 & 0.3698 & 0.0400 \\
0.8  & 0.5481 & 0.3469 & 0.0390 \\
0.85 & 0.5534 & 0.3259 & 0.0392 \\
0.9  & 0.5592 & 0.3078 & 0.0404 \\
0.95 & 0.5557 & 0.2841 & 0.0371 \\
\hline
\end{tabular}
\end{table}

\begin{table}[!htb]
\centering
\caption{ Realisations of the numerical solution of
problem~\eqref{b1} for $N_R=10000$ on $[0,1]$. Variation of the noise intensity
$\sigma$ (additionally here is chosen $\delta=7$  to ensure that the deterministic case blows up
within the simulated horizon).}
\label{T3}
\bigskip
\begin{tabular}{||c||c|c|c||}\hline
$\sigma$ & Blow--up Probability & $m(\tau_b)$ & $\Var(\tau_b)$\\ \hline
0      & 1      & 0.7800 & 0      \\
0.001  & 0.9346 & 0.7828 & 0.0065 \\
0.005  & 0.6364 & 0.6508 & 0.0190 \\
0.01   & 0.5819 & 0.5638 & 0.0240 \\
0.05 & 0.5839  &  0.3758  &  0.0331\\
0.075 & 0.5886  &  0.3411 &   0.0354\\
 0.1 &  0.6006  &  0.3190  &  0.0367 \\
0.5 & 0.6626   & 0.2126  &  0.0384 \\
1 &  0.6854  &  0.1819  &  0.0379 \\
\hline
\end{tabular}
\end{table}

\begin{table}[!htb]
\centering
\caption{ Realisations of the numerical solution of
problem~\eqref{b1} for $N_R=10000$ on $[0,1]$. Variation of the initial
condition coefficient $c$ in $u(0,x)=c(1-x^2)+V_1(x)$. }
\label{T4}
\bigskip
\begin{tabular}{||c||c|c|c||}\hline
$c$ & Blow--up Probability & $m(\tau_b)$ & $\Var(\tau_b)$\\ \hline
0.01 & 0.5797 & 0.3034 & 0.0338 \\
0.05 & 0.6115 & 0.7828 & 0.0363 \\
0.1  & 0.6255 & 0.2902 & 0.0345 \\
1    & 0.8828 & 0.1238 & 0.0101 \\
2    & 0.9976 & 0.0590 & 0.0003 \\
2.5  & 1      & 0.0478 & 0      \\
\hline
\end{tabular}
\end{table}

\begin{table}[!htb]
\centering
\caption{Realisations of the numerical solution of
problem~\eqref{s1} for $N_R=10000$ on $[0,1]$. 
Variation of the fractional Laplacian exponent $\alpha$ (reported as $\alpha/2$).}
\label{T5}
\bigskip
\begin{tabular}{||c||c|c|c||}\hline
$\alpha/2$ & Blow--up Probability & $m(\tau_b)$ & $\Var(\tau_b)$\\ \hline
0.95 & 0.4821 & 0.4590 & 0.0394 \\
0.9  & 0.4899 & 0.4496 & 0.0390 \\
0.8  & 0.4815 & 0.4512 & 0.0394 \\
0.7  & 0.4895 & 0.4402 & 0.0377 \\
0.6  & 0.4943 & 0.4417 & 0.0396 \\
0.5  & 0.5097 & 0.4460 & 0.0389 \\
\hline
\end{tabular}
\end{table}

\paragraph{Interpretation of Tables~\ref{T1}--\ref{T5}.}
\begin{enumerate}[label=(\roman*)]
\item \textbf{Effect of the reaction exponent $q$ (Table~\ref{T1}).}
Keeping $p=2$ fixed, increasing $q$ strengthens the nonlocal source term
$\delta\int_D u^q$, which promotes faster growth at high amplitudes. This is
consistent with Table~\ref{T1}: once $q$ is sufficiently large (here near
$q\approx 2$), the estimated blow--up probability within $[0,1]$ increases and
then plateaus around $0.53$--$0.55$. Meanwhile, the conditional mean blow--up time
decreases, indicating that blow--up (when it occurs) tends to occur earlier for
larger $q$.

\item \textbf{Effect of temporal correlation $H$ (Table~\ref{T2}).}
Table~\ref{T2} suggests an overall increase in the blow--up probability as the
Hurst index $H$ increases from $0.5$ toward $1$. Since larger $H$ corresponds to
more persistent increments of $B^H$, favourable growth phases can last longer,
which is compatible with both the increased probability and the decreasing mean
blow--up time.

\item \textbf{Effect of the noise intensity $\sigma$ (Table~\ref{T3}).}
We choose $\delta$ large enough so that the deterministic problem ($\sigma=0$)
blows up within the time horizon $[0,1]$. As $\sigma$ increases, the estimated
blow--up probability over $[0,1]$ decreases, indicating that stronger
multiplicative fluctuations generate more sample paths that are driven away from
the blow--up regime (or delay blow--up beyond $t=1$). We observe a sharp drop as
soon as $\sigma>0$, followed by a slight increase for larger $\sigma$; however,
the probability remains well below the deterministic value $1$. Moreover, among
those trajectories that do blow--up, the blow--up time tends to be smaller on
average, consistent with the fact that occasional large positive excursions of
the noise can trigger a more abrupt blow--up.

\item \textbf{Effect of initial amplitude (Table~\ref{T4}).}
Increasing the coefficient $c$ raises the overall initial mass and drives the
system deeper into the unstable regime. Table~\ref{T4} exhibits the expected
threshold-type behaviour: the blow--up probability increases rapidly toward $1$
and the mean blow--up time decreases markedly.

\item \textbf{Effect of the diffusion exponent $\alpha$ (Table~\ref{T5}).}
Over the tested range, the blow--up probability and the first two moments of
$\tau_b$ vary only mildly. For this particular domain, horizon, and parameter
regime, the reaction and noise effects appear to dominate moderate changes in
$\alpha$, so the influence of fractional diffusion on the reported blow--up
statistics is comparatively weak.
\end{enumerate}

Finally, we illustrate the almost sure exponential decay predicted by Theorem~\ref{cor:sto-exp-fBm-gb1} by numerical simulations on the time interval $[0,T]$ with $T=10$; see Figure~\ref{Fig_simTh46}. We choose $p=4>q=1$, $\delta=1$, a small noise intensity $\sigma=0.01$, and $\gamma=0.01$, while $\beta=2$ is taken sufficiently large relative to $\delta$ and $\gamma$ to enforce the strong dissipative regime. Moreover, we set $H=0.6$ and $\alpha=1.2$. For this value of $\alpha$, a numerical approximation of the principal Dirichlet eigenvalue of the fractional Laplacian yields $\lambda_1\simeq 1.3037$, and hence $\gamma<\lambda_1$ as required.
\begin{figure}[htb]
\vspace*{-1cm}
\begin{center}
\includegraphics[bb= 330 230 250 600, scale=.9]{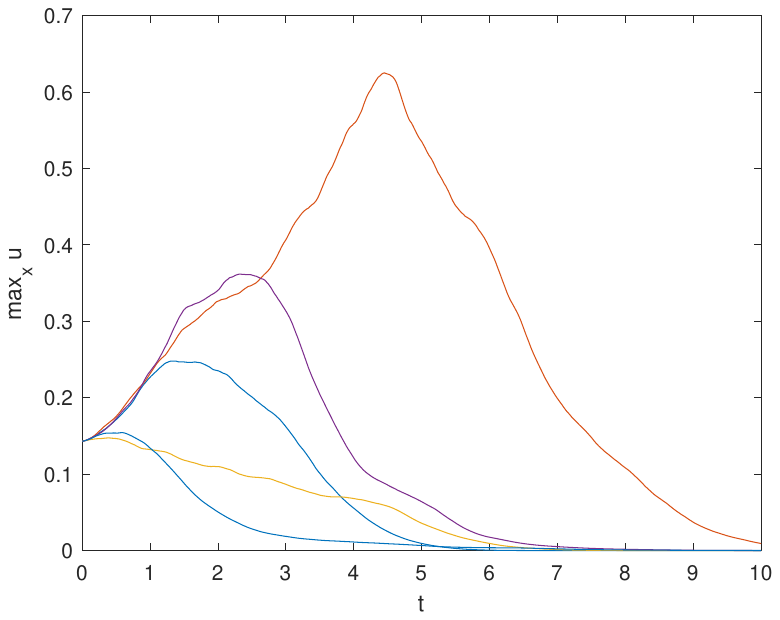}
 \vspace*{-2cm}
\end{center}
\caption{ The evolution of $\|u(\cdot,t)\|_\infty$ for five independent
   realisations with the same parameters chosen in a range given by Theorem \ref{cor:sto-exp-fBm-gb1}. In all of the cases the solution becomes  exponential small.}\label{Fig_simTh46}
\end{figure}
 We observe that the solution in the simulation range becomes exponentially small for all of the realizations presented. 


    \medskip\noindent
	{\bf Acknowledgements:}  The first author would like to thank Indian Institute of Technology Bhubaneswar, India, for providing an excellent research environment and the necessary facilities to carry out this work successfully. This work was carried out under the Innovation in Science Pursuit for Inspired Research (INSPIRE) Faculty Scheme research grant, supported by the Department of Science and Technology (DST), Project No. RP328, at the Indian
    Institute of Technology Bhubaneswar. M. T. Mohan would  like to thank the Department of Science and Technology (DST), India for Innovation in Science Pursuit for Inspired Research (INSPIRE) Faculty Award (IFA17-MA110). The third author is supported by the Fund for Improvement of Science and Technology Infrastructure (FIST) of DST (SR/FST/MSI-115/2016).
	
\section{Appendix}\label{app}
    \begin{proof}[Proof of Lemma~\ref{l1}]
We split the proof into two parts.

\medskip

\noindent\textbf{Step 1. Weak maximum principle: $w\ge 0$ in $\overline D_T$.}

We argue by contradiction. Assume that $w$ takes a negative value on
$\overline D_T$, i.e.
\[
m := \inf_{[0,T]\times\mathbb{R}^d} w < 0.
\]
By continuity, this infimum is attained at some point
$(t_0,x_0)\in [0,T]\times\mathbb{R}^d$.
The boundary and initial conditions imply that:
\begin{itemize}
	\item $w(0,x)\ge 0$ for all $x\in D$,
	\item $w(t,x)\ge 0$ for all $(t,x)\in (0,T]\times(\mathbb{R}^d\setminus D)$.
\end{itemize}
Thus the negative minimum cannot lie at $t=0$ or outside $D$, and we must have
$(t_0,x_0)\in D_T$.

\medskip

\noindent\emph{Exponential change of variable.}
Let
\[
M_1:=\|c_1\|_{L^\infty(D_T)}, \qquad
M_2:=\|c_2\|_{L^\infty((0,T)\times D)}, \qquad
|D|:=\text{Lebesgue measure of }D.
\]
Choose
\[
\Lambda > M_1 + M_2 |D|,
\]
and define
\[
v(t,x) := e^{-\Lambda t} w(t,x).
\]
Then $v$ is bounded and continuous on $[0,T]\times\mathbb{R}^d$, and attains a
negative global minimum
\[
m_v := \inf_{[0,T]\times\mathbb{R}^d} v < 0
\]
at some point $(t_1,x_1)\in D_T$ (since $w$ and $v$ have the same sign and the
same zero sets in time). Clearly,
\[
v(t,x)\ge 0 \quad\text{for } t=0 \text{ or } x\in\mathbb{R}^d\setminus D,
\]
so the minimum is indeed attained in $D_T$.

We claim that $v$ is a viscosity supersolution in $D_T$ of
\begin{equation}\label{eq:v-visc}
v_t(t,x) - d\,\Delta_{\alpha} v(t,x)
\;\ge\;
(c_1(t,x)-\Lambda)\,v(t,x)
+ \int_D c_2(t,y)\,v(t,y)\,dy.
\end{equation}
Indeed, let $\varphi\in C^{1,2}$ and suppose $v-\varphi$ has a local minimum at
$(t,x)\in D_T$ with $(v-\varphi)(t,x)=0$. Define the test function
\[
\psi(\tau,z) := e^{\Lambda \tau}\varphi(\tau,z),
\]
and observe that $w-\psi$ also has a local minimum at $(t,x)$:
\[
(w-\psi)(\tau,z) = e^{\Lambda \tau}v(\tau,z) - e^{\Lambda\tau}\varphi(\tau,z)
= e^{\Lambda\tau}(v-\varphi)(\tau,z).
\]
Since $w$ is a viscosity supersolution of \eqref{eq:visc-fractional}, we have
\[
\psi_t(t,x) - d\Delta_{\alpha} \psi(t,x)
\;\ge\;
c_1(t,x) w(t,x)
+ \int_D c_2(t,y) w(t,y)\,dy.
\]
But $w=e^{\Lambda t}v$ and
\[
\psi_t(t,x)
= \Lambda e^{\Lambda t}\varphi(t,x) + e^{\Lambda t}\varphi_t(t,x)
= e^{\Lambda t}\big(\Lambda \varphi(t,x) + \varphi_t(t,x)\big),
\]
\[
\Delta_{\alpha}\psi(t,x) = e^{\Lambda t}\Delta_{\alpha}\varphi(t,x)
\]
(using linearity and translation invariance of $\Delta_{\alpha}$). Dividing by $e^{\Lambda t}>0$, we get
\[
\varphi_t(t,x) + \Lambda \varphi(t,x) - d\Delta_{\alpha}\varphi(t,x)
\;\ge\;
c_1(t,x) v(t,x)
+ \int_D c_2(t,y) v(t,y)\,dy.
\]
Since $\varphi(t,x)=v(t,x)$ at the touching point, this is exactly
\eqref{eq:v-visc} in viscosity sense. Thus $v$ is a viscosity supersolution of
\eqref{eq:v-visc}.

\medskip

\noindent\emph{Contradiction at the global minimum.}
Let $(t_1,x_1)\in D_T$ be such that $v(t_1,x_1)=m_v=\min v <0$.
By continuity, for any $\varepsilon>0$ we can define a quadratic test function
\[
\varphi_\varepsilon(t,x)
:= m_v + \varepsilon\big(|x-x_1|^2 + (t-t_1)^2\big)
\]
so that $v-\varphi_\varepsilon$ attains a local minimum $0$ at $(t_1,x_1)$.
(Indeed $v(t_1,x_1)=m_v$ and $v\ge m_v$ everywhere, while
$\varphi_\varepsilon> m_v$ away from $(t_1,x_1)$ for small enough $\varepsilon$.)

Since $\varphi_\varepsilon$ is $C^{1,2}$, we may use it as test function in the
viscosity supersolution inequality \eqref{eq:v-visc} at $(t_1,x_1)$:
\[
(\varphi_\varepsilon)_t(t_1,x_1) - d\Delta_{\alpha} \varphi_\varepsilon(t_1,x_1)
\;\ge\;
(c_1(t_1,x_1)-\Lambda)\,v(t_1,x_1)
+ \int_D c_2(t_1,y)\,v(t_1,y)\,dy.
\]

Now,
\[
(\varphi_\varepsilon)_t(t_1,x_1) = 0,
\]
and $\varphi_\varepsilon(t_1,\cdot)$ has a \emph{strict global minimum} at
$x_1$ (since $|x-x_1|^2>0$ for $x\neq x_1$). By the very definition of
the fractional Laplacian,
\[
\Delta_{\alpha} \varphi_\varepsilon(t_1,x_1)
= c_{d,s}\,\mathrm{P.V.}\!\int_{\mathbb{R}^d}
\frac{\varphi_\varepsilon(t_1,x_1) - \varphi_\varepsilon(t_1,y)}{|x_1-y|^{d+\alpha}}
\,dy,
\]
and the integrand is strictly negative on a set of positive measure
($\varphi_\varepsilon(t_1,y)>\varphi_\varepsilon(t_1,x_1)$ for $y\neq x_1$).
Hence
\[
\Delta_{\alpha} \varphi_\varepsilon(t_1,x_1) > 0,
\]
and therefore
\[
(\varphi_\varepsilon)_t(t_1,x_1) - d\Delta_{\alpha}\varphi_\varepsilon(t_1,x_1)
= -d\Delta_{\alpha}\varphi_\varepsilon(t_1,x_1) < 0.
\]

On the other hand, since $(t_1,x_1)$ is a global minimum point of $v$,
we have $v(t_1,y)\ge m_v$ for all $y\in D$, and thus
\[
\int_D c_2(t_1,y)\,v(t_1,y)\,dy
\;\ge\;
m_v \int_D c_2(t_1,y)\,dy.
\]
Therefore the right-hand side of \eqref{eq:v-visc} at $(t_1,x_1)$ satisfies
\[
\begin{aligned}
(c_1(t_1,x_1)-\Lambda)\,v(t_1,x_1)
+ \int_D c_2(t_1,y)\,v(t_1,y)\,dy
&\ge
m_v\Big(c_1(t_1,x_1)-\Lambda + \int_D c_2(t_1,y)\,dy\Big).
\end{aligned}
\]

By the choice of $\Lambda$ and the bounds on $c_1,c_2$,
\[
c_1(t_1,x_1)-\Lambda + \int_D c_2(t_1,y)\,dy
\le
M_1 - \Lambda + M_2 |D|
< 0,
\]
so the factor in parentheses is strictly negative, and since $m_v<0$ we get
\[
m_v\Big(c_1(t_1,x_1)-\Lambda + \int_D c_2(t_1,y)\,dy\Big) > 0.
\]
Hence the right-hand side of \eqref{eq:v-visc} at $(t_1,x_1)$ is strictly positive:
\[
(c_1(t_1,x_1)-\Lambda)\,v(t_1,x_1)
+ \int_D c_2(t_1,y)\,v(t_1,y)\,dy > 0.
\]

This contradicts the viscosity inequality, which says that the left-hand side,
equal to $-d\Delta_{\alpha}\varphi_\varepsilon(t_1,x_1)$, must be greater or equal
to the right-hand side. Thus our assumption that $m_v<0$ (equivalently $m<0$)
is false, and we conclude that
\[
w(t,x)\ge 0 \quad\text{for all } (t,x)\in\overline D_T.
\]

\medskip

\noindent\textbf{Step 2. Strong maximum principle: strict positivity in $D_T$ if $w(0,\cdot)\not\equiv 0$.}

Assume now that $w(0,\cdot)\not\equiv 0$ in $D$ and $w\ge 0$ on $\overline D_T$.
Suppose, by contradiction, that there exists $(t_*,x_*)\in D_T$ such that
\[
w(t_*,x_*)=0.
\]
Then $w$ attains its global minimum $0$ in the open set $D_T$.

By Step~1, $w$ is a bounded, nonnegative viscosity supersolution of
\eqref{eq:visc-fractional}, and hence of
\[
w_t - d\Delta_{\alpha} w - c_1(t,x)w
\;\ge\;
\int_D c_2(t,y)\,w(t,y)\,dy \ge 0
\quad\text{in }D_T.
\]
That is, $w$ is a viscosity supersolution of a linear, nonlocal parabolic equation
with a (possibly sign-changing) zeroth-order term and a nonnegative right-hand side.

For nonlocal parabolic equations with Lévy operators such as $(-\Delta)^s$ and
bounded coefficients, a \emph{strong maximum principle} for viscosity solutions
is known: if a bounded viscosity supersolution $u$ of
\[
u_t - \mathcal{L}[u] - a(t,x)u \ge 0
\]
(in an open set) attains a nonnegative minimum at an interior point and satisfies
appropriate nonlocal boundary conditions, then either $u$ is strictly positive
in the domain, or $u$ is constant. For linear operators of fractional Laplacian type,
this is proved, for instance, in
\cite{BarriosMedina2020,BiswasJakobsen2020,Coville2010,JarohsWeth2016}
(see also the general frameworks in
\cite{AlvarezTourin1996,BarlesImbert2008,CaffarelliSilvestre2009}).

Applying such a strong maximum principle to $w$ (or to $v=e^{-\Lambda t}w$),
we conclude that either:
\begin{itemize}
	\item $w(t,x)>0$ for all $(t,x)\in D_T$, or
	\item $w(t,x)\equiv 0$ in $D_T$.
\end{itemize}

The second alternative is not compatible with the assumption that
$w(0,\cdot)\not\equiv 0$ in $D$ and the continuity in time. Hence only the first
alternative can hold, and we obtain
\[
w(t,x)>0 \quad\text{for all } (t,x)\in D_T.
\]
This completes the proof.
\end{proof}


\begin{thebibliography}{99}


\bibitem{alinei_poiana2023_frac_oncology}
T.~Alinei-Poiana, E.~H.~Dulf and L.~Kovacs,
Fractional calculus in mathematical oncology,
\emph{Sci. Rep.} \textbf{13} (2023), 10083.

\bibitem{Allen2017}
L.~J.~S.~Allen,
\emph{Stochastic Population and Epidemic Models},
Springer, Cham, 2017.

\bibitem{AlvarezTourin1996}
O.~Alvarez and A.~Tourin,
Viscosity solutions of nonlinear integro-differential equations,
\emph{Ann. Inst. H. Poincar\'e Anal. Non Lin\'eaire} \textbf{13} (1996), no.~3, 293--317.

\bibitem{AndersonChaplain1998}
A.~R.~A.~Anderson and M.~A.~J.~Chaplain,
Continuous and discrete mathematical models of tumor-induced angiogenesis,
\emph{Bull. Math. Biol.} \textbf{60} (1998), no.~5, 857--899.

\bibitem{law}
M.~A.~Arcones,
On the law of the iterated logarithm for Gaussian processes,
\emph{J. Theoret. Probab.} \textbf{8} (1995), 877--903.

\bibitem{Arisawa2006}
M.~Arisawa,
A new definition of viscosity solutions for a class of second-order degenerate elliptic integro-differential equations,
\emph{Ann. Inst. H. Poincar\'e Anal. Non Lin\'eaire} \textbf{23} (2006), no.~5, 695--711.


\bibitem{bessonov2022_nonlocal_biomed}
M.~Banerjee, M.~Kuznetsov, O.~Udovenko and V.~Volpert,
Nonlocal reaction-diffusion equations in biomedical applications,
\emph{Acta Biotheoretica} \textbf{70} (2022), 1--35.

\bibitem{BarlesImbert2008}
G.~Barles and C.~Imbert,
Second-order elliptic integro-differential equations: viscosity solutions' theory revisited,
\emph{Ann. Inst. H. Poincar\'e Anal. Non Lin\'eaire} \textbf{25} (2008), no.~3, 567--585.

\bibitem{BarriosMedina2020}
B.~Barrios and M.~Medina,
Strong maximum principles for fractional elliptic and parabolic problems with mixed boundary conditions,
\emph{Proc. R. Soc. Edinb. A} \textbf{150} (2020), no.~2, 475--495.

\bibitem{BiswasJakobsen2020}
A.~Biswas and E.~R.~Jakobsen,
Hopf's lemma for viscosity solutions to a class of non-local equations with applications,
\emph{Nonlinear Anal.} \textbf{193} (2020), 111472.

\bibitem{BogdanGrzywnyRyznar2010}
K.~Bogdan, T.~Grzywny, and M.~Ryznar,
Heat kernel estimates for the fractional Laplacian with Dirichlet conditions,
\emph{Ann. Probab.} \textbf{38} (2010), no.~5, 1901--1923.

\bibitem{BucurValdinoci2016}
C.~Bucur and E.~Valdinoci,
\emph{Nonlocal Diffusion and Applications},
Lecture Notes of the Unione Matematica Italiana~20, Springer, 2016.

\bibitem{CaffarelliSilvestre2007}
L.~Caffarelli and L.~Silvestre,
An extension problem related to the fractional Laplacian,
\emph{Comm. Partial Differential Equations} \textbf{32} (2007), 1245--1260.

\bibitem{CaffarelliSilvestre2009}
L.~Caffarelli and L.~Silvestre,
Regularity theory for fully nonlinear integro-differential equations,
\emph{Comm. Pure Appl. Math.} \textbf{62} (2009), no.~5, 597--638.

\bibitem{chang2012}
T.~Chang and K.~Lee,
On a stochastic partial differential equation with a fractional Laplacian operator,
\emph{Stochastic Process. Appl.} \textbf{122} (2012), no.~9, 3288--3311.

\bibitem{chen2019_nonlocal_migration}
L.~Chen, K.~Painter, C.~Surulescu and A.~Zhigun,
Mathematical models for cell migration: a nonlocal perspective,
\emph{Philos. Trans. R. Soc. B} \textbf{375} (2020), 20190379.

\bibitem{ChenKimSong2010}
Z.-Q.~Chen, P.~Kim, and R.~Song,
Heat kernel estimates for the Dirichlet fractional Laplacian,
\emph{J. Eur. Math. Soc. (JEMS)} \textbf{12} (2010), no.~5, 1307--1329.

\bibitem{ChenSong2005}
Z.-Q.~Chen and R.~Song,
Two-sided eigenvalue estimates for subordinate processes in domains,
\emph{J. Funct. Anal.} \textbf{226} (2005), no.~1, 90--113.

\bibitem{cher2001}
P.~Cheridito,
Mixed fractional Brownian motion,
\emph{Bernoulli} \textbf{7} (2001), 913--934.

\bibitem{chow09}
P.~L.~Chow,
Unbounded positive solutions of nonlinear parabolic It\^o equations,
\emph{Commun. Stoch. Anal.} \textbf{3} (2009), no.~2, 211--222.

\bibitem{chow11}
P.~L.~Chow,
Explosive solutions of stochastic reaction-diffusion equations in mean $L^p$-norm,
\emph{J. Differential Equations} \textbf{250} (2011), no.~5, 2567--2580.

\bibitem{Coville2010}
J.~Coville,
Maximum principles, sliding techniques and applications to nonlocal equations,
\emph{Electron. J. Differential Equations} \textbf{2007} (2007), no.~68, 1--23.


\bibitem{DaPratoZabczyk1992}
G.~Da Prato and J.~Zabczyk,
\emph{Stochastic Equations in Infinite Dimensions},
Encyclopedia of Mathematics and its Applications~44, Cambridge Univ. Press, 1992.


\bibitem{DiNezzaPalatucciValdinoci2012}
E.~Di Nezza, G.~Palatucci and E.~Valdinoci,
Hitchhiker's guide to the fractional Sobolev spaces,
\emph{Bull. Sci. Math.} \textbf{136} (2012), 521--573.


\bibitem{doz2010}
M.~Dozzi and J.~A.~L\'opez-Mimbela,
Finite-time blowup and existence of global positive solutions of a semi-linear SPDE,
\emph{Stochastic Process. Appl.} \textbf{120} (2010), 767--776.

\bibitem{doz2013}
M.~Dozzi, E.~T.~Kolkovska and J.~A.~L\'opez-Mimbela,
Exponential functionals of Brownian motion and explosion times of a system of semilinear SPDEs,
\emph{Stochastic Anal. Appl.} \textbf{31} (2013), no.~6, 975--991.

\bibitem{dozfrac2013}M.~Dozzi, E.~T.~Kolkovska, and J.~A.~López-Mimbela. Finite-time blowup and existence of global positive solutions of a semi-linear stochastic partial differential equation with fractional noise, \emph{Modern stochastics and applications}. Cham: Springer International Publishing, 2013, 95--108.

\bibitem{doz2020}
M.~Dozzi, E.~T.~Kolkovska and J.~A.~L\'opez-Mimbela,
Global and non-global solutions of a fractional reaction-diffusion equation perturbed by a fractional noise,
\emph{Stochastic Anal. Appl.} \textbf{38} (2020), no.~6, 959--978.




\bibitem{Duo18}
S.~Duo, H.~Wang and Y.~Zhang,
A comparative study on nonlocal diffusion operators related to the fractional Laplacian,
\emph{Discrete Contin. Dyn. Syst. Ser. B} \textbf{24} (2019), no.~1, 231--256.

\bibitem{dung2}
N.~T.~Dung,
Tail estimates for exponential functionals and applications to SDEs,
\emph{Stochastic Process. Appl.} \textbf{128} (2018), 4154--4170.

\bibitem{dung}
N.~T.~Dung,
The probability of finite time blow-up of a semilinear SPDEs with fractional noise,
\emph{Statist. Probab. Lett.} \textbf{149} (2019), 86--92.

\bibitem{DydaKuznetsovKwasnicki2017}
B.~Dyda, S.~Kuznetsov, and M.~Kwa\'snicki,
Eigenvalues of the fractional Laplace operator in the unit ball,
\emph{J. London Math. Soc.} \textbf{95} (2017), 500--518.

\bibitem{fahimi2020_chaos_stoch_cancer}
M.~Fahimi, K.~Nouri and L.~Torkzadeh,
Chaos in a stochastic cancer model,
\emph{Physica A} \textbf{545} (2020), 123810.
\bibitem{FarwigIwabuchi2024}
R.~Farwig and T.~Iwabuchi,
\newblock Sobolev spaces on arbitrary domains and semigroups generated by the fractional Laplacian,
\newblock \emph{Bull. Sci. Math.} \textbf{193} (2024), 103440.
\newblock doi:10.1016/j.bulsci.2024.103440.



\bibitem{foon2016}
M.~Foondun, J.~B.~Mijena and E.~Nane,
Non-linear noise excitation for some space-time fractional stochastic equations in bounded domains,
\emph{Fract. Calc. Appl. Anal.} \textbf{19} (2016), 1527--1553.

\bibitem{FoondunOsgoodSPDE}
M.~Foondun and E.~Nualart,
The Osgood condition for stochastic partial differential equations,
\emph{Bernoulli} \textbf{27} (2021), no.~1, 295--311.





\bibitem{fritz2025_stoch_cahn_hilliard}
M.~Fritz and  L.~Scarpa,
Analysis and computations of a stochastic Cahn--Hilliard model for tumor growth with chemotaxis and variable mobility, \emph{Stochastics and Partial Differential Equations: Analysis and Computations} \textbf{13} (2025), 1051--1096.

\bibitem{fuji1966}
H.~Fujita,
On the blowing up of solutions of the Cauchy problem for $u_t=\Delta u+u^{1+\alpha}$,
\emph{J. Fac. Sci. Univ. Tokyo Sect. I} \textbf{13} (1966), 109--124.

\bibitem{Fuji1968}
H.~Fujita,
On some nonexistence and nonuniqueness theorems for nonlinear parabolic equations,
\emph{Proc. Sympos. Pure Math.} \textbf{18} (1968), 105--113.

\bibitem{nonlocal_tumor_growth2025}
G.~Granero-Belinchon and M.~Magliocca,
A nonlocal equation describing tumor growth, \emph{Mathematical Models and Methods in Applied Sciences} \textbf{35} (2025), no.~03, 585--609.



\bibitem{Eug2017}
E.~Guerrero and J.~A.~L\'opez-Mimbela,
Perpetual integral functionals of Brownian motion and blowup of semilinear systems of SPDEs,
\emph{Commun. Stoch. Anal.} \textbf{11} (2017), no.~3, 335--356.

\bibitem{gundogdu2025_time_frac_cancer}
H.~Gundogdu and H.~Joshi,
Numerical analysis of time-fractional cancer models with different types of net killing rate,
\emph{Mathematics} \textbf{13} (2025), 536.

\bibitem{IGCR}
I.~Gy\"ongy and C.~Rovira,
On $L^p$-solutions of semilinear stochastic partial differential equations,
\emph{Stochastic Process. Appl.} \textbf{90} (2000), 83--108.

\bibitem{Hening2018}
A.~Hening, D.~H.~Nguyen and G.~Yin,
Stochastic population growth in spatially heterogeneous environments: the density-dependent case,
\emph{J. Math. Biol.} \textbf{76} (2018), 697--754.


\bibitem{H18}
S.~Huda \emph{et al.},
L\'evy-like movement patterns of metastatic cancer cells revealed in microfabricated systems and implicated in vivo,
\emph{Nat. Commun.} \textbf{9} (2018), 4539.

\bibitem{JarohsWeth2016}
S.~Jarohs and T.~Weth,
On the strong maximum principle for nonlocal operators,
\emph{Math. Z.} \textbf{293} (2019), no.~1, 81--111.

\bibitem{kaplan}
S.~Kaplan,
On the growth of solutions of quasilinear parabolic equations,
\emph{Comm. Pure Appl. Math.} \textbf{16} (1963), no.~3, 305--333.

\bibitem{Karatzas}
I.~Karatzas and S.~E.~Shreve,
\emph{Brownian Motion and Stochastic Calculus},
Graduate Texts in Mathematics, Springer, 1991.

\bibitem{kava}
N.~I.~Kavallaris, C.~V.~Nikolopoulos and A.~N.~Yannacopoulos,
On the impact of noise on quenching for a nonlocal diffusion model driven by a mixture of Brownian and fractional Brownian motions,
\emph{Discrete Contin. Dyn. Syst. Ser. S} \textbf{17} (2024), no.~3, 1222--1268.

\bibitem{K15}
N.~I.~Kavallaris,
Explosive solutions of a stochastic non-local reaction-diffusion equation arising in shear band formation,
\emph{Math. Methods Appl. Sci.} \textbf{38} (2015), no.~16, 3564--3574.

\bibitem{KS18}
N.~I.~Kavallaris and T.~Suzuki,
\emph{Non-local Partial Differential Equations for Engineering and Biology},
Math. Ind. (Tokyo), vol.~31, Springer, Cham, 2018.

\bibitem{KY20}
N.~I.~Kavallaris and Y.~Yan,
Finite-time blow-up of a non-local stochastic parabolic problem,
\emph{Stochastic Process. Appl.} \textbf{130} (2020), no.~9, 5605--5635.

\bibitem{kazmierczak2025_tissue_growth}
B.~Kazmierczak and V.~Volpert,
Mathematical modelling of tissue growth control by positive and negative feedbacks,
\emph{PLoS One} \textbf{20} (2025), e0319120.


\bibitem{car2013}
E.~T.~Kolkovska and J.~A.~L\'opez-Mimbela,
Sub- and super-solution of a nonlinear PDE and application to semilinear SPDE,
\emph{Pliska Stud. Math. Bulgar.} \textbf{22} (2013), 109--116.

\bibitem{Kwasnicki2017}
M.~Kwa\'snicki,
Ten equivalent definitions of the fractional Laplace operator,
\emph{Fract. Calc. Appl. Anal.} \textbf{20} (2017), 7--51.

\bibitem{kwossek2025_sde_fbm}
A.~P.~Kwossek, A.~Neuenkirch and D.~J.~Pr\"omel,
Stochastic differential equations driven by fractional Brownian motion: dependence on the Hurst parameter, preprint (2025), arXiv:2504.04860.

\bibitem{Li2015fBm}
K.~Li,
Stochastic delay fractional evolution equations driven by fractional Brownian motion,
\emph{Math. Methods Appl. Sci.} \textbf{38} (2015), no.~8, 1582--1591.

\bibitem{li}
M.~Li and H.~Gao,
Estimates of blow-up times of a system of semilinear SPDEs,
\emph{Math. Methods Appl. Sci.} \textbf{40} (2017), no.~11, 4149--4159.

\bibitem{mathematical_oncology_review2025}
C.~Li and J.~Lei,
Mathematical modeling of tumor--immune interactions: methods, applications, and future perspectives, preprint (2025), arXiv:2511.00507.

\bibitem{liang}
F.~Liang and S.~Zhao,
Global existence and finite time blow-up for a stochastic non-local reaction-diffusion equation,
\emph{J. Geom. Phys.} \textbf{178} (2022), 104577.


\bibitem{Lord}
G.~J.~Lord, C.~E.~Powell and T.~Shardlow,
\emph{An Introduction to Computational Stochastic PDEs},
Cambridge University Press, 2014.

\bibitem{liu2017}
J.~Liu and C.~A.~Tudor,
Stochastic heat equation with fractional Laplacian and fractional noise: existence of the solution and analysis of its density,
\emph{Acta Math. Sci.} \textbf{37} (2017), 1545--1566.




\bibitem{MaslowskiNualart2003}
B.~Maslowski and D.~Nualart,
Evolution equations driven by a fractional Brownian motion,
\emph{J. Funct. Anal.} \textbf{202} (2003), 277--305.

\bibitem{yor2005}
H.~Matsumoto and M.~Yor,
Exponential functionals of Brownian motion I: probability laws at fixed time,
\emph{Probab. Surv.} \textbf{2} (2005), 312--347.

\bibitem{meerschaert_scheffler2017_ctrw}
M.~M.~Meerschaert and H.~P.~Scheffler,
Continuous time random walks and space--time fractional differential equations,
in \emph{Handbook of Fractional Calculus with Applications},
De Gruyter, Berlin, 2019.

\bibitem{Mishura2008}
Y.~S.~Mishura,
\emph{Stochastic Calculus for Fractional Brownian Motion and Related Processes},
Lecture Notes in Mathematics~1929, Springer, 2008.

\bibitem{Nualart2006}
D.~Nualart,
\emph{The Malliavin Calculus and Related Topics},
2nd ed., Springer, 2006.

\bibitem{NualartRascanu2002}
D.~Nualart and A.~R\u{a}\c{s}canu,
Differential equations driven by fractional Brownian motion,
\emph{Collect. Math.} \textbf{53} (2002), 55--81.

\bibitem{NV06}
D.~Nualart and P.~Vuillermot,
Variational solutions for partial differential equations driven by a fractional noise,
\emph{J. Funct. Anal.} \textbf{232} (2006), 390--474.


\bibitem{Pazy1983}
A.~Pazy,
\emph{Semigroups of Linear Operators and Applications to Partial Differential Equations},
Applied Mathematical Sciences~44, Springer, 1983.


\bibitem{PrevotRoeckner2007}
C.~Pr\'ev\^ot and M.~R\"ockner,
\emph{A Concise Course on Stochastic Partial Differential Equations},
Lecture Notes in Mathematics~1905, Springer, 2007.

\bibitem{SoupletQuittner2007}
P.~Quittner and Ph.~Souplet,
\emph{Superlinear Parabolic Problems: Blow-Up, Global Existence and Steady States},
Birkh\"auser Advanced Texts, Birkh\"auser Verlag, Basel, 2007.

\bibitem{Ralchenko2018}
K.~Ralchenko and G.~Shevchenko,
Existence and uniqueness of mild solution to fractional stochastic heat equation,
\emph{Modern Stochastics: Theory and Applications} \textbf{6} (2019), 57--79.

\bibitem{revuz1999}
D.~Revuz and M.~Yor,
\emph{Continuous Martingales and Brownian Motion},
Springer, Berlin, 1999.


\bibitem{smk1}
S.~Sankar, M.~T.~Mohan and S.~Karthikeyan,
Lower and upper bounds for the explosion times of a system of semilinear SPDEs,
\emph{Stochastics} \textbf{96} (2024), 846--886.

\bibitem{smk2}
S.~Sankar, M.~T.~Mohan and S.~Karthikeyan,
Blow-up estimates for a system of semilinear SPDEs driven by mixed fractional Brownian motions,
preprint/communicated (2024), arXiv:2207.12343.

\bibitem{smk3}
S.~Sankar, M.~T.~Mohan and S.~Karthikeyan,
Global existence and non-existence of weak solutions for non-local stochastic semilinear reaction-diffusion equations driven by a fractional noise,
preprint/communicated (2024), arXiv:2311.05926.


\bibitem{ServadeiValdinoci2014_Visc}
R.~Servadei and E.~Valdinoci,
Weak and viscosity solutions of the fractional Laplace equation,
\emph{Publ. Mat.} \textbf{58} (2014), no.~1, 133--154.

\bibitem{Stinga2018}
P.~R.~Stinga,
User's guide to the fractional Laplacian and the method of semigroups,
in \emph{Handbook of Fractional Calculus with Applications}, Vol.~2,
De Gruyter, 2018, pp.~235--265.



\bibitem{sweilam2025_frac_stoch_rd_cancer}
N.~Sweilam, S.~Al-Mekhlafi, W.~Abdel~Kareem and G.~Alqurishi,
Piecewise cancer tumor disease of partial differential equations based on exponential and non-singular kernel; numerical treatments,
\emph{Frontiers in Scientific Research and Technology} (accepted manuscript), 2025.

\bibitem{teschl}
G.~Teschl,
\newblock \emph{Ordinary Differential Equations and Dynamical Systems},
\newblock Graduate Studies in Mathematics, Vol.~140,
\newblock American Mathematical Society, Providence, RI, 2012.


\bibitem{Tesfay2021}
A.~Tesfay, D.~Tesfay, A.~Khalaf, and J.~Brannan,
Mean exit time and escape probability for the stochastic logistic growth model with multiplicative $\alpha$-stable L\'evy noise,
\emph{Stochastics Dynam.} \textbf{21} (2021), 2150016.


\bibitem{TindelTudorViens2003}
S.~Tindel, C.~A.~Tudor and F.~Viens,
Stochastic evolution equations with fractional Brownian motion,
\emph{Probab. Theory Related Fields} \textbf{127} (2003), 186--204.

\bibitem{tsai2012_fractal_tumour}
Y.-H.~Tsai \emph{et al.},
A fractal anomalous diffusion model with microenvironment for cancer invasion,
\emph{Theor. Biol. Med. Model.} \textbf{9} (2012), 27.

\bibitem{volpert2025_reaction_diffusion_waves}
V.~Volpert and S.~Petrovskii,
Reaction--diffusion waves in biology: new trends, recent developments,
\emph{Phys. Life Rev.} \textbf{52} (2025), 1--20.

\bibitem{wang}
X.~Wang,
Blow-up solutions of the stochastic nonlocal heat equations,
\emph{Stochastics Dynam.} \textbf{19} (2019), no.~2, 1950014.


\bibitem{west2022_fractal_tapestry}
B.~J.~West,
The fractal tapestry of life II: entailment of fractional oncology by physiology networks,
\emph{Front. Netw. Physiol.} \textbf{2} (2022), 845495.

\bibitem{yor2001}
M.~Yor,
\emph{Exponential Functionals of Brownian Motion and Related Processes},
Springer Finance, 2001.

\bibitem{young1936}
L.~C.~Young,
An inequality of the H\"older type connected with Stieltjes integration,
\emph{Acta Math.} \textbf{67} (1936), 251--282.

\bibitem{zaburdaev2015_levy_walks}
V.~Zaburdaev, S.~Denisov and J.~Klafter,
L\'evy walks,
\emph{Rev. Mod. Phys.} \textbf{87} (2015), 483--530.

\bibitem{Zahle1998}
M.~Z\"ahle,
Integration with respect to fractal functions and stochastic calculus,
\emph{Probab. Theory Related Fields} \textbf{111} (1998), 333--374.

\bibitem{Zahle}
M.~Z\"ahle,
Integration with respect to fractal functions and stochastic calculus II,
\emph{Math. Nachr.} \textbf{225} (2001), 145--183.

\end{thebibliography}
\end{document}